\documentclass[11pt,a4paper]{article}

\usepackage{amsmath}
\usepackage{amsfonts}
\usepackage{amsthm}
\usepackage{amssymb}

\newtheorem{thm}{Theorem}[section]
\newtheorem{lem}{Lemma}[section]
\newtheorem{hyp}{Hypothesis}[section]

\newtheorem{cor}[lem]{Corollary}
\newtheorem{prop}[thm]{Proposition}

\newtheorem{rmk}{Remark}[section]

\theoremstyle{remark}

\newcommand{\supp}{\operatorname{supp}}

\newcommand{\Z}{\mathbb{Z}}
\newcommand{\Rm}{\mathbb{R}}
\newcommand{\R}{\mathbb{R}}

\newcommand{\mC}{\ensuremath{\mathcal{C}}}

\newcommand{\mT}{\ensuremath{\mathcal{T}}}

\newcommand{\mD}{\ensuremath{\mathcal{D}}}

\newcommand{\Nm}{\ensuremath{\mathbb{N}}}
\newcommand{\Zm}{\ensuremath{\mathbb{Z}}}

\newcommand{\mA}{\ensuremath{\mathcal{A}}}

\newcommand{\mI}{\ensuremath{\mathcal{I}}}

\newcommand{\mW}{\ensuremath{\mathcal{W}}}
\newcommand{\mN}{\ensuremath{\mathcal{N}}}
\newcommand{\Tm}{\ensuremath{\mathbb{T}}}
\newcommand{\T}{\ensuremath{\mathbb{T}}}

\newcommand{\e}{\ensuremath{\epsilon}}

\def\lto{\longrightarrow}
\def\to{\longrightarrow}
\def\lmto{\longmapsto}

\def\leq{\leqslant}
\def\geq{\geqslant}
\def\le{\leqslant}
\def\ge{\geqslant}

\def\bdef{\begin{definition}}
\def\endef{\end{definition}}
\def\bthm{\begin{thm}}
\def\ethm{\end{thm}}
\def\blm{\begin{lemma}}
\def\elm{\end{lemma}}
\def\brm{\begin{rmk}}
\def\erm{\end{rmk}}
\def\bprop{\begin{proposition}}
\def\eprop{\end{proposition}}
\def\bcor{\begin{cor}}
\def\ecor{\end{cor}}
\def\be{\begin{eqnarray}}
\def\ee{\end{eqnarray}}
\def\beal{\begin{aligned}}
\def\enal{\end{aligned}}

\def\eps{\varepsilon}
\def\phi{\varphi}
\def\A{\mathcal A}
\def\R{\mathbb R}

\def\T{\mathbb T}
\def\Q{\mathbb Q}
\def\Z{\mathbb Z}

\def\cC{\mathcal C}

\def\Gm{\Gamma}
\def\Lb{\Lambda}
\def\th{\theta}

\def\lb{\lambda}

\def\be {\begin{equation}}
\def\ee {\end{equation}}
\def\bdef{\begin{definition}}
\def\endef{\end{definition}}
\def\blm{\begin{lem}}
\def\elm{\end{lem}}
\def\beal{\begin{aligned}}
\def\enal{\end{aligned}}

\newtheorem{definition}{Definition}
\author{P. Bernard\footnote{Universit{\'e} Paris-Dauphine (\texttt{patrick.bernard\@ ceremade.dauphine.fr})},
 V. Kaloshin\footnote{University of Maryland at College Park (\texttt{vadim.kaloshin\@ gmail.com})},
 K. Zhang\footnote{University of Toronto (\texttt{kzhang\@ math.utoronto.edu})}}

\title{Arnold diffusion in arbitrary degrees of freedom \\
and crumpled $3$-dimensional normally hyperbolic invariant cylinders}

\begin{document}
\maketitle

\begin{abstract}
In the present paper we prove a form of Arnold diffusion.
The main result says that for
a ''generic'' perturbation of a nearly integrable system of arbitrary
degrees of freedom $n\ge 2$
\[
H_0(p)+\eps H_1(\th,p,t),\quad \th\in \T^n,\ p\in B^n,\ t\in \T=\R/\T,
\]
with strictly convex $H_0$ there exists  an orbit $(\th_{\e},p_{e})(t)$
exhibiting Arnold diffusion in the sens that
\[
\sup_{t>0}\|p(t)-p(0) \| >l(H_1)>0
\]
where $l(H_1)$ is a positive constant independant of  $\e$.

Our proof is a combination of geometric and variational methods.
We first build  $3$-dimensional
normally hyperbolic invariant cylinders of limited regularity, but of large size,
extrapolating on \cite{Be3} and \cite{KZZ}.
Once these cylinders are constructed we use versions of
Mather variational method developed in Bernard \cite{Be1},
Cheng-Yan \cite{CY1, CY2}.
\end{abstract}

\markboth{P. Bernard, V. Kaloshin, K. Zhang}{Arnold diffusion along crumpled normally hyperbolic invariant cylinders}

\section{Introduction}

Let $(\th,p)\in \T^n \times U$ be the phase space
of an integrable Hamiltonian system $H_0(p)$ with
$\T^n$ being the $n$-dimensional torus $\T^n=\R^n / \Z^n
\ni \th=(\th_1,\cdots,\th_n)$ and $U$ being an open
set in $\R^n$, $p=(p_1,\cdots,p_n)\in B^n$. Assume
that $H_0$ is strictly convex, i.e. Hessian
$\partial^2_{p_ip_j} H_0$ is strictly positive definite.

Consider a smooth time periodic perturbation
\[
H_\eps(\th,p,t)=H_0(p)+\eps H_1(\th,p,t),\quad t\in \T=\R/\T.
\]
We study Arnold diffusion for this system, namely,
existence of orbits $\{(\th,p)(t)\}_t$ such that
\[
|p(t)-p(0)|>O(1) \quad \textup{independently of }\eps.
\]

Integer relations $\vec k_1 \cdot \partial_p H_0  +k_0=0$ with
$\vec k=(\vec k_1,k_0) \in (\Z^n \setminus 0) \times \Z$ and $\cdot$
being the inner product define {\it one-dimensional resonances}.
Under the condition that Hessian of $H_0$ is non-degenerate,
these resonances define smooth hyper-surfaces embedded into
action space
$$
\Gm_{\vec k}=\{p\in B^n:\ \vec k_1 \cdot \partial_p H_0 +k_0=0\}.
$$
If one intersects $k$ linearly independent resonances
$\{\vec k_j\}_{j=1}^n$, we get a {\it $k$-dimensional resonance}
$\cap \Gm_{\vec k^j}$, which is defined by an $(n-k)$-dimensional
surface in $B^n$.

\subsection{Apriori unstable systems}

In the case $n=2$ Arnold proposed the following example
\[
H(I\,\phi,q,p,t)=\dfrac{I^2}{2}+\dfrac{p^2}{2}+
\eps(1-\cos q)(1+\mu (\sin \th+\sin t )).
\]
The feature of this example is that it has a $3$-dimensional
normally hyperbolic invariant  cylinder (NHIC). This allows
in a sense to reduce $5$-dimensional dynamics to
a $3$-dimensional one. There is a rich literature on Arnold
example and we do not intend to give extensive list of references;
we mention \cite{AKN, BB, Be4, Bs1, Zha}, and references therein.

This example gave rise to a family of
examples of systems of $n+1/2$ degrees of freedom of the form
\[
H_\eps(I,\phi,p,q,t)=H_0(I)+K_0(p,q)+\eps H_1(I,\phi,p,q,t),
\]
where $(q,p) \in \T^{n-1}\times \R^{n-1},\ I\in \R,\ \phi, t \in \T$.
Moreover, the Hamiltonian $K_0(p,q)$ has a saddle fixed point
at the origin and $K_0(0,q)$ attains its strict maximum at $q=0$.
For small $\eps$ this system has a $3$-dimensional NHIC $\Lb$.

For $n=2$ systems of this type were successfully studied by
different groups. Two groups were using deep geometric methods.

-- In \cite{DH, DGLS, DLS} the authors carefully analyze two types
of dynamics induced on the cylinder $\Lb$. These two dynamics
are given by so-called inner and outer maps.

--- In \cite{T1,T2} a return ({\it separatrix}) map along invariant of
$\Lb$ is constructed. A detailed analysis of this separatrix map
gives diffusing orbits.

The other two groups \cite{Be1,CY1} are inspired and influenced by
Mather variation method \cite{Ma1,Ma2} and build diffusing orbits
variationally. We essentially rely on their technique in this paper.

In the case $n>2$ Treshev \cite{T1,T2} and Cheng-Yan \cite{CY2}
also managed to prove diffusion. In Treschev's paper he even
showed existence of orbits with an optimal diffusion speed
$\sim |\ln \eps|/\eps$. Other examples of Arnold diffusion can be found in
\cite{Bs2,Bo,BK,KL1,KL2,KLS,KS,KZZ,LM,MS,Mo,Zhe}.

\subsection{Dynamics along a single resonance for $n=2$
and crumpled invariant cylinders}

Pick a single resonant line $\Gm_{\vec k} \subset B^2$.
Notice that on a dense set in $\Gm_{\vec k}$ there is
an additional resonant relation. If one deletes a finite number
of such additional resonant relations with relatively small $k$'s
usually called {\it double resonances},
then on each segment between consecutive deleted points one can
hope to find a ``nice'' smooth NHIC. This turns out to be {\it wrong}\,!
However, one can indeed find a NHIC whose regularity blows up
as $\eps \to 0$. Moreover, one can use this cylinder for
Arnold diffusion. This is {\it a new phenomenon discovered in this paper.}

\subsection{Dynamics along a codimension one resonance for $n>2$
and crumpled invariant cylinders}

Pick a codimension one (or dimension $n-1$) resonant line
$\Gm=\cap_{j=1}^{n-1} \Gm_{\vec k_j} \subset B^n$ with
$\{\vec k_j\}_{j=1}^{n-1}$ being linearly independent. As before
on a dense set in $\Gm$ there is an additional resonant relation.
However, qualitatively the picture as in the case $n=2$.
Namely, if one deletes a finite number of such additional resonant
relations with relatively small $k$'s, then on each segment of $\Gm$
between consecutive deleted points there is a {\it crumpled
$3$-dimensional NHIC} $\Lb$. It is a crumpled in the sense that
{\it its regularity blows up as} $\eps \to 0$.  With some efforts
this allows to reduce dynamics to $3$-dimensional one and
essentially reduce the proof to the twist maps case.

\subsection{Main result}
We study dynamics near a resonance of
codimension one, i.e. near a segment in $B^n$. For any
resonance of codimension one there is an integer linear
symplectic transformation which brings integer vectors
$k_1,\dots,k_{n-1}\in \Z^n$, defining the resonance, to
the form $k_j=(0,\cdots,1_j,0,\cdots,0)$. Since we are
interested in a local property assume that a resonance,
denoted $\Gm$, of codimension one is of the following form:
\[
(\partial_{p_1} H_0(p),\cdots, \partial_{p_{n-1}} H_0(p))
=(\dot \th_1, \cdots ,\dot \th_{n-1})=0 \quad
\textup{ for }\eps=0.
\]
In the case $H_0(p)=\frac 12 \sum_{j=1}^n p_j^2$ we have
$\Gm=\{(p_1,\cdots,p_{n-1})=0\}$. Thus, it is naturally
parametrized by $p_n$.

Consider the space of $C^r$ perturbations
$C^r(\T^n \times B^n \times \T,\R)$ with a natural $C^r$ norm
given maximum of all partial derivatives of order up to $r$.
 Denote by $S^r$ the unit sphere in this space.

\begin{thm}\label{main} For $r\ge 4$, there is an open and
dense set ${\mathcal U} \subset S^r$, a nonnegative function
$l:S^r \to \R_+$ with $l|_{\mathcal U}>0$ and a positive function
$\eps_0=\eps_0(H_1)$, we write ${\mathcal V}=\{\epsilon H_1:
H_1\in {\mathcal U} ,\ \ 0<\eps<\eps_0\}$. We have that, for an open
and dense set of $\epsilon H_1\in {\mathcal V}$ the Hamiltonian
system $H_\epsilon=H_0 + \epsilon H_1$  has an orbit
$\{(\th,p)(t)\}_t$ whose action component
\[
\|p(T)-p(0)\|>l(H_1).
\]
Moreover, for all $0<t<T$ the action component $p(t)$
stays close to the codimension one resonance $\Gm$.
\end{thm}

\begin{rmk} This Theorem provides a form of Arnold diffusion for
generic Hamiltonian systems. The type of generic condition in
Theorem~\ref{main} is a version of Mather's cusp residue condition
introduced in \cite{Ma3}.
\end{rmk}

The present work is in large part inspired by the work of Mather \cite{Ma3,Ma4,Ma5}.
In \cite{Ma3}, Mather announced a much stronger version of Arnold diffusion for $n=2$
(the system is time-periodic hence the degree of freedom is $2\frac12$).
The proof of Mather's result is partially written (see \cite{Ma4}),
and he has given lectures about some parts of the proof \cite{Ma5}.
One of the ideas underlying his proof is to construct diffusion along a segment
of a resonance and away from other low order resonances. Conceptually, the proof
of our result has similar features to parts of Mather's proof \cite{Ma4}
for single resonances. The novelty of our approach is the use of normal form
theory and construction of normally hyperbolic cylinders in
an \emph{a priori stable} setting.
Application of normal forms to construct normally $3$-dimensional
hyperbolic invariant cylinders in apriori stable situation in $3$ degrees
of freedom is proposed in \cite{KZZ}.  Independently in the case of
arbitrary degrees of freedom it is proposed in \cite{Be3}. In the latter
it is shown that such cylinders have length independent of $\epsilon$.

\subsection{Plan of the proof}
The proof of this Theorem proceeds in three steps.

{\it Step 1.} Build a normal form for $H_\eps$
for for $p$ near $\Gm$. In section~\ref{sec:normal-form}
 we prove the existence of a normal form, which takes
a particular nice form along subsegments of $\Gamma$,
which we will call passage segments, defined in the
next section. The length and choice of the passage
segments depends on $H_0$ and $H_1$ only.

{\it Step 2.} For $H_1\in {\mathcal U}$, we establish existence
 of finitely many  $3$-dimensional normally hyperbolic cylinder
along  $\Gm$. This is discussed in Section~\ref{sec:NHIC}.

{\it Step 3.} For a generic perturbation, we show that there
exists diffusion orbit along a passage segment, using the
normally hyperbolic cylinders. This steps uses variational methods
of Bernard \cite{Be1} and of Cheng-Yan \cite{CY1, CY2} which
are based on ideas of Mather (see \cite{Ma4}).
These constructions are discussed in
Section~\ref{sec:localization} and Section~\ref{sec:variational}.

\section{Notations and terminology}
\label{notations}

We denote $\theta^s=(\theta_1, \cdots, \theta_{n-1})$, $p^s=(p_1, \cdots, p_{n-1})$ and $\theta^f=\theta_n$, $p^f=p_n$. These are the slow-fast variables associated to the resonance $\Gamma=\{\partial_{p^s}H_0(p)=0\}$. It is natural to use $p^f$ as a parameter for $\Gamma$, i.e.  we may write $\Gamma\cap B=\{p_*(p^f)=(p^s_*(p^f), p^f), p^f\in [a_{min}, a_{max}]\}$.
\[
Z(\theta^s, p):= \iint H_1(\th^s,p^s,\th^f,p^f,t)\, d\th^f\, dt.
\]
If the perturbation $H_1(\theta,p,t)$
is expanded as
$$
H_1(\theta,p,t)=H_1(\theta_s,\theta_f,p,t)=
\sum_{k^s\in \Zm^{n-1},k^f\in \Zm,l\in \Zm} h_{[k^s,k^f,l]}(p)
e^{2i\pi(k^s\cdot \theta^s+k^f\cdot \theta^f+l\cdot t)},
$$
 then
$$
Z(\theta^s,p)=
\sum_{k^s} h_{[k^s,0,0]}(p)
e^{2i\pi(k^s\cdot \theta^s)}.
$$

We would like to impose the following set of non-degeneracies and
notations. Consider the function $Z(\theta^s, p_*(p^f))$ as a family of functions on $\T^{n-1}$ parametrized by $p^f$.

Call a value $p^f$ on $\Gm$ \emph{regular} if $Z(\theta^s, p_*(p^f))$ has a unique global maximum on $\T^s \ni \th^s$ at some $\th^s_*=\th^s(p^f)$. We say the maximum is non-degenerate if the Hessian of $Z$ with respect to $\theta^s$ is strictly negative definite.

Call a value $p^f$ on $\Gm$ \emph{bifurcation} if $Z(\theta^s, p_*(p^f))$ has exactly two global maxima on $\T^s \ni \th^s$ at some
$\th^s_1=\th^s_1(p^f)$ and $\th^s_2=\th^s_2(p^f)$.

Call a regular $p^f$ on $\Gm$ \emph{non-degenerate} if the unique maximum is non-degenerate. If $p^f$ is a bifurcation, it is called non-degenerate if both maxima are non-degenerate, furthermore, the values at these maxima moves with different speed with respect to the parameter $p^f$ Otherwise, it is called {\it degenerate}.

The generic condition that defines ${\mathcal U}\subset S_r$ is a higher dimensional version of the conditions (C1)-(C3) given by Mather \cite{Ma3}. These conditions may be described as follows: Each value $p^f\in [a_{min}, a_{max}]$ is a non-degenerate regular or bifurcation point. Note that the non-degeneracy condition implies that there are at most finitely many bifurcation points. Let $a_1<\cdots<a_{s-1}$ be the set of bifurcation points in the interval $(a_{min}, a_{max})$, and consider the partition of the interval $[a_{min}, a_{max}]$ by  $\{[a_j, a_{j+1}]\}_{j=0}^{s-1}$.  Here we give an explicit quantitative version of the above condition: There exists $\lambda>0$ such that
\begin{enumerate}
\item[{[G0]}] There are smooth functions $\theta^s_j(p^f):[a_j-\lambda, a_{j+1}+\lambda]\to \T^n$, $j=0, \cdots, s-1$, such that
for each $p^f\in [a_j-\lambda, a_{j+1}+\lambda]$,  $\theta^s_j(p^f)$ is a local maximum of $Z(\theta^s, p_*(p^f))$ satisfying
$$
\lambda I \le -\partial^2_{\theta^s \theta^s}Z(\theta^s_j, p)\le  I,
$$
where $I$ is the identity matrix.
\item[{[G1]}]
For $p^f\in (a_j, a_{j+1})$, $\theta^s_j$ is the unique maximum for $Z$. For $p^f=a_{j+1}$, $\theta^s_j$ and $\theta^s_{j+1}$ are the only maxima.
\item[{[G2]}] At  $p^f=a_{j+1}$ the maximum value of $Z$ has different derivatives with respect to $p^f$, i.e.
$$\frac{d}{d p^f}Z(\theta^s_j(a_{j+1}),p_*(p^f)) \ne  \frac{d}{d p^f}Z(\theta^s_{j+1}(a_{j+1}),p_*(p^f)). $$
\end{enumerate}

\begin{thm}\label{open-dense}
  The set  ${\mathcal U}$ of functions $H_1\in S_r$ such that the corresponding $Z(\theta^s, p)$ satisfies  conditions [G0]-[G2] is open and dense.
\end{thm}

The proof of Theorem~\ref{open-dense} will be given in Appendix~\ref{sec:generic}.

Write $\omega(p)=\partial_pH_0(p)=(\partial_{p^s}H_0, \partial_{p^f}H_0)$, clearly for any $p\in \Gamma$ we have that $\omega(p)=(0, \partial_{p^f}H_0))$. We say that $p^f$ has an additional resonance if there exists integers $k_n, l$ such that $k_n\partial_{p^f}H_0(p)+l=0$. Given a large integer $K$, let
\begin{equation}\label{eq:sigma} \Sigma_K=\{ p\in \Gamma\cap B; \quad \exists k_n, l \in \Z, \, |k_n|, |l|\le K, \, k_n\cdot \partial_{p^f}H_0(p)+l=0\}.
\end{equation}
 Given $H_1\in {\mathcal U}$, we will define a small $\delta=\delta(H_1, n, r)>0$ and integer $K=K(\delta, n,r)$ and call the elements of  $\Sigma_K$  \emph{punctures}. We need to exclude a neighborhood of the punctures from $\Gamma\cap B$. Let $U_{3\epsilon^{\frac16}}(\Sigma_K)$ stand for $3\epsilon^{\frac16}$ neighborhood of $\Sigma_K$,  then $\Gamma\cap B \setminus U_{3\epsilon^{\frac16}}(\Sigma_K)$ is a collection of disjoint segments. Each of these segments is called a \emph{passage segment}. On a neighborhood of each passage segment there exists a convenient normal form for the Hamiltonian $H_\epsilon$.

\section{Normal forms}\label{sec:normal-form}
Let $\Gamma=\{(p^s=p_*(p^f))\}$ be
the resonant line of equation  $\partial_{p^s}H_0=0$.
For  $p\in \Gamma$ we have  $\omega(p)=(0, \partial_{p^f}H_0)$.
 We say that
 $p$ has an additional resonance if the remaining frequency
$\partial_{p^f}H_0(p)$ is rational.
In order to reduce the system to an appropriate normal form,
we must remove some additional resonances.
More precisely, let $\mD(K,s)\subset B$ be the set of momenta $p$ such that
\begin{itemize}
 \item $\|\partial_{p^s} H_0(p)\|\leq s$, and
\item $|k^f\partial_{p^f} H_0(p) +k^t|\geq 3Ks\quad$ for each $(k^f,k^t)\in \Zm^2$ satisfying
$\max(k^f,k^t)\in ]0,K]$.
\end{itemize}

\begin{thm}\label{normal-form}[Normal Form]
Let $H_0(p)$ be a $C^4$ Hamiltonian.
For each $\delta \in ]0,1[$, there exists positive parameters $K_0, \e_0, \beta$
such that, for each $C^4$ Hamiltonian $H_1$ with $\|H_1\|_{C^4}\leq 1$ and each
$
K\geq K_0,\epsilon \leq \epsilon_0,
$
there exists
a $C^2$ change of coordinates
$$
\Phi :\Tm^n\times B \times \Tm
\lto \Tm^n\times \Rm^n\times \Tm
$$
satisfying $\|\Phi-id\|_{C^0}\leq \sqrt{\e} $ and  $\|\Phi-id\|_{C^2}\leq \delta$
and such that, in the new coordinates, the Hamiltonian $H_0+\e H_1$
takes the form
\begin{equation}\label{eq:normal-form}
 N_\epsilon =  H_0(p) + \epsilon Z(\theta^s, p) + \epsilon R(\theta, p, t),
\end{equation}
with   $\|R\|_{C^2}\le \delta$
on $\Tm^n\times \mathcal{D}(K,\beta \e^{1/4})\times \Tm$.
We can take $K_0=c\delta^{-2}, \beta=c\delta^{-1-n}, \e_0=\delta^{6n+5}/c$,
where $c>0$ is some constant depending only on $n$ and $\|H_0\|_{C^4}$.
\end{thm}
The proof actually builds   a symplectic diffeomorphism
$\tilde \Phi$ of $\Tm^{n+1}\times \Rm^{n+1}$
of the form
$$
\tilde \Phi (\theta,p,t,e)=\big(\Phi(\theta, p,t),e+f(\theta,p,t)\big)
$$
and such that
$$
N_{\e}+e=(H_{\e}+e)\circ \tilde \Phi.
$$
We have the estimates $\|\tilde \Phi-id\|_{C^0}\leq \sqrt{\e} $ and
 $\|\tilde \Phi-id\|_{C^2}\leq \delta$.

\begin{rmk}\label{length}[Length of passage segment]
  On the interval, the distance between 2 adjacent rationals with denominator at most $K$ is $\frac{1}{K^2}$. It follows that the distance between $p^f_1, p^f_2\in \Sigma_K$ (see \eqref{eq:sigma}) is at least $\|\partial^2H_0^{-1}\|\frac{1}{K^2}\ge c (c(n,r)^{-1}\delta)^{\frac{4}{r-3}}$, assuming that $\|\partial^2 H_0^{-1}\|$ is bounded by some universal constant.
\end{rmk}

To prove Theorem~\ref{normal-form} we proceed in $3$ steps.
We first obtain a global normal form $N_{\e}$
adapted to all resonances.
 We then show that this normal form takes the desired form on the domain
$
\mathcal{D}_{K,\e}.
$
 However, the averaging procedure lowers smoothness,
in particular, the technique requires the smoothness $r\ge n+5$.
 To obtain a result that does not require this relation between $r$ and $n$,
we use a smooth approximation trick that goes back to Moser.

\subsection{A global normal form adapted to all resonances.}

We first state a result for autonomous systems.
The time periodic version will come as a corollary.
Consider the Hamiltonian $H_\epsilon(\phi, J)=H_0(J)+\epsilon H_1(\phi, J)$, where $(\phi,J)\in \T^m \times \R^m$
(later, we will take $m=n+1$). Let $B=\{|J|\le 1\}$ be the unit ball in $\R^m$. Given any integer vector $k\in \Z^m\setminus\{0\}$, let $[k]=\max\{k_i\}$. To avoid zero denominators in some calculations, we make the unusual convention that $[(0, \cdots, 0)]=1$.
We fix once and for all a bump function  $\rho:\R\to \R$ be a $C^\infty$
such that
$$\rho(x)=
\begin{cases}
  1, & |x|\le 1 \\
  0, & |x|\ge 2
\end{cases}
$$
and $0<\rho(x) < 1$ in between. For each  $\beta>0$ and $k\in \Z^m$,
we define the function
$\rho_k(J)=\rho(\frac{k\cdot \partial_J H_0}{\beta \epsilon^{1/4}[k]})$,
where $\beta>0$ is a parameter.
%
\begin{thm}\label{autonomous}
There exists a constant $c_m>0$, which depends only on $m$, such that the
following holds.
Given :
\begin{itemize}
 \item A $C^4$ Hamiltonian $H_0(J)$,
\item A $C^r$ Hamiltonian
$H_1(\varphi,J)$ with $\|H_1\|_{C^r}=1$,
\item Parameters $r\geq m+4$, $\delta \in ]0,1[$, $\epsilon\in ]0,1[$,
$\beta>0$, $K>0$,
\end{itemize}
satisfying
\begin{itemize}
 \item $K \geq c_m \delta^{\frac{-1}{r-m-3}}$,
\item $\beta \geq c_m (1+\|H_0\|_{C^4})\delta^{-1/2},$
\item $\beta\e^{1/4}\leq \|H_0\|_{C^r}$,
\end{itemize}
there exists a $C^2$ symplectic diffeomorphism $\Phi$ such that,
in the new coordinates, the Hamiltonian $H_{\e}=H_0+\e H_1$ takes the form
$$
H_{\e}\circ \Phi=H_0+\e R_1+\e R_2
$$
with
\begin{itemize}
\item $R_1=\sum_{k\in \Z^m, |k|\le K}\rho_k(J)h_k(J)e^{2\pi i (k\cdot \phi)}$, here $h_k(J)$ is the $k^{th}$ coefficient for the Fourier expansion of $H_1$,
\item $\|R_2\|_{C^r}\le \delta$,
\item  $\|\Phi-id\|_{C^0}\le \delta\sqrt{\e}$ and $\|\Phi-id\|_{C^2}\le \delta.$
\end{itemize}
\end{thm}

We  now  prove Theorem~\ref{autonomous}.
To avoid cumbersome notations, we will denote by $c_m$ various different constants
depending only on the dimension $m$.
We have the following basic estimates about  the Fourier series of a function $g(\phi, J)$.
Given a multi-index
$\alpha=(\alpha_1, \cdots, \alpha_m)$, we denote  $|\alpha|=\alpha_1+\cdots+ \alpha_m$.

\begin{lem}\label{fourier-est}
  For $g(\phi,J)\in C^r(\T^m\times B)$, we have
  \begin{enumerate}
  \item
If $l\leq r$, we have
$\|g_k(J)e^{2\pi i (k\cdot \varphi)}\|_{C^l}\le [k]^{l-r}\|g\|_{C^r}$.

  \item Let $g_k(J)$ be a series of functions such that
the inequality
$\|\partial_{J^{\alpha}}g_k\|_{C^0}\leq M[k]^{-|\alpha|-m-1}$  holds for each mult-index $\alpha$
with $|\alpha|\le l$, for some $M>0$.  Then, we have \newline 
$\|\sum_{k\in \Z^m}g_k(J)e^{2\pi i(k\cdot\varphi)}\|_{C^l}\le c \kappa_m M$.
  \item Let $\Pi^+_Kg=\sum_{|k|>K}g_k(J)e^{2\pi i (k\cdot \phi)}$. Then for $l\leq r-m-1$,
we have   $\|\Pi^+_kg\|_{C^l}\le \kappa_m K^{m-r+l+1}\|g\|_{C^r}$.
  \end{enumerate}
\end{lem}
\begin{proof}
  1. Let us assume that $k\neq 0$ and take $j$
  such that $k_j=[k]$.
  Let $\alpha$ and $\eta$ be two multi-indices
  such that $|\alpha+\eta|\leq l$.
  Finally, let $b=r-l$, and let $\beta$ be the multi-index
   $\beta=(0,\ldots,0,b,0,\ldots,0)$, where $\beta_j=b$.
We have
$$
g_k(J)e^{2i\pi(k,\varphi)}=
\int_{\Tm^m} g(\theta, J)e^{2i\pi(k,\varphi-\theta)}d\theta
=
\int_{\Tm^m} g(\theta+\varphi, J)e^{-2i\pi(k,\theta)}d\theta,
$$
  hence
\begin{align*}
\partial_{\varphi^{\alpha}J^{\eta}}\big(
g_k(J)e^{2i\pi(k,\varphi)}\big)
&=
\int_{\Tm^m} \partial_{\varphi^{\alpha}J^{\eta}}g(\theta+\varphi, J)e^{-2i\pi(k,\theta)}d\theta,\\
&=
\int_{\Tm^m} \frac{\partial_{\varphi^{\alpha+\beta}J^{\eta}}
g(\theta+\varphi, J)}
{(2i\pi k_j)^b}
e^{-2i\pi(k,\theta)}d\theta.
\end{align*}
  Since $|\alpha+\beta+\eta|\leq r$, we conclude that
   $$
\|g_k(J)   e^{2i\pi(k,\varphi)}\|_{C^l}\leq
\|g\|_{C^r}/(2\pi[k])^b\leq \|g\|_{C^r}[k]^{l-r}.
   $$

2. We have
$\|g_k(J)e^{2i\pi(k\cdot \varphi)}\|_{C^l}\leq
$
$$\|\sum_{k\in \Z^m}h_k(J)e^{2\pi i(k\cdot\varphi)}\|_{C^l} \le \sum_{k\in \Z^m}c_l|k|^{-r+l}M\le c_l \kappa_m M, $$
recall that $\kappa_m=\sum_{k\in \Z^m}|k|^{-m-1}$.

3.
\begin{multline*} \|\Pi^+_Kg\|_{C^2} \le c \sum_{|k|>K} |k|^{-r+2}\|g\|_{C^r} \le c K^{-r+m+3}\sum_{|k|>K}|k|^{-m-1}\|g\|_{C^r}\\
\le c K^{-r+m+3}\kappa_m \|g\|_{C^r}= c \kappa_m K^{-r+m+3} \|g\|_{C^r}.
\end{multline*}
3. Using 1., we get
\begin{align*} \|\Pi^+_Kg\|_{C^l}
&\le
 \sum_{|k|>K} [k]^{l-r}\|g\|_{C^r}
 \le
 \|g\|_{C^r} K^{m-r+l+1}\sum_{|k|>K}[k]^{-m-1}\\
 &\le
 \|g\|_{C^r} K^{m-r+l+1}\sum_{k\in \Zm^m}[k]^{-m-1}.
\end{align*}
\end{proof}

\begin{proof}[Proof of Theorem~\ref{autonomous}]

Let $G(\phi,J)$ be the function that solves the cohomological equation
$$ \{H_0,G\}+H_1=R_1 +R_+,$$
where $R_+=\Pi^+_K H_1$.
Observing that $\rho_k(J)=1$
when $k\cdot \partial_JH_0=0$,
we have the following explicit formula for $G$:
$$ G(\varphi,J)=\sum_{|k|\le K}\frac{(1-\rho_k(J))h_k(J)}{k\cdot \partial_JH_0}e^{2\pi i (k\cdot \phi)} $$
where each of the functions
$(1-\rho_k(J))h_k(J)/(k\cdot \partial_J H_0)$
is extended by continuity at the points where
the denominator vanishes. This function hence
takes the value zero at these points.
$G$ is well defined thanks to the smoothing terms $1-\rho_k$ we introduced, as whenever $k \cdot \partial_J H_0=0$ we also have $1-\rho_k=0$ and that term is considered non-present.

Let $\Phi^t$ be the Hamiltonian flow generated by  $\epsilon G$.
Setting $F_t =  R_1 + R_+ + t(H_1-R_1-R_+)$,
we have the standard computation
\begin{align*}
\partial_t\big((H_0+\e F_t)\circ \Phi^t)\big)
&=
\e\partial_t F_t \circ \Phi^t+\e\{H_0+\e F_t,G\}\circ \Phi^t\\
&=\e\big(\partial_tF_t+\{H_0,G\}\big)\circ \Phi^t+
\e^2\{F_t,G\}\circ \Phi^t\\
&=\e^2\{F_t,G\}\circ \Phi^t,
\end{align*}
from which follows that
$$ H_\epsilon \circ \Phi^1= H_0 + \epsilon R_1 + \epsilon R_+ + \epsilon^2 \int_0^1 \{F_t, G\}\circ \Phi^t dt. $$
Let us estimate the $C^2$ norm
of the function
 $R_2 := R_+ + \epsilon \int_0^1 \{F_t, G\}\circ \Phi^t dt$.
It follows from Lemma~\ref{fourier-est} that
$$
\|R_+\|_{C^2}\le  \kappa_m K^{-r+m+2} \|H_1\|_{C^r}\le \frac12 \delta.
$$
We now focus on the term
  $ \int_0^1 \{F_t, G\}\circ \Phi^t dt$.
  To estimate the norm of $F_t$, it is convenient to write $F_t =\tilde F_t + (1-t)R_1$, where $\tilde F_t=(1-t)R_++tH_1$. Notice that the coefficients of the Fourier expansion of $\tilde F_t$ is simply a constant times that of $H_1$, Lemma~\ref{fourier-est} then implies that
$$
\|\tilde F_t\|_{C^3}\le
\sum_{k\in{\Zm^m}} [k]^{3-r}\|H_1\|_{C^r}
=
 \kappa_m \|H_1\|_{C^r}
$$
provided that  $r\ge m+4$, where $\kappa_m=\sum_{\Zm^n}[k]^{m+1}$.

We now have to estimate the norm of $R_1$ and $G$. These estimates require additional estimates of the smoothing terms $\rho_k$ as well as the small denominators $k\cdot \partial_JH_0$. We always assume that $l\in \{0,1,2,3\}$ in the following estimates:
\begin{itemize}
\item[-]  $\rho_k(J)\ne 1\quad \Rightarrow \quad
|(k\cdot \partial_J H_0)^{-1}|\le \beta^{-1}\epsilon^{-1/4} |k|^{-1}$.
\item[-]
$\|(k\cdot \partial_JH_0)^{-1}\|_{C^l}\le c_m \beta^{-l-1}
\epsilon^{-(l+1)/4}\|H_0\|_{C^4}^{l+1}$ on $\{\rho_k\neq 1\}$.
\item[-]
$ \|\rho_k(J)\|_{C^l} \le c_m\beta^{-l} \epsilon^{-l/4} \|H_0\|_{C^4}^l$ and
$ \|1-\rho_k(J)\|_{C^l} \le c_m \beta^{-l} \epsilon^{-l/4} \|H_0\|_{C^4}^l.$
\end{itemize}
We have been using the following estimates on the derivative of
composition of functions: For $f:\R^m\to \R$ and $g: \R^m\to \R^m$ we have
$\|f\circ g\|_{C^l}\le c_{m,l}\|f\|_{C^l}(1+\|g\|_{C^l}^l)$.
\begin{itemize}
\item[-] For each multi-index $|\alpha|\le 3$, we have that
\[
\begin{aligned}&\|\partial_{J^\alpha} \left((1-\rho_k(J))h_k(J)(k\cdot \partial_JH_0)^{-1}\right) \|_{C^0} \\
\le & \sum_{\alpha_1+\alpha_2+\alpha_3=\alpha} \|1-\rho_k(J)\|_{C^{|\alpha_1|}}\|h_k\|_{C^{|\alpha_2|}}\|(k\cdot \partial_JH_0)^{-1}\|_{C^{|\alpha_3|}
(\{\rho_k\neq 1\})} \\
\le&  c_m \sum_{\alpha_1+\alpha_2+\alpha_3= \alpha}  \Big(
\beta^{-|\alpha_1|}\epsilon^{-|\alpha_1|/4} \|H_0\|_{C^4}^{|\alpha_1|}
\cdot [k]^{-r+|\alpha_2|} \|H_1\|_{C^r} \\
&\cdot \beta ^{-|\alpha_3|-1}\epsilon^{-(|\alpha_3|+1)/4}
\|H_0\|_{C^4}^{|\alpha_3|+1}\Big) \\
\le & c_m \beta^{-|\alpha|-1}\e^{-(|\alpha|+1)/4}  [k]^{|\alpha|-r}
 \|H_0\|_{C^4}^{|\alpha|+1} \|H_1\|_{C^r} .
\end{aligned}
\]
\end{itemize}
In these computations, we have used the hypothesis
$\beta \e^{1/4}\leq \|H_0\|_{C^4}$.
Since $G(\varphi, J)=\sum_{k\in \Z^m}(1-\rho_k(J))h_k(J)(k\cdot \partial_JH_0)^{-1}e^{2\pi i (k\cdot \varphi)}$,   Lemma~\ref{fourier-est} implies (since $r\geq m+1$) :
\begin{itemize}
\item[-]
 $\|G\|_{C^l}\le  c_m \beta^{-l-1}\epsilon^{-(l+1)/4}
\|H_0\|_{C^4}^{l+1} \|H_1\|_{C^r}\leq \epsilon^{-1} . $
 \end{itemize}
We now turn our attention to $R_1=\sum_{|k|\le K}\rho_k(J)h_k(J)e^{2i\pi (k\cdot \phi)}$:
 \begin{itemize}
 \item[-] $\|h_k\|_{C^l}\leq [k]^{l-r}\|H_1\|_{C^r}$.
\item[-]$
\|\rho_kh_k\|_{C^l}\le c_m \beta^{-l}  \epsilon^{-l/4}[k]^{-r+l}\|H_0\|_{C^4}^l\|H_1\|_{C^r}.
$
\item[-]
 $\|R_1\|_{C^l}\le c_m \beta^{-l}  \epsilon^{-l/4}\|H_0\|_{C^4}^l\|H_1\|_{C^r}$, provided  $r\ge m+4$.
\end{itemize}
We obtain
 $$\|F_t\|_{C^l}\le \|R_1\|_{C^l}+\|\tilde F_t\|_{C^l}
\le c_m \beta^{-l}  \epsilon^{-l/4}\|H_0\|_{C^4}^l \|H_1\|_{C^r},
 $$
and
$$ \|\{F_t, G\}\|_{C^2}\le
\sum_{|\alpha_1+\alpha_2|\le 3}
 \|F_t\|_{C^{|\alpha_1|}}\|G\|_{C^{|\alpha_2|}} \le
 c_m \beta^{-4} \epsilon^{-1}\|H_0\|_{C^4}^4 \|H_1\|_{C^r}^2 .
$$
Concerning the flow $\Phi^t$,
we observe that $\|\epsilon G\|_{C^3}\leq 1$, and get the following estimate
(see \textit{e. g.} \cite{DH}):
\begin{itemize}
\item[-]$\|\Phi^t-id\|_{C^2}\le
c_m\e\|G\|_{C^3}\le c_m \beta^{-4}
\|H_0\|_{C^4}^{4} \|H_1\|_{C^r}\leq \delta ,$
\item[-]
$\|\Phi^t-id\|_{C^0}\leq c_m \e\|G\|_{C^1}\leq c_m \beta^{-2} \sqrt{\e}
\|H_0\|_{C^4}^2\|H_1\|_{C^2}\leq \delta  \sqrt\e .$
\end{itemize}
Finally,
we obtain
\begin{align*}
 \epsilon\|\{F_t, G\}\circ \Phi^t\|_{C^2}
&\le c_m \epsilon \|\{F_t, G\}\|_{C^2}\|\Phi^t\|_{C^2}^2
\\
&\le c_m \beta^{-4} \|H_0\|^4_{C^4} \|H_1\|_{C^r}^2 \le \delta/2.
 \end{align*}
\end{proof}

\subsection{Normal form away from additional resonances}

We now return to our non-autonomous system and apply
Theorem \ref{autonomous} around the resonance under study.
To the non-autonomous Hamiltonian
$$
H_{\e}(\theta, p,t)=H_0(p)+\e H_1(\theta,p,t):\Tm^n\times \Rm^n\times\Tm\lto \Rm
$$
we associate the autonomous Hamiltonian
$$
\tilde H_e(\varphi,J)=H_0(I)+e+\e H_1(\theta, I,t):\Tm^{n+1}\times \Rm^{n+1} \lto \Rm,
$$
where $\varphi=(\theta, t)$ and $J=(I,e)$.
We denote the frequencies $\omega\in \Rm^{n+1}$  by
$\omega =(\omega^f,\omega^s,\omega^t)\in \Rm^{n-1}\times \Rm \times \Rm$, and define the set
$$
\Omega(K,s):=\{ \omega\in \Rm^{n+1}:\,
\|\omega^s\|> s, \,
|k^f\omega^f+k^t\omega^t|\geq 3sK \quad\forall (k^s,k^t)\in \Zm^2_K
\},
$$
where we have denoted by $\Zm^2_K$ the set of pairs $(k^f,k^t)$ of integers
such that $0<\max(k^f,k^t)\leq K$.
Note that
$$
\mD (K,s)=\{ p\in \Rm^n : (\partial_p H_0(p),1)\in \Omega(K,s)
\}.
$$

\begin{cor}\label{norm}
There exists a constant $c_n>0$, which depends only on $n$, such that the
following holds.
Given :
\begin{itemize}
 \item A $C^4$ Hamiltonian $H_0(p)$,
\item A $C^r$ Hamiltonian
$H_1(\theta,p,t)$ with $\|H_1\|_{C^r}=1$,
\item Parameters $r\geq n+5$, $\delta \in ]0,1[$, $\epsilon\in ]0,1[$,
$\beta>0$, $K>0$,
\end{itemize}
satisfying
\begin{itemize}
 \item $K \geq c_n\delta^{\frac{-1}{r-n-4}}$,
\item $\beta \geq c_n(1+\|H_0\|_{C^4})\delta^{-1/2},$
\item $\beta\e^{1/4}\leq \|H_0\|_{C^r}$,
\end{itemize}
there exists a $C^2$ symplectic  diffeomorphism $\tilde \Phi$ of $\Tm^{n+1}\times \Rm^{n+1}$
such that,
in the new coordinates, the Hamiltonian $H_{\e}=H_0+\e H_1$ takes the form
$$
N_{\e}=H_0+\e Z+\e R_2,
$$
with
\begin{itemize}
\item $\|R_2\|_{C^r}\le \delta$ on $\Tm^n\times \mD(K,\beta\e^{1/4})\times \Tm$,
\item  $\|\tilde \Phi-id\|_{C^0}\le \delta\sqrt{\e}$ and $\|\tilde \Phi-id\|_{C^2}\le \delta.$
\end{itemize}
The diffeomorphism $\tilde \Phi$ is of the forme
$$
\tilde \Phi(\theta,p,t,e)=(\Phi(\theta,p,t),e+f(\theta,p,t))
$$
where $\Phi$ is a diffeomorphism of $\Tm^n\times \Rm^n \times \Tm$ fixing the last variable $t$.
\end{cor}

\begin{proof}
We apply Theorem \ref{autonomous} with
$\tilde H_{\e}$, $m=n+1$ and $\tilde \delta=\delta/2$.
We get a diffeomorphism $\tilde \Phi$ of $\Tm^{n+1}\times \Rm^{n+1}$
as time-one flow of the Hamiltonian $G$.
By inspection in the proof of Theorem \ref{autonomous}, we observe that
$G$ does not depend on $e$, which implies that $\tilde \Phi$ has the desired form.
We have
$$\tilde H_{\e}\circ \tilde \Phi=\tilde H_0(J)+\e \tilde R_1+\e \tilde R_2
$$
where $\|\tilde R_2\|_{C^2}\leq \delta/2$ and
$$
\tilde R_1( \theta,p,t)=
\sum_{[k]\leq K}
\rho\left(\frac{k^f\cdot\partial_{p^f}H_0+k^s\partial_{p^s}H_0+k^t}{\beta\e^{1/4}[k]}\right)
g_k(p)
e^{2i\pi k\cdot(\theta,t)}.
$$
Let us  compute this sum under the assumption that $p\in \mD(K,\beta\e^{1/4})$.
We have
$$
\left|
\frac{k^f\cdot \partial_{p^f}H_0}{\beta\e^{1/4}[k]}
\right|\leq 1
$$
hence
$$
\rho\left(\frac{k^f\cdot\partial_{p^f}H_0+k^s\partial_{p^s}H_0+k^t}{\beta\e^{1/4}[k]}\right)
=1
$$
for  $k$ such that $k^s=0=k^t$.
For the other terms, we have, by definition of $\Omega(K,s)$,
$$
\left|
\frac{k^s\partial_{p^s}H_0+k^t}{\beta\e^{1/4}[k]}
\right|\geq
\left|
\frac{k^s\partial_{p^s}H_0+k^t}{\beta\e^{1/4}K}
\right|\geq 3,
$$
hence
$$
\left|
\frac{k^f\cdot\partial_{p^f}H_0+k^s\partial_{p^s}H_0+k^t}{\beta\e^{1/4}[k]}
\right|\geq 2
$$
and  these terms vanish in the expansion of $\tilde R_1$.
We conclude that
$$
\tilde R_1(\theta,p,t)=\sum _{k^f\in \Zm^{n-1},[k^f]\leq K} g_{(k_f,0,0)}(p)
e^{2i\pi k^f\cdot \theta^f}
$$
hence
$\tilde R_1=Z-\Pi_K^+(Z)$, with the notation of Lemma \ref{fourier-est}.
Finally
$
\tilde H_{\e}\circ \tilde \Phi=
\tilde H_0+\e Z+\e R_2
$
with $R_2=\tilde R_2-\Pi^+_KZ$.
From Lemma \ref{fourier-est}, we see that
$$
\|\Pi^+_KZ\|_{C^2}\leq c_nK^{m+3-r}\|Z\|_{C^r}\leq c_nK^{m+3-r}\|H_1\|_{C^r}
\leq c_nK^{m+3-r}\leq \delta/2.
$$
On the other hand, $\|\tilde R_2\|_{C^2}\leq \delta/2$, hence
$\|R_2\|_{C^2}\leq \delta$.
\end{proof}

\subsection{Smooth approximation}

Finally we remove the restriction on $r$ by the following smooth approximation lemma:
\begin{lem}\label{appoximation} \cite{SZ}
  Let $f:\R^n\to \R$ be a $C^{r}$ function, with $r\geq 4$.
Then for each $\tau>0$ there exists an analytic function  $S_\tau f$  such that
$$ \|S_\tau f -f\|_{C^3}< c(n,r) \|f\|_{C^3}\tau^{r-3},$$
$$ \|S_\tau f\|_{C^{r_1}}< c(n,r) \|f\|_{C^r_1} \tau^{-(r_1-r)},$$
for  each $r_1>r$, where $c(n,r)$ is a constant which depends only on $n$ and $r$.
\end{lem}

If $r<n+5$, we use Lemma~\ref{appoximation} to approximate $H_1$ by an
analytic  function $H_1^*$.
We can then apply Corollary~\ref{norm} to the Hamiltonian
$$
H^*_{\e}:=H_0+\e H^*_1=H_0+\e_2 H_2
$$
with $H_2=H_1^*/\|H_1^*\|_{C^{r_2}}$, with $\e_2=\e\|H_1^*\|_{C^{r_2}}$, and
with some  parameters $r_2\geq r$ and $\delta_2\leq \delta$
to be specified later.
We find a change of coordinates $\tilde \Phi$ such that
$$
\tilde H^*_{\e}\circ \tilde  \Phi=\tilde H_0+\e_2 Z_2+\e_2 R_2
$$
and $\|R_2\|_{C^2}\leq \delta_2$, where
$Z_2(\theta^s,p)=\int H_2d\theta^f dt$.
As usual, we have denoted by $\tilde H^*_{\e}$ and $\tilde H_0$ the automomized Hamiltonians
$\tilde H^*_{\e}=H^*_{\e}+e$ and $\tilde H_0=H_0+e$.
With the same map $\tilde \Phi$, we obtain
$$
\tilde H_{\e}\circ \tilde \Phi=\tilde H_0+\e Z + \e R
$$
with
$$
R=\|H_1^*\|_{C^{r_2}}R_2+(Z-Z^*)+(H_1^*-H_1)\circ  \Phi.
$$
In the expression above, the map $\Phi$ is the trace on the $(\theta,p,t)$ variables
of the map $\tilde \Phi$.
Choosing $\tau=\delta_2^{1/(r_2-3)}$,
we get
\begin{itemize}
 \item[-]$\|H_1^*-H_1\|_{C^3}\leq c(n,r_2)\delta_2^{\frac{r-3}{r_2-3}}$
\item[-] $\|H_1^*\|_{C^{r_2}}\leq c(n,r_2)\delta_2^{-\frac{r_2-r}{r_2-3}}$
\item[-] $\|Z^*-Z\|_{C^2}\leq \|H_1^*-H_1\|_{C^2}
\leq c(n,r_2)\delta_2^{\frac{r-3}{r_2-3}}$
\item[-] $\| \tilde \Phi\|_{C^2}\leq \delta_2\leq \delta \leq 1,$
\item[-]
$
\|(H_1^*-H_1)\circ  \Phi\|_{C^2}\le c_n\|H_1^*-H_1\|_{C^2 }
(\|\Phi\|_{C^2}+\| \Phi\|_{C^2}^2) \le c_n\|H_1^*-H_1\|_{C^2 } .
$
\end{itemize}
and finally
$$
\|R\|_{C^2}\leq c(n,r_2)\delta_2^{\frac{r-3}{r_2-3}}.
$$
We now set
$$\delta_2=\delta^{\frac{r_2-3}{r-3}}/c(n,r_2)\leq \delta
$$
and get $\|R\|_{C^2}\leq \delta$.
To apply Corollary~\ref{norm} as we just did, we need the following conditions
to hold on the parameters:
\begin{itemize}
\item[-] $K\geq  c(n,r_2)\delta^{\frac{r_2-3}{(r-3)(r_2-n-4)}}$, which implies
$K\geq c_n\delta_2^{\frac{-1}{r-n-4}}$,
\item[-] $\beta\geq
 c(n,r_2)(1+\|H_0\|_{C^4})\delta^{-\frac{r_2-3}{2(r-3)}}$
which implies $\beta\geq c_n (1+\|H_0\|_{C^4})\delta_2^{-1/2}$,
\item[-] $\beta \e^{1/4}\leq \|H_0\|_{C^4}\delta ^{\frac{r_2-r}{4(r-3)}}$
which implies $\beta \e_2^{1/4}\leq \|H_0\|_{C^4}$.
\end{itemize}
We apply the above discussion with $r_2=2n+5$ and get
Theorem \ref{normal-form}.
Note the estimate
$
\|id-\tilde \Phi\|_{C^0} \leq \delta_2\sqrt{\e_2}\leq
\delta_2^{1-\frac{r_2-r}{2(r_2-3)}}\sqrt{\e}\leq \sqrt{\e}.
$
\qed

\section{Normally hyperbolic cylinders}\label{sec:NHIC}
In this section, we study the Hamiltonian
in normal form
$$
N_\epsilon(\theta,p,t)=H_0(p) + \epsilon Z(\theta^s,p) +
\epsilon R(\theta,p,t).
$$
We
denote as above
by
$p^s_*(p^f)\in  \Rm^{n-1}$ the solution of the equation
$\partial_{p^s}H_0(p^s_*(p_f),p_f)=0$. We recall also the notation
$p_*(p^f):=(p^s_*(p_f),p^f)$
from Section \ref{notations}.
Fixing  parameters
$$
\lambda \in ]0,1], \quad a^-<a^+,
$$
we
 assume that there exists, for each
$p^f\in [a_--\lambda,a_++\lambda]$, a
local minimum $\theta^s_*(p^f)$ of the map
$\theta^s \lmto Z(\theta^s,p_*(p_f))$, and that $\theta^s_*$
is a $C^2$ function of $p^f$.
We assume in addition that
$$
\lambda I \le \partial^2_{\theta^s \theta^s}
Z(\theta^s_*(p_f), p_*(p^f))\le  I
$$
for each $p_f\in [a_--\lambda,a_++\lambda]$,
where as before $I$ is the identity matrix. We shall at some occasions lift the map
$\theta^s_*$ to a $C^2$ map
taking values in $\Rm^{n-1}$ without changing its name.
We  assume that $\|Z\|_{C^3}\leq 1$,
and set $\|R\|_{C^2}=\delta$.
Finally, we assume that
$D^{-1}I\leq \partial^2_{pp}H_0\leq D\,I$
 for some $D\geq 1$.
To simplify notations, we will be using the $O(\cdot)$ notation, where $f=O(g)$ means $|f|\le Cg$ for a constant $C$ independent of
$\e$, $\lambda$, $\delta$, $n$ and $r$. In particular, we will not be keeping track of the parameter $D$, which is considered fixed throughout the paper.

\begin{thm}\label{nhic-mult}
There exists $\epsilon_0\in ]0,1[$ such that, if
$$
0<\epsilon<\epsilon_0\lambda^{7/2}\quad ,
\quad 0\leq \delta<\sqrt \e_0\lambda^{2},
$$
 then there exists  a $C^1$ map
 $$
(\Theta^s, P^s)(\theta^f, p^f, t):\T\times [a_- - \lambda/2,a_+ + \lambda/2]\times\Tm\lto \T^{n-1}\times \Rm^{n-1}$$
such that the cylinder
$$\mC=\{ (\theta^s, p^s)=(\Theta^s_j,P^s_j)(\theta^f, p^f, t));
\quad p^f\in [a_- - \lambda/2, a_{+} + \lambda/2], (\theta^f,t)\in \T\times \T\}$$
is weakly invariant with respect to  $N_\epsilon$ in the sense that the Hamiltonian vector field is tangent to $\mC$.
The cylinder $\mC$ is contained in  the set
\begin{align*}
V:=\big\{&(\theta,p,t); p^f\in [a_- - \lambda/2, a_{+} + \lambda/2], \\
&\|(\theta^s-\theta^s_*(p^f)\|\leq O\big(\epsilon_0^{1/4}\lambda\big),
\quad
\|p^s-p^s_*(p^f)\|\le O\big(\e_0^{1/4}\lambda ^{5/4}\e^{1/2}\big)
\big\},
\end{align*}
and it contains all the full orbits  of
$N_{\e}$ contained in $V$.
We have the estimates
$$
\|\Theta^s(\theta^f,p^f,t)-\theta^s_*(p^f)\|\leq
O\big(\lambda^{-1}\delta+\lambda^{-3/4}\sqrt{\e}\big),
$$
$$
\|P^s(\theta^f,p^f,t)-p^s_*(p^f)\|\leq
\sqrt{\e}\,O\big(\lambda^{-3/4}\delta+\lambda^{-1/2}\sqrt{\e}\big),
$$
$$ \left\|\frac{\partial\Theta^s}{\partial p^f}\right\|=
O\left(\frac{\lambda^{-2}\sqrt{\e}+\lambda^{-5/4}
\sqrt \delta}{\sqrt{\e}}
\right)
\quad,\quad
\quad \left\|\frac{\partial \Theta^s}{\partial(\theta^f, t)}\right\|=
O\left(\lambda^{-2}\sqrt{\e}+\lambda^{-5/4}\sqrt\delta
\right).
 $$

%
\end{thm}

The proof of Theorem \ref{nhic-mult} occupies the rest of the section.
The Hamiltonian flow admits the following equation of motion :
\begin{equation}\label{eq:perturbed}
\begin{cases}
  \dot{\theta}^s = \partial_{p^s}H_0 + \epsilon\partial_{p^s}Z + \epsilon\partial_{p^s} R \\
  \dot{p}^s = -\epsilon \partial_{\theta^s}Z - \epsilon\partial_{\theta^s} R \\
  \dot{\theta}^f = \partial_{p^f}H_0 + \epsilon\partial_{p^f}Z +  \epsilon\partial_{p^f} R \\
  \dot{p}^f = -\epsilon \partial_{\theta^f} R \\
  \dot{t}=1
\end{cases}.
\end{equation}
It is convenient in the sequel to lift the angular variables
to real variables and to consider the above system as defined on
$\Rm^{n-1}\times \Rm^{n-1}\times \Rm\times \Rm\times \Rm.
$
We will see this system as a perturbation of the model system
\begin{equation}\label{eq:model}
 \dot{\theta}^s = \partial_{p^s}H_0\quad,\quad
 \dot{p}^s = -\epsilon \partial_{\theta^s}Z \quad, \quad
\dot{\theta}^f = \partial_{p^f}H_0\quad,
\dot{p}^f = 0\quad, \quad
\dot{t}=1.
\end{equation}
The graph
of the map
$$
(\theta^f,p^f,t)\lmto (\theta^s_*(p_f),p^s_*(p_f))
$$
on $\Rm\times J\times \Rm$ is obviously invariant
for the model flow.
For each fixed $p_f$, the point $(\theta^s_*(p_f),p^s_*(p_f))$
is a hyperbolic fixed point of the partial system
$$ \dot{\theta}^s = \partial_{p^s}H_0(p^s,p^f)\quad,\quad
 \dot{p}^s = -\epsilon \partial_{\theta^s}Z(\theta^s,p^s,p^f)
$$
where $p^f$ is seen as a parameter.
This hyperbolicity is the key property we will
use, through the theory of normally hyperbolic invariant manifolds.
It is  not obvious to  apply this theory here  because the model
system itself depends on $\epsilon$, and because we have to deal
with the problem of
non-invariant boundaries.
We will however manage to apply the quantitative version
exposed  in Appendix
\ref{sec:abstract-nhic}.

We perform some  changes of coordinates in order
to put the system in the
framework of Appendix
\ref{sec:abstract-nhic}. These coordinates appear naturally from
the study of the model system as follows.
We set
$$
B(p^f):=\partial^2_{p^sp^s}H_0(p_*(p^f))\quad,\quad
A(p_f):=- \partial^2_{\theta^s\theta^s}Z(\theta^s_*(p^f),p_*(p^f)).
$$
If we fix the variable $p^f$ and consider
the model system in $(\theta^s,p^s)$, we observed that this system
has a  hyperbolic fixed point at $(\theta^s_*(p^f),p^s_*(p^f))$.
The linearized system at this point is
$$
\dot \theta^s=B(p^f)\,p^s
\quad,\quad
\dot p^s=\e A(p_f)\,\theta^s.
$$
To put this system under a simpler form, it is useful to consider the matrix
$$
L(p^f):=\big(B^{1/2}(p^f)(B^{1/2}(p^f)A(p^f)B^{1/2}(p^f))^{-1/2}
B^{1/2}(p^f)
\big)^{1/2}
$$
which is symetric, positive definite, and satisfies
$L^2(p^f)A(p^f)L^2(p^f)=B(p^f)$, as can be checked by a direct computation.
We finally introduce the symmetric positive definite matrix
$$
\Lambda(p^f):= L(p^f)A(p^f)L(p^f)=L^{-1}(p^f)B(p^f)L^{-1}(p^f).
$$
In  the new variables
$$
\xi=L^{-1}(p^f) \theta^s+\e^{-1/2}L(p^f)p^s
\quad,\quad
\eta=L^{-1}(p^f) \theta^s-\e^{-1/2}L(p^f)p^s,
$$
the linearized system is reduced to the following block-diagonal form:
$$
\dot \xi=\e^{1/2}\Lambda(p^f)\xi
\quad,\quad
\dot \eta=-\e^{1/2}\Lambda(p^f)\eta,
$$
see
\cite{Be3} for more details.
This motivates us to introduce the following set of new coordinates for our full system:

$$
x=L^{-1}(p^f) (\theta^s-\theta^s_*(p^f))+
\e^{-1/2}L(p^f)(p^s-p^s_*(p^f))
$$
$$
y=L^{-1}(p^f) (\theta^s-\theta^s_*(p^f))-
\e^{-1/2}L(p^f)(p^s-p^s_*(p^f)),
$$
$$
I=\e^{-1/2}p^f\quad,\quad \Theta=\gamma \theta^f,
$$
where $\gamma$ is a parameter which will be taken  later equal to $\delta^{1/2}$.
Note  that
$$
\theta^s=\theta^s_*(\e^{1/2}I)+\frac{1}{2}L(\e^{1/2}I)(x+y),\quad
p^s=p^s_*(\e^{1/2}I)+\frac{\e^{1/2}}{2}L^{-1}(\e^{1/2}I)(x-y).
$$
\begin{lem}\label{lambda}
We have $\Lambda (p^f)\geq \sqrt{\lambda/D}\ I$
for each $p^f\in [a_-,a_+]$.
\end{lem}
\proof
The matrix $\Lambda$ is symmetric, hence it satisfies
$\Lambda\geq \lambda_*I$, where $\lambda_*>0$ is its smallest eigenvalue.
The real number $\lambda_*$ is then an eigenvalue of the matrix
$
\begin{bmatrix}\Lambda& 0\\0&-\Lambda\end{bmatrix}
$
which is similar  to
$
\begin{bmatrix}0&B\\A&0\end{bmatrix}.
$
Since both $A$ and $B$ are square matrices of equal size,
we conclude that $\lambda_*^{-2}$ is an eigenvalue of $A^{-1}B^{-1}$.
Since $\|A^{-1}\|\leq \lambda^{-1}$ and $\|B^{-1}\|\leq D$, we
have $\lambda_*^{-2}\leq \|A^{-1}B^{-1}\|\leq D\/\lambda^{-1}$.
We conclude that $\lambda_*\geq \sqrt{\lambda/D}$.
\qed

The links between the various parameters
$\e$, $\delta$, $\gamma$,  $\lambda$, $\rho$ which appear in
the computations below will be specified later.
We will however assume from the beginning that
$$
\delta\leq \rho\leq  \lambda
\quad , \quad
 \sqrt{\e}\leq \rho
\quad,\quad
0< \gamma\leq \lambda.
$$


Let us first collect some estimates that will be
useful to see that the  system (\ref{eq:perturbed})
is indeed a perturbation of the model system.

\begin{lem}\label{all-estimates}
On the domain $\|x\|\leq \rho, \|y\|\leq \rho$, we have the estimates
$$
\|L\|=O(\lambda^{-1/4}),\  \|L^{-1}\|=O(1),\
\|\partial_{p^f} L\|\leq O(\lambda ^{-3/2}),\ \ \ \
\|\partial_{p^f} L^{-3/2}\|\leq O(\lambda ^{-3/4})
$$
$$
\|\partial_{
p^f}\theta^s_*\|\leq O( \lambda^{-1}),\ \
\|p^s_*\|_{C^2}=O(1),\
\|\theta^s-\theta^s_*\|\leq O(\lambda^{-1/4}\rho), \
\|p^s-p^s_*\|\leq O(\e^{1/2}\rho).
$$
\end{lem}

\proof
We recall that
$
L=\big(B^{1/2}(B^{1/2}AB^{1/2})^{-1/2}B^{1/2}\big)^{1/2}.
$
Since $D^{-1}I\leq B\leq D\,I$ and $\lambda I\leq A\leq I$,
 we obtain that $\|L\|\leq O(\lambda^{-1/4})$
and that $\|L^{-1}\|\leq O(1)$,
using the expression
$L^{-1}=\big(B^{-1/2}(B^{1/2}AB^{1/2})^{1/2}B^{-1/2}\big)^{1/2}$,
we obtain that $\|L\|\leq O(\lambda^{-1/4})$
and that $\|L^{-1}\|\leq O(1)$,
To estimate the derivative of $L$, we consider the
map
$F:M\lmto M^{1/2}$ defined on positive symmetric matrices.
It is known
that

$$
dF_M\cdot N=\int_0^{\infty}e^{-tM^{1/2}}Ne^{-tM^{1/2}}dt.
$$

To verify this one can diagonalize $M$, perform integration,
and match terms in $(M^{1/2}+\eps dF_M\cdot N)
(M^{1/2}+\eps dF_M\cdot N)=M+\eps N + O(\eps^2)$.
This implies that
$$\|dF_M\|\leq \|M^{1/2}\|^{-1}/2\leq \|M^{-1/2}\|/2
$$
Now we apply this bound several times
to estimate $\partial_{p^f}L$ and $\partial_{p^f}L^{-1}$.
In our situation, we have $\partial A=O(1)$, $\partial B =O(1)$.
Using $M= A$ and $B$, we get $\partial (A^{1/2})=O(\lambda^{-1/2})$
and $\partial (B^{1/2})=O(1)$ resp.
Using $M= B^{1/2}A B^{1/2} $
we get
$\partial(B^{1/2}AB^{1/2})^{1/2}=O(\lambda^{-1/2})$.
We now recall that the differential at $M$ of the map
$M\lmto M^{-1}$ is the linear map
$N\lmto -M^{-1}NM^{-1}$, whose norm is bounded by
$\|M^{-1}\|^2$. At $M=(B^{1/2}AB^{1/2})^{1/2}$,
we obtain
$$ \|\partial(B^{1/2}AB^{1/2})^{-1/2}\|\leq
\|M^{-1}\|^2\|\partial M\|=
O(\lambda^{-3/2}).
$$
Using $M= B^{1/2}(B^{1/2}A B^{1/2})^{-1/2}B^{1/2}$
we get
$\partial L=\partial M^{1/2}=O(\lambda^{-1})$ and
using
$
M=B^{-1/2}(B^{1/2}A B^{1/2})^{1/2}B^{-1/2}
$
 we get
$
\partial L^{-1}=\partial M^{1/2}=O(\lambda^{-3/4})
$.
The other estimates are straightforward.
\qed

\begin{lem}\label{uniform-estimates}
The equations of motion in the new coordinates
take the form
\begin{align*}
\dot x &=-\sqrt{\e}\Lambda(\sqrt{\e}I)x
+\e^{1/2}O(\lambda^{-1/4}\delta+\lambda^{-3/4}\rho^2)
+O(\e)
\\
\dot y  &=\sqrt{\e}\Lambda(\sqrt{\e}I)y
+\e^{1/2}O(\lambda^{-1/4}\delta+\lambda^{-3/4}\rho^2)
+O(\e)
\\
\dot I &= O(\sqrt{\e}\delta),
\end{align*}
where $\rho=\max(\|x\|,\|y\|)$ is assumed to satisfy $\rho\leq \lambda $.
The expression for $\dot \Theta$ is not useful here.
\end{lem}
\proof
The last part of the statement is obvious.
We prove the part concerning $\dot x$, the calculations
for $\dot y$ are exactly the same.
In the original coordinates
the vector field (\ref{eq:perturbed}) can be written
$$
\dot \theta^s=B(p^f)(p^s-p^s_*(p^f))+O(\|p^s-p^s_*(p^f)\|^2)+
O(\e),
$$
$$
\dot p^s=\e A(p^f)(\theta^s-\theta^s_*(p^f))+
O(\e\|\theta^s-\theta^s_*(p^f)\|^2)+O(\e \delta).
$$
As a consequence, we have
\begin{align*}
\dot x&=
L^{-1}B(p^s-p^s_*)+\e^{1/2}LA(\theta^s-\theta^s_*)\\
&+L^{-1} \cdot O(\|p^s-p^s_*\|^2+\e)
+\e^{1/2}L \cdot O(\|\theta^s-\theta^s_*\|^2+\delta)\\
&+(\partial_{p^f} L^{-1}) \,\dot p^f(\theta^s-\theta^s_*)
+\e^{-1/2}(\partial_{p^f} L)\, \dot p^f (p^s-p^s_*)\\
&-L^{-1}(\partial_{p^f} \theta^s_* )\,\dot p^f
-\e^{-1/2}L(\partial_{p^f}p^s_*)\,\dot p^f.
\end{align*}
We use the estimates of Lemma \ref{all-estimates}
to simplify (recall also that $\dot p^f=O(\epsilon \delta)$):
\begin{align*}
\dot x&=
L^{-1}B(p^s-p^s_*)+\e^{1/2}LA(\theta^s-\theta^s_*)\\
&+O(\e\rho +\e) +
O(\e^{1/2} \lambda^{-3/4}\rho^2+\e^{1/2} \lambda^{-1/4}\delta)\\
&+O(\lambda^{-1}\e \delta \rho^2)+O(\lambda^{-5/4}\e \delta \rho)
+O(\lambda^{-1}\e \delta+\lambda^{-1/4}\e^{1/2}\delta).
\end{align*}
\qed

\begin{lem}\label{linearized-estimates}
In the new coordinate system $(x,y,\Theta,I,t)$, the linearized
system is given  by the matrix
\begin{align*}
L=\begin{bmatrix}
\sqrt{\e}\Lambda &
0&
0 & 0  &0\\
0&-\sqrt{\e}\Lambda &
0 &
  0&0\\
0 &0 & 0 &
 0  & 0\\
0 & 0 & 0 & 0 & 0\\
0 & 0 & 0 & 0 & 0
\end{bmatrix}+O(\sqrt{\e}\delta \lambda^{-1/4}\gamma^{-1}+\sqrt{\e}\lambda^{-3/4}\rho+
\e \lambda^{-5/4}+\sqrt{\e}\gamma),
\end{align*}
where $\rho=\max (\|x\|,\|y\|)$.
%
%
%
%
\end{lem}

\proof
Most of the estimates below are based on Lemma \ref{all-estimates}.
In the original coordinates, the matrix of the linearized system
is:
$$
\tilde L=\begin{bmatrix}
  O(\e) & \partial_{p^sp^s}^2H_0+O(\e) & 0 & \partial_{p^fp^s}^2H_0   +O(\e)& 0 \\
-\e \partial^2_{\theta^s\theta^s}Z &O(\e) & 0 & O(\e)&0\\
O(\e) &O(1) & 0 & O(1) & 0\\
0 & 0 & 0 & 0 & 0\\
0 & 0 & 0 & 0 & 0
\end{bmatrix}+O(\delta\e),
$$
%
In our notations we have
$$
\tilde L=\begin{bmatrix}
  O(\e) & B+O(\e+\sqrt \e \rho) & 0 & \partial_{p^fp^s}^2H_0      +O(\e)& 0 \\
-\e A+O(\e  \lb^{-1/4} \rho) &O(\e) & 0 & O(\e)&0\\
O(\e) &O(1) & 0 & O(1) & 0\\
0 & 0 & 0 & 0 & 0\\
0 & 0 & 0 & 0 & 0
\end{bmatrix}+O(\delta\e),
$$
In the new coordinates, the matrix is the product
$$
L=\left[\frac{\partial (x,y,\Theta, I, t)}{\partial (\theta^s,p^s,\theta^f,p^f,t)}
\right]\cdot
\tilde L\cdot
\left[\frac{\partial (\theta^s,p^s,\theta^f,p^f,t)}{\partial (x,y,\Theta, I, t)}\right].
$$
We have
$$
\left[\frac{\partial (\theta^s,p^s,\theta^f,p^f,t)}
{\partial (x,y,\Theta, I, t)}\right]=
\begin{bmatrix}
L/2 &L/2   & 0 &O(\sqrt{\e}\lambda^{-1})& 0 \\
\sqrt{\e}L^{-1}/2 &-\sqrt{\e}L^{-1}/2 & 0 & \sqrt{\e}\partial_{p^f}p^s_*+O(\e\lambda^{-3/4}\rho)&0\\
0 &0  & \gamma^{-1} & 0 & 0\\
0 & 0 & 0 & \sqrt{\e} & 0\\
0 & 0 & 0 & 0 & 1
\end{bmatrix}
$$

hence
\begin{align*}
&\tilde L\left[\frac{\partial (\theta^s,p^s,\theta^f,p^f,t)}
{\partial (x,y,\Theta, I, t)}\right]=O(\gamma^{-1}\delta\e)+\\
&\begin{bmatrix}
\sqrt{\e}BL^{-1}/2+O(\e\lambda^{-1/4}) & -\sqrt{\e}BL^{-1}/2+O(\e\lambda^{-1/4})& 0&
O(\e\lambda^{-3/4}\rho+\e^{3/2}\lambda^{-1})  &0\\
\e AL/2+O(\e\lambda^{-1/2}\rho) &\e AL/2+O(\e\lambda^{-1/2}\rho) & 0 &
 \e^{3/2}O( \lambda^{-5/4}\rho+\lambda^{-1})&0\\
O(\sqrt{\e}) &O(\sqrt{\e})  & 0 &
 O(\sqrt{\e})  & 0\\
0 & 0 & 0 & 0 & 0\\
0 & 0 & 0 & 0 & 0
\end{bmatrix}.
\end{align*}
%
%
%
%
This expression is the result of a tedious, but obvious, computation.
Let us just detail the computation of  the coefficient
on the first line, fourth row, which contains
an important cancellation:
\begin{align*}
&\sqrt{\e}\partial_{p^sp^s}^2H_0\partial_{p^f}p^s_*+
\sqrt{\e}\partial_{p^fp^s}^2H_0
+O(\e\lambda^{-3/4}\rho+\e^{3/2}\lambda^{-1})\\
=&
\sqrt{\e}\partial_{p^f}\big(\partial_{p^s}H_0(p_*(p^f)\big)
+O(\e\lambda^{-3/4}\rho+\e^{3/2}\lambda^{-1})
=O(\e\lambda^{-3/4}\rho+\e^{3/2}\lambda^{-1}).
\end{align*}
We now write
$$
\left[\frac{\partial (x,y,\Theta, I, t)}{\partial (\theta^s,p^s,\theta^f,p^f,t)}
\right]
=
\begin{bmatrix}
L^{-1} & \e^{-1/2}L   & 0 &O(\e^{-1/2}\lambda^{-1/4})& 0 \\
L^{-1} & -\e^{-1/2}L  &0& O(\e^{-1/2}\lambda^{-1/4})&0\\
0 &0  & \gamma & 0 & 0\\
0 & 0 & 0 & \e^{-1/2}& 0\\
0 & 0 & 0 & 0 & 1
\end{bmatrix},
$$
and compute  that
\begin{align*}
L=&\begin{bmatrix}
\sqrt{\e}\Lambda+O(\sqrt{\e}\lambda^{-3/4}\rho) &
O(\sqrt{\e}\lambda^{-3/4}\rho)&
0 &  O(\e \lambda^{-5/4})  &0\\
O(\e\lambda^{-3/4}\rho)&-\sqrt{\e}\Lambda+ O(\sqrt{\e}\lambda^{-3/4}\rho) &
0 & O( \e\lambda^{-5/4})
  &0\\
 O(\sqrt{\e}\gamma) &O(\sqrt{\e}\gamma)  & 0 &
  O(\sqrt{\e}\gamma)  & 0\\
0 & 0 & 0 & 0 & 0\\
0 & 0 & 0 & 0 & 0
\end{bmatrix}\\
+&O(\sqrt{\e}\delta \lambda^{-1/4}\gamma^{-1}).
\end{align*}
\qed

In order to prove the existence of a normally hyperbolic
invariant strip (for the lifted system), we apply  Proposition \ref{realNHI}
to the system in  coordinates $(x,y,\Theta,I,t)$.
More precisely, with the notations of appendix \ref{sec:abstract-nhic},
we set:
$$u=x, s=y, c_1=(\Theta,t), c_2=I,
 \Omega=\Rm^2\times \Omega^{c_2}=\Rm^2\times \left[\frac{a_- - \lambda/2}{\sqrt{\e}},\frac{a_+ + \lambda/2}{\sqrt{\e}}\right].
$$
We fix $\gamma=\sqrt{\delta}$ and  $\alpha =\sqrt{\e\lambda/4D}$,
recall that $\sqrt{\e}\Lambda\geq 2\alpha I$,
by Lemma \ref{lambda}.
We take  $\sigma=\lb\e^{-1/2}/2$, so that
$$
\Omega_{\sigma}=\Rm^2\times
\left[\frac{a_-  - \lb}{\sqrt{\e}},\frac{a_+ + \lb}
{\sqrt{\e}}\right].
$$
We assume, as in the statement of the Theorem,
that $0<\e<\e_0\lambda^{7/2}$ and that
$0\leq \delta<\sqrt \e_0 \lambda ^2$.
We can apply Proposition \ref{realNHI}
with $B^u=\{u: \|u\|\leq \rho\}$ and $B^s=\{s: \|s\|\leq \rho\}$
provided
$$
\e_0^{-1/4}(\lambda^{-3/4}\delta+\lambda^{-1/2}\sqrt{\e})\leq \rho\leq 2\epsilon_0^{1/4}\lambda^{5/4}.
$$
It is easy to check under our assumptions
on the parameters that such values of $\rho$ exist.
%
These estimates along with Lemma \ref{all-estimates} imply that
\[
\|(\theta^s-\theta^s_*(p^f)\|\leq O\big(\epsilon_0^{1/4}\lambda\big)
\qquad,
\qquad
\|p^s-p^s_*(p^f)\|\le O\big(\e_0^{1/4}\lambda ^{5/4}\e^{1/2}\big).
\]
Provided that the cylinder $\cC$ exists, this gives the first set of estimates in Theorem
\ref{nhic-mult}.

Let us check the isolating block condition.
By Lemma \ref{uniform-estimates}, we have
$$
\dot x\cdot x\geq 2\alpha\|x\|^2-
\|x\|\ O( \e^{1/2} \lambda^{-1/4}\delta+\e^{1/2} \lambda^{-3/4}\rho^2
+\e)
$$
if $x\in B^u,y\in B^s$.
If in addition $\|x\|=\rho$, then
$$
\lambda^{-3/4}\delta\leq \e_0^{1/4}\|x\|\quad,\quad
\lambda^{-3/4}\rho^2\leq 2\e_0^{1/4}\|x\|\quad,\quad
\sqrt{\e/\lambda}\leq \e_0^{1/4}\|x\|,
 $$
hence
$$
\dot x\cdot x\geq 2\alpha\|x\|^2-
\|x\|^2\epsilon_0^{1/4}O(\sqrt{\e\lambda})\geq \alpha \|x\|^2
$$
provided  $\epsilon_0$ is small enough.
Similarly,
$\dot y \cdot y\leq -\alpha \|y\|^2$
on $B^u\times \partial B^s$ provided $\e_0$ is small enough.
Concerning the linearized system,
we have
\begin{align*}
L_{uu}&=\sqrt{\e}\Lambda+
O(\sqrt{\e}\delta \lambda^{-1/4}\gamma^{-1}+\sqrt{\e}\lambda^{-3/4}\rho+
\e \lambda^{-5/4}+\sqrt{\e}\gamma)\\
&=\sqrt{\e}\Lambda+ O(\e_0^{1/4}\sqrt{\e\lambda})\geq \alpha I,\\
L_{ss}&=-\sqrt{\e}\Lambda+ O(\e_0^{1/4}\sqrt{\e\lambda})\leq -\alpha I
\end{align*}
on $B^u\times B^s\times \Omega_r$.
These  inequalities holds when $\e_0$ is small enough
because $\sqrt{\e} \Lambda\geq 2\alpha I$ and
$\sqrt{\e\lambda}\leq O(\alpha)$.
Finally, still with the notations of Proposition
\ref{realNHI},  we  take
\begin{align*}
m&=O(\sqrt{\e}\delta \lambda^{-1/4}\gamma^{-1}+\sqrt{\e}\lambda^{-3/4}\rho+
\e \lambda^{-5/4}+\sqrt{\e}\gamma
+\sqrt \e \delta/\sigma
)\\
&=\sqrt{\e\lambda}\,O(\sqrt\delta \lambda^{-3/4}+\rho \lambda^{-5/4}+
\sqrt{\e}\lambda^{-7/4})
=\sqrt{\e\lambda}\,O(\e_0^{1/4}).
\end{align*}
If $\e_0$ is small enough, we have $4m<\alpha$ hence
$$
K\leq 2m/\alpha\leq O(\e_0^{1/4})<2^{-1/2},
$$
 and Proposition
\ref{realNHI} applies.
The invariant strip obtained from the proof of Proposition  \ref{realNHI}
does not depend on the choice of $\rho$.
It contains all the full orbits contained in
$$
\{x:\|x\|\leq \e_0^{1/4}\lambda^{-5/4}\}\times
\{y:\|y\|\leq \e_0^{1/4}\lambda^{-5/4}\}\times \Rm\times
\left[\frac{a_- - \lambda/2}{\sqrt{\e}},\frac{a_+ + \lambda/2}{\sqrt{\e}}\right]\times \Rm,
$$
hence all the full orbits contained in $V$, as defined
 in the statement of Theorem \ref{nhic-mult}.
The possibility of taking
$\rho=\e_0^{-1/4}(\lambda^{-3/4}\delta+\lambda^{-1/2}\sqrt{\e})$
now implies that the cylinder is actually contained in the domain
where
$$
\|x\|, \|y\|\leq \e_0^{-1/4}(\lambda^{-3/4}\delta+\lambda^{-1/2}\sqrt{\e}).
$$
%
Moreover, with this choice of $\rho$
and using that  $K=O( m/\sqrt{\e\lambda})$, we can obtain
an improved estimate of the Lipschitz constant $K$:
\begin{align*}
K  &=O\big(\sqrt\delta \lambda^{-3/4}+\rho \lambda^{-5/4}+
\sqrt{\e}\lambda^{-7/4}\big)\\
&  =O\big(
\sqrt\delta \lambda^{-3/4}+
\e_0^{-1/4}\delta \lambda^{-2}+
\e_0^{-1/4}\sqrt{\e}\lambda^{-7/4}+
\sqrt{\e}\lambda^{-7/4}
\big)\\
&  =O\big(
\sqrt\delta \lambda^{-3/4}+
\sqrt\delta \lambda^{-1}+
\e_0^{-1/4}\sqrt{\e}\lambda^{-7/4}
\big)\\
&  =O\big(
\sqrt\delta \lambda^{-1}+
\e_0^{-1/4}\sqrt{\e}\lambda^{-7/4}
\big).
\end{align*}
%
%
Observe finally that, since the system is $1/\gamma$-periodic in
$\Theta$ and $1$-periodic in $t$, so is the invariant strip that
we obtain, as follows from Proposition \ref{translation}.
We have obtained the existence of a $C^1$ map
$$
w^c=(w^c_u,w^c_s):(\Theta,I,t)\in\Rm\times
\left[\frac{a_- - \lambda/2}{\sqrt{\e}},
\frac{a_+ + \lambda/2}{\sqrt{\e}}\right]\times \Rm
\lto \Rm^{n-1}\times \Rm^{n-1}
$$
which is $2K$-Lipschitz, $1/\gamma$-periodic in $\Theta$ and
$1$-periodic in $t$, and the graph of which is weakly invariant.

Our last task is to return to the original coordinates
by setting
\begin{align*}
\Theta^s(\theta^f,p^f,t)&=\theta^s_*(p^f)+\frac{1}{2}L(p^f)
\cdot(w^c_u+w^c_s)(\gamma \theta^f,\e^{-1/2}p^f,t)\\
P^s(\theta^f,p^f,t)&=p^s_*(p^f)+\frac{\sqrt{\e}}{2}L^{-1}(p^f)\cdot
(w^c_u-w^c_s)(\gamma \theta^f,\e^{-1/2}p^f,t).
\end{align*}
All the estimates stated in Theorem \ref{nhic-mult} follow directly
from these expressions, and from the fact that $\|dw^c\|\leq 2K$.
This concludes the proof of Theorem~\ref{nhic-mult}. \qed

\section{Localization and Mather's projected graph theorem}\label{sec:localization}

We now study the system in normal form $N_\epsilon=H_0 +\epsilon Z + \epsilon R$
from the point of view of Mather theory.
We study the normal form system $N_\epsilon=H_0 +\epsilon Z + \epsilon R$ on the neighborhood of the set $\{p=p_*(p^f), p^f\in [a_-, a_+]\subset [a_{min}, a_{max}]\}$. We assume that $Z$ satisfies the generic conditions [G0]-[G2] and that $\|R\|_{C^2}\le \delta$. Recall that there exists a partition of $[a_{min}, a_{max}]=\bigcup_{j=1}^{s-1}[a_j, a_{j+1}]$, such that for $p^f\in [a_j-\lambda, a_{j+1}+\lambda]$ the function $Z(\theta^s, p^s, p^f)$ as a nondegenerate local maximum at $\theta^s_j$. It is clear that we can restrict this partition to $[a_-, a_+]$. We abuse notation and still write $[a_-, a_+]=\bigcup_{j=1}^{s-1}[a_j, a_{j+1}]$.

We first point out the following consequences of the genericity conditions [G0]-[G2]: there exists $0<b<\lambda/4$ depending on $H_1$ such that
\begin{enumerate}
\item[{[G1']}] $$Z(\theta^s_j(p^f),p_*(p^f))-Z(\theta^s, p_*(p^f)) \ge b \|\theta^s-\theta_j^s(p^f)\|,$$
 for each $p^f\in [a_j+b, a_{j+1}-b]$.
\item[{[G2']}] For $p^f\in [a_{j+1}-b, a_{j+1}+b]$, $j=0,\cdots, s-2$, we have
\begin{multline*} \max\{Z(\theta^s_j,p_*(p^f)), Z(\theta^s_{j+1},p_*(p^f))\}-Z(\theta^s,p_*(p^f))\\ \ge b \min\{\|\theta^s-\theta^s_j\|, \|\theta^s-\theta^s_{j+1}\|\}^2.
\end{multline*}
\end{enumerate}

In the first case, the function $Z$ has a single non-degenerate maximum, which we will call the ``single peak'' case, while the second case will be called the ``double peak'' case. The shape of the function $Z$ allows us to localize the Aubry set and Ma\~{n}e set of the Hamiltonian $N_\epsilon$.

According to Theorem~\ref{nhic-mult}, for each $[a_j-\lambda/2, a_{j+1}+\lambda/2]$ there exists
$$X_j=\{ (\theta^s, p^s)=(\Theta^s_j,P^s_j)(\theta^f, p^f, t)); \quad p^f\in [a_j-\frac{\lambda}2, a_{j+1}+\frac{\lambda}2], (\theta^f,t)\in \T\times \T\},$$
which are maximally invariant set on
$N_j:=\{(\theta,p,t); p^f\in [a_j-\frac{\lambda}2, a_{j+1}+\frac{\lambda}2], \|(\theta^s,p^s)-(\theta^s_j,p^s_*)\|\le \rho_1\}$.

These information allows us to study the  Mather set, Aubry set and Ma\~{n}e set of the Hamiltonian $N_\epsilon$.

\begin{thm}[Localization]\label{localize}
  For $N_\epsilon=H_0+ \epsilon Z + \epsilon R$ such that $Z$ satisfies [G0]-[G2], then there exists $\epsilon_0$, $\delta_0$ and $0<\rho_2<\rho_1$ such that for $0<\epsilon<\epsilon_0$ and $0<\delta<\delta_0$ the following hold.
  \begin{enumerate}
  \item For any $c=(p^s_*(c^f), c^f)$ such that $c^f\in [a_j+b, a_{j+1}-b]$, $\tilde{\mathcal N}(c)$ is contained in
$$ \{(\theta, p, t), \|p-c\|\le 6A \sqrt{n\epsilon}, \|\theta^s-\theta^s_j(p^f)\|\le \rho_2\}. $$
  \item For $c=(p^s_*(c^f), c^f)$ such that $c^f\in [a_{j+1}-b, a_{j+1}+b]$, we have that $\tilde{{\mathcal A}}_{N_\epsilon}(c)$ is contained in
$$ \{(\theta, p, t), \|p-c\|\le 6A \sqrt{n\epsilon}, \min\{\|\theta^s-\theta^s_j(p^f)\|, \|\theta^s-\theta^s_{j+1}(p^f)\|\}\le \rho_2\}. $$
  \end{enumerate}
\end{thm}

Apply the statements of the previous theorem with Theorem~\ref{nhic-mult}, we may further localize these sets on the normally hyperbolic cylinders. Moreover, locally these sets are graphs over the $\theta^f$ component, which is a version of Mather's projected graph theorem.

\begin{thm}[Mather's projected graph theorem]\label{graph}
For any $N_\epsilon$ such that $Z$ satisfies [G0]-[G2], there exists $\delta_0$ and $\epsilon_0$ depending on  $b$, $n$ and $r$ such that for $\delta\le\delta_0$ and $\epsilon<\epsilon_0$ we have:
  \begin{enumerate}
  \item There exists  $0<\rho_2<\rho_1$ such that for for $c=(p^s_*(c^f), c^f)$ with $c^f\in (a_j+b, a_{j+1}-b)$ the Ma\~{n}e set  $\tilde{\mathcal N}_c$ is contained in the normally hyperbolic cylinder $X_j$.

Moreover, let $\pi_{\theta^f}$ be the projection to the $\theta^f$ component, we have that  $\pi_{\theta^f}|\tilde{\mathcal A}_c$ is one-to-one and the inverse is Lipshitz.
  \item For $c^f\in [a_{j+1}-b, a_{j+1}+b]$, we have that ${\mathcal A}_c \subset X_j\cup X_{j+1}$.

$\pi_{\theta^f}|\tilde{\mathcal A}_c\cap X_j$ and $\pi_{\theta^f}|\tilde{\mathcal A}_c\cap X_{j+1}$ are both one-to-one and have Lipshitz inverses.
  \end{enumerate}
\end{thm}

For the rest of this section, we will derive various estimates
of quantities and set arising from Mather theory.
We deduce Theorem \ref{localize} and Theorem~\ref{graph} from these estimates.

\subsection{Vertical estimates}
We derive estimates on the Mather sets of a general Hamiltonian
$H(t, \theta, p)$, depending on $\e$, under the assumptions that
$$
I/A\leq \partial_ {pp}H\leq AI
$$
in the sense of quadratic forms,
and
$$
\|\partial_{\theta} H\|_{C^1} \leq 2\e.
$$
Note that both Hamiltonians $H_{\e}$ and $N_{\e}$
satisfy these assumptions.
The main result in this section is:

\begin{prop}\label{vertical}
We assume that $\e\leq 1$.
For each cohomology $c\in \Rm^n$ and  each Weak KAM solution $u$ of $H_{\e}$
at cohomology $c$,
the set $\tilde \mI(u,c )$ is contained in a $36A \sqrt{\e}$-Lipshitz graph,
and in the domain $\|p-c\|\leq 6A \sqrt{n\e}$.
\end{prop}

It is useful to use the Lagrangian $L(t,\theta, v)$
associated to $H$.
Recalling the expressions
$$
\partial_{vv}L(t,\theta,v)=\big( \partial_{pp}H(t,\theta,\partial_vL(t,\theta,v)\big)^{-1},
$$
$$
\partial_{\theta v}L(t,\theta,v)=-\partial_{\theta p}
H\big(t,\theta,\partial_vL(t,\theta,v)\big) \partial_{vv}L(t,\theta,v)
$$
and
$$
\partial_{\theta \theta}L(t,\theta,v)=
-\partial_{\theta\theta}H(t,\theta,\partial_vL(t,\theta,v))
-\partial_{\theta p}
H\big(t,\theta,\partial_vL(t,\theta,v)\big) \partial_{\theta v}L(t,\theta,v),
$$
we obtain the estimates
$$
\|\partial_{vv}L\|_{C^0}\leq A, \quad
\|\partial_{\theta v}L\|_{C^0}\leq 2 A \e, \quad
\|\partial_{\theta \theta}L\|_{C^0}\leq 3  \e
$$
when $\epsilon<\epsilon_0(A)$.

We  recall the concept of semi-concave function on $\Tm^n$.
A function $u:\Tm^n\lto \Rm$ is called $K$-semi-concave if the function
$$
x\lmto
u(x)-K\|x\|^2/2
$$
is concave on $\Rm^n$, where $u$ is seen as a periodic function on $\Rm^n$.
It is equivalent to require that, for each $\theta\in \Tm^n$,
there exists a linear form $l$ on $\Rm^n$
such that the inequality
$$
u(\theta+y)\leq u(\theta) + l\cdot y +K\|y\|^2/2
$$
holds for each $y\in \Rm^n$. The following Lemma is a simple case of
Lemma A.10 in \cite{Be1}:

\begin{lem}\label{scl}
If $u: \Tm^n \lto \Rm$
is $K$-semi-concave, then it is $(K\sqrt{n})$-Lipshitz.
\end{lem}
\proof
For each $x\in \Tm^n$, there exists $l_x\in \Rm^n$ such that
$$
u(x+y)\leq u(x)+l_x\cdot y +K\|y\|^2
$$
for all $y\in \Rm^n$. By applying this inequality with
$y=(\pm 1,0,0,\cdots,0)$, we conclude that
the first component $(l_x)_1$ of $l_x$ satisfies
$|(l_x)_1|\leq K$.
Similar estimates hold for the other components of $l_x$, and we conclude that
 that $\|l_x\|\leq K\sqrt{n}$ for each $x$,
and thus that $u$ is $K\sqrt{n}$-Lipshitz.
\qed

We will   need the following regularity result of Fathi:

\begin{lem}\label{freg}
Let $u$ and $v$ be  $K$-semiconcave functions, and let $\mI\subset \Tm^n$
be the set of points where the sum $u+v$ is minimal.
Then the functions $u$ and $v$ are differentiable at each point of $\mI$,
and the differential $x\lmto du(x)$ is $6K$-Lipshitz on $\mI$.
\end{lem}

Let us recall that the Weak KAM solutions of cohomology $c$ are defined as fixed points
of the operator
$\mT_c: C(\Tm^n)\lto C(\Tm^n)$
defined by
$$
\mT_c(u)  (\theta):=
\min_{\gamma} u(\gamma(0))+\int_0^1 L(t,\gamma(t),\dot \gamma(t)) +
c\cdot \dot \gamma(t) dt,
$$
where the minimum is taken on the set of $C^1$ curves $\gamma:[0,1]\lto \Tm^n$
satisfying the final condition $\gamma(T)=\theta$.

\begin{prop}\label{scc}
For each $c\in \Rm^n$, each Weak KAM solution
$u$ of $L+c\cdot v$ is $6A\sqrt{\e}$-semi-concave
and $6A \sqrt{n\e}$-Lipshitz.
\end{prop}

\proof
For each $T\in \Nm$ and $\theta\in \Tm^n$,
we have
$$u(\theta)=\min_{\gamma} u(\gamma(0))+\int_0^T L(t,\gamma(t),\dot \gamma(t)) +
c\cdot \dot \gamma(t) dt,
$$
where the minimum is taken on the set of $C^1$ curves $\gamma:[0,T]\lto \Tm^n$
satisfying the final condition $\gamma(T)=\theta$.
Let $\Theta(t)$ be an optimal curve in that expression, meaning that $\Theta(t)=\theta$ and that
$$u(\theta)=
u(\Theta(0))+\int_{0}^T L(t,\Theta(t),\dot \Theta(t) )dt.
$$
We lift $\Theta$  (and the point $\theta=\Theta(T)$) to a curve in $\Rm^n$ without changing its
name, and consider, for each $x\in \Rm^n$,
the curve
$$\Theta_x(t):=\Theta(t)+tx/T,$$
so that $\Theta_x(T)=\theta+x$.
We have the inequality
$$
u(\theta+x)-u(\theta)\leq \int_0^T L(t,\Theta_x(t),\dot \Theta_x(t))-
L(t,\Theta(t), \dot \Theta(t)) +c\cdot x/T \,dt.
$$
The integrand can be estimated as follows:
\begin{equation}\label{action-cp}
\begin{aligned}
L(t,\Theta_x(t),\dot \Theta_x(t))&\leq
L(t,\Theta(t), \dot \Theta(t))\\
&+
\partial_{\theta}L (t,\Theta(t), \dot \Theta(t))\cdot tx/T
+\partial_vL (t,\Theta(t), \dot \Theta(t))\cdot x/T\\
&+3\e |tx/T|^2+2A \e t|x/T|^2+A|x/T|^2/2
\end{aligned}
\end{equation}
Integrating, and using the Euler-Lagrange equation, we conclude that
$$
u(\theta+x)-u(\theta)\leq (c+\partial_vL(T,\Theta(T), \dot \Theta(T))\cdot x
+(\e T+ \e +1/2T)A|x|^2
$$
for each $T\in \Nm$.
Taking $T\in [1/2\sqrt{\e},1/\sqrt{\e}]$ (this interval contains an integer since
$\e\leq 1$), we obtain
$$
u(\theta+x)-u(\theta)\leq (c+\partial_vL(T,\Theta(T), \dot \Theta(T))\cdot x
+3A\sqrt{\e}|x|^2.
$$
This ends the proof of the semi-concavity.
The Lipshitz constant can then be obtained from Lemma \ref{scl}.
\qed

Let $u$ be a weak KAM solution, and let $\check u$ be the conjugated dual weak KAM
solution.
Then the set $\tilde \mI(u,c)$ can be characterized as follows:
Its projection $\mI(u,c)$ on $\Tm^n$ is the set where $u=\check u$,
and
$$\tilde \mI(u,c)=\{(x,c+du(x)), x\in \mI(u,c)\}.
$$
Since $-\check u$ is semi-concave, it is a consequence of Lemma \ref{freg}
that the differential $du(x)$ exists for $x\in \mI(u,c)$.
Moreover, we can prove exactly as in Proposition  \ref{scc} that
$-\breve u$ is $6A\sqrt{\e}$-semi-concave.
Lemma \ref{freg} then implies that the map $x\lmto du(x)$ is
$36A\sqrt{\e}$-Lipschitz on $\mI(u,c)$.
Moreover, $du(x)$ is bounded by the Lipschitz constant $6A\sqrt{n\e}$ of $u$.
This ends the proof of Proposition \ref{vertical}
\qed

\subsection{Horizontal localization}
For $\theta^s\in \T^{n-1}$ and $r>0$, let $D(\theta, r)$ denote the closed Euclidean ball centered at $\theta^s$ with radius $r$.
\begin{prop}\label{loc}
Let $N_\epsilon=H_0+ \epsilon Z + \epsilon R$ with $Z$ satisfying $[G0]-[G2]$, then there exists $\kappa>0$ depending only on $A$, $b$ and $n$, $\epsilon_0>0$ depending only on $A$ and $\delta$ such that for each $\e\in ]0,\e_0[$ and $c\in \Gamma(\epsilon)$ we have the following results on the projected Ma\~{n}e set and the Aubry set.
\begin{enumerate}
\item If $c=(p^s_*(c^f), c^f)$ with $c^f\in [a_j+b, a_{j+1}-b]$, then for each weak KAM solution $u$ of $N_\epsilon$ at cohomology $c$ we have that
$$ {\mathcal I}(u,c) \subset  D\left(\theta_j^s(c^f), \kappa \delta^{\frac14}\right)\times \T, $$
as a consequence, we have
$$ {\mathcal N}_{N_\epsilon}(c) \subset  D\left(\theta_j^s(c^f), \kappa \delta^{\frac14}\right)\times \T. $$
\item If $c=(p^s_*(c^f), c^f)$ with $c^f\in [a_{j+1}-b, a_{j+1}+b]$, then
$$ {\mathcal A}_{N_\epsilon}(c)\subset\left(  D\left(\theta_j^s(c^f), \kappa \delta^{\frac14}\right)\times \T \right)\cup\left(  D\left(\theta_{j+1}^s(c^f), \kappa \delta^{\frac14}\right)\times \T\right).  $$
\end{enumerate}
\end{prop}
The Lagrangian $N^*(t,\theta,v)$ associated to $N_{\e}$ will play a central role
in the proof.
We write it
$$
N^*(t,\theta,v)=L_0(v)+\e Z(\theta^s,\partial_v L_0 (v))+\e L_2(t,\theta,v,\e),
$$
where $L_0$ is the Legendre dual of $H_0$.
We have
$$2I/A\leq\partial_{vv}L_0\leq  AI/2
$$
in the sense of quadratic forms.

\begin{lem}
We have the estimate
$$
\inf H_2
\leq L_2(t,\theta,v)\leq A \e+\sup H_2.
$$
\end{lem}
\proof
Let us first
consider the truncated Hamiltonian
$$\tilde H(t,\theta,p)=
H_0(p)+\e Z(\theta^s,p)
$$
and the associated Lagrangian $\tilde L(\theta^s,v)$.
We claim that
\begin{equation}\label{tilde}
L_0(v)-\e Z(\theta^s,\partial L_0(v))
\leq
\tilde L(\theta^s,v)
\leq
L_0(v)-\e Z(\theta^s,\partial L_0(v))+\frac{\e}{8\kappa}.
\end{equation}
In order to prove the left inequality, we write
\begin{align*}
\tilde L(\theta^s,v)&=\sup_p [p\cdot v-H_0(p)-\e Z(\theta^s,p)]\\
&\geq \partial L_0(v)\cdot v -H_0(\partial L_0(v))-\e Z(\theta^s,\partial L_0(v))\\
&\geq L_0(v)-\e Z(\theta^s,\partial L_0(v)).
\end{align*}
The right inequality follows from the following computation:
\begin{align*}
\tilde L(\theta^s,v)&=\sup_p [p\cdot v-H_0(p)-\e Z(\theta^s,p)]\\
&\leq \sup_p \Big[ p\cdot v- H_0(\partial L_0(v))-v\cdot(p-\partial L_0(v))
-\|p-\partial L_0(v)\|^2/A\\
& \quad \quad -\e Z(\theta^s,\partial L_0(v))
+\e \|p-\partial L_0(v)\|\Big]\\
&\leq \sup_p\Big[v\cdot \partial L_0(v) -H_0(\partial L_0(v))
-\e Z(\theta^s,\partial L_0(v))\\
&\quad \quad + \e \|p-\partial L_0(v)\|-  \|p-\partial L_0(v)\|^2 /A
\Big]\\
&\leq L_0(v)-\e Z(\theta^s,\partial L_0(v))+\sup_{y\geq 0} [\e y- y^2/A)]\\
&\leq L_0(v)-\e Z(\theta^s,\partial L_0(v))+A\e^2.
\end{align*}
Now we have estimated $\tilde L$, we observe that
$$
\tilde H(t,\theta,p)-\e \sup H_2
\leq
N_{\e}(t,\theta,p)\leq
\tilde H(t,\theta,p)-\e \inf H_2
$$
from which follows that
$$
\tilde L(t,\theta,v)+\e \inf H_2 \leq
L(t,\theta,v)\leq \tilde L(t,\theta,v)+\e \sup H_2,
$$
which implies the desired estimates in view of (\ref{tilde}).
\qed

Let us now estimate the $\alpha$ function of $N_{\e}$:

\begin{prop}
The $\alpha$ function of Mather is estimated at the points $c\in \Gamma$
 in the following way:
$$
 H_0(c)+\e Z(\theta^s(c),c)- \e \max_{(t,\theta)\in \Tm^{n+1}}L_2(t,\theta,\partial H_0(c))
\leq \alpha(c)\leq
$$
$$ \leq
H_0(c)+\e Z(\theta^s(c),c)- \e \min_{(t,\theta)\in \Tm^{n+1}}H_2(t,\theta,c)
$$
thus
$$
H_0(c)+\e Z(\theta^s(c),c)- \e \|H_2\|_{C^0}-A\e^2
\leq \alpha(c)\leq H_0(c)+\e Z(\theta^s(c),c)+ \e \|H_2\|_{C^0}
$$
\end{prop}

\proof
We have
$$\alpha(c)\leq \max_{(t,\theta)} H(t,\theta,c)\leq H_0(c)
+\e \max_{\theta^s}Z(\theta^s,c)
-\e \min_{(t,\theta)\in \Tm^{n+1}}H_2(t,\theta,c)
$$
which is the desired right hand side.
On the other hand, let us set $\omega=\partial H_0(c)\in \Rm^n$
and observe that $c=\partial L_0(\omega)$.
We  can consider the Haar measure $\mu$  of the torus
$\Tm\times \Tm \times \{\Theta^f(c)\}\times \{\omega\},$
 this measure is not necessarily invariant  but it is closed.
We thus have
\begin{align*}\alpha(c)\geq c\cdot \omega -\int N^* d\mu
&=c\cdot \omega -L_0(\omega)+\e Z(\Theta^f(c),c) -\e \int L_2 d\mu \\
&\geq H_0(c)+\e Z(\Theta^f(c),c) -
\e \max_{(t,\theta)\in \Tm^{n+1}}L_2(t,\theta,\omega)
\end{align*}

\qed
\begin{lem}
For each $c\in \Gamma$, have the estimates
\begin{align}
N^*(t,\theta,v)-c\cdot v +\alpha(c)
&\geq \|v-\partial H_0(c)\|^2/(2A)-\e \hat Z_c(\theta^s)-\e \eta\\
N^*(t,\theta,v)-c\cdot v +\alpha(c)
&\leq
A\|v-\partial H_0(c)\|^2 /2-\e \hat Z_c(\theta^s)+\e \eta
 \end{align}
where
$\hat Z_c(\theta^s)=Z(\theta^s,c)-\max_{\theta^s}Z(\theta^s,c)$ and
$$
\eta =2\|H_2\|_{C^0} + (2A+A^3) \e.
$$
\end{lem}
\proof
It is a direct computation:
\begin{align*}
 N^*(t,\theta,v) - & c\cdot v  +\alpha(c) \leq \\
& \leq L_0(v)-c\cdot v+H_0(c)-\e Z(\theta^s,\partial_vL_0(v))\\
&+\e \max_{\theta^s}Z(\theta^s,c)+A\e^2+\e \sup H_2 -\e \min H_2(t,\theta,c)\\
&\leq A\|v-\partial H_0(c)\|^2/4-\e \hat Z_c(\theta^s)+\e \|\partial L_0(v)-c\|
+A\e^2+ 2\e \|H_2\|_{C^0}\\
&\leq A\|v-\partial H_0(c)\|^2/4-\e \hat Z_c(\theta^s)+A\e  \|v-\partial H_0(c)\|
+A\e^2+ 2\e \|H_2\|_{C^0}\\
&\leq A\|v-\partial H_0(c)\|^2/2-\e \hat Z_c(\theta^s)+2A{\e}^2
+ 2\e \|H_2\|_{C^0}\\
N^*(t,\theta,v)- & c\cdot v  +\alpha(c)
\geq L_0(v)-c\cdot v+H_0(c)-\e Z(\theta^s,\partial_vL_0(v))\\
&+\max_{\theta^s}Z(\theta^s,c)-2\e \|H_2\|_{C^0} -A\e^2 \\
&\geq \|v-\partial H_0(c)\|^2/A-\e \hat Z_c(\theta^s)-\e \|\partial L_0(v)-c\|
- 2\e \|H_2\|_{C^0}-A\e^2 \\
&\geq \|v-\partial H_0(c)\|^2/A-\e \hat Z_c(\theta^s)-A\e  \|v-\partial H_0(c)\|
-2\e \|H_2\|_{C^0}-A\e^2 \\
&\geq \|v-\partial H_0(c)\|^2/(2A)-\e \hat Z_c(\theta^s)-A{\e}^2 -A^3\e^2
- 2\e \|H_2\|_{C^0}
\end{align*}
\qed

 We can now estimate the oscillation of a weak KAM solution near $\theta^s_j$ and $\theta^s_{j+1}$.

\begin{lem}
  Let $u(t, \theta)$ be a weak KAM solution at cohomology $c$.
  \begin{enumerate}
  \item If $c=(p^s_*(c^f), c^f)$ with $c^f\in [a_j+b, a_{j+1}-b]$, then for any $(t_1, \theta_1), (t_2, \theta_2)\in \T\times D(\theta_j^s(c^f))\times \T$
$$ u(t_2, \theta_2)-u(t_1, \theta_1)\le 4r_1 \sqrt{nA\epsilon}$$
where $r_1=\sqrt{4\eta/b}$.
  \item If $c=(p^s_*(c^f), c^f)$ with $c^f\in [a_{j+1}-b, a_{j+1}+b]$, then for either $(t_1, \theta_1), (t_2, \theta_2)\in \T\times D(\theta_j^s(c^f))\times \T$ or $(t_1, \theta_1), (t_2, \theta_2)\in \T\times D(\theta_{j+1}^s(c^f))\times \T$,
$$ u(t_2, \theta_2)-u(t_1, \theta_1)\le 4r_1 \sqrt{nA\epsilon}. $$
  \end{enumerate}
\end{lem}

\begin{proof}

Using $\|Z\|_{C^2}\le 1$, we have that
$$\hat Z_c(\theta^s)\geq - \frac12 \|\theta^s-\theta^s_j(c^f)\|^2. $$

We take two points
$(t_i, \theta_i)$, $i=1$ or $2$ in this domain,
and consider the curve
$$
\theta(t)=\theta_1+(t- \tilde t_1)
\frac{\tilde  \theta_2- \tilde\theta_1+ [(T+\tilde t_2-\tilde t_1)\partial H_0(c)]}{\tilde t_2-\tilde t_1+T}
$$
where $T\in \Nm$ is a parameter to be fixed later,
and where $\tilde t_i\in [0,1[$ and $\tilde \theta _i\in [0,1[^{n}$
are representatives of the angular variables $t_i, \theta_i$,
and $[\omega]\in \Zm^{n}$ is the component-wise integral part of $\omega$.
Note that $\theta(\tilde t_1)=\theta_1$ and $\theta(\tilde t_2+T)=\theta_2$,
hence
\begin{align*}
u(t_2,\theta_2)-u(t_1,\theta_1)
&\leq \int_{\tilde t_1}^{\tilde t_2+T}
L(t,\theta(t),\dot \theta (t))-c\cdot \dot \theta(t)+\alpha(c)dt\\
&\leq \int_{\tilde t_1}^{\tilde t_2+T}
A\|\dot \theta -\partial H_0(c)\|^2/2+\e \hat Z_c(\theta^s(t))+\e \eta dt\\
&\leq \int_{\tilde t_1}^{\tilde t_2+T}
\frac{2An}{(T+\tilde t_2-\tilde t_1)^2}+\e r_1^2/2+\e\eta dt\\
&\leq \int_{\tilde t_1}^{\tilde t_2+T}
\frac{2An}{(T+\tilde t_2-\tilde t_1)^2}+\e r_1^2 dt\\
&\leq \frac{2An}{(T+\tilde t_2-\tilde t_1)}+(T+\tilde t_2-\tilde t_1)\e r_1^2.
\end{align*}
This inequality holds for all $T\in \Nm$,
in particular, we can choose $T\in \Nm$
so that
$$
\sqrt{\frac{nA}{\e r_1^2}}\leq T+\tilde t_2-\tilde t_1\leq 2\sqrt{\frac{nA}{\e r_1^2}}
$$
and obtain
$$
u(t_2,\theta_2)-u(t_1,\theta_1)
\leq
4r_1\sqrt{nA\e}.
$$
\end{proof}

Up to now, we used that $\|Z\|_{C^2}\leq 1$, but we used no information  on
the shape of $Z$. Now we use properties [G1'] and [G2'] to prove Proposition~\ref{loc}.

\subsubsection{The single peak case}

This concerns the first statement of Proposition~\ref{loc}, where
 $$c=p_*(c^f),  \quad c^f\in [a_j+b, a_{j+1}-b]. $$
By [G1'] the function $Z(\theta^s, c)$ as a single peak at $\theta_j^s$, as a consequence
$$ \hat Z_c(\theta^s) \le - b \|\theta^s-\theta^s_j(c^f)\|^2. $$

 Let $\theta(t):\Rm \lto \Tm^m$ be a curve calibrated by $u$.
Then the function
$$
t\lmto L(t,\theta(t),\dot \theta(t))-c\cdot \dot \theta(t)+\alpha(c)
$$
is integrable. Since
$$L(t,\theta,v)\geq -\e\hat Z_c(\theta^s)-\e \eta \geq \e b r_1^2-\e \eta\geq \e b r_1^2/4
$$
if $\theta^s$ does not belong to $D(\theta^s_j, r_1)$, we conclude that
the set of times $t$ for which $\theta^s(t)$ does not belong
to $D_c(r_1)$ has finite measure, and is an open set.
Let $]t_1,t_2[$ be a connected component of this open set of times.
Then $\theta^s(t_1)$ and $\theta^s(t_2)$ belong to $D(\theta_j^s, r_1)$
hence
$$
\int_{t_1}^{t_2}L(t,\theta(t),\dot \theta(t))-c\cdot \dot \theta(t)+\alpha(c)dt=
u(t_2,\theta(t_2))-u(t_1,\theta(t_1))\leq 4r_1\sqrt{nA\e}.
$$
Now let $r_0$ be the maximum of the distance $\|\theta^s(t)-\Theta^s(c)\|$,
assume that $r_0\geq 2 r_1$.
Let $t_4$ be the smallest solution of the equation
$\|\theta^s(t)-\Theta^s(c)\|=r_0$ in $]t_1,t_2[$, and let
$t_3<t_1$ be the greatest solution of the equation
$\|\theta^s(t)-\Theta^s(c)\|=r_0/2$ in $[t_1,t_4]$.
Note that
$$
\int_{t_3}^{t_4}L(t,\theta(t),\dot \theta(t))-c\cdot \dot \theta(t)+\alpha(c)dt
\leq 4r_1\sqrt{nA\e}
$$
because the integrand is positive on $]t_1,t_2[$.
We conclude that
$$
 \int_{t_3}^{t_4}
 \|\dot \theta^s(t)\|^2/(2A)+ \e b r_0^2-\e \eta dt \leq 4r_1\sqrt{nA\e} ,
$$
and, using  the Cauchy-Schwartz inequality, that
\begin{equation}
  \label{eq:loc}
\begin{aligned}
  \int_{t_3}^{t_4}
 \|\dot \theta^s(t)\|^2/(2A)+ \e b r_0^2-\e \eta dt &\geq \frac{1}{2A(t_4-t_3)}\left(\int _{t_3}^{t_4} \|\dot \theta ^s(t)\|dt\right)^2 + \frac{\e b r_0^2}{2}(t_4-t_3)\\
 &  \geq \frac{ r_0^2}{2A(t_4-t_3)}+ \frac{\e b r_0^2}{2}(t_4-t_3) \geq \frac{ r_0^2\sqrt{b\e}}{2\sqrt{A}} .
\end{aligned}
\end{equation}
Finally, we obtain
$$
\frac{ r_0^2\sqrt{b\e}}{2\sqrt{A}}\leq 4r_1\sqrt{nA\e}.
$$
or equivalently
$$
r_0^2\leq r_1 8A\sqrt{ n/b}=\frac{8A\sqrt{2n}}{b}\sqrt {\eta}.
$$

\subsubsection{Double peak case}
\label{sec:double-peak}

We now turn to the case of
 $$c=p_*(c^f),  \quad c^f\in [a_{j+1}-b, a_{j+1}+b], $$
where the function $Z(\theta^s, c)$ has two potential maxima. It follows from [G2'] that
$$ \hat Z_c(\theta^s) \le - b \left(\min\{\|\theta^s-\theta^s_j\|, \|\theta^s-\theta^s_{j+1}\|\}\right)^2. $$
Let $\theta_0=(\theta_0^s, \theta_0^f)\in {\mathcal A}_{N_\epsilon}(c)$ be where the function (of $\theta$) $\min\{\|\theta^s-\theta^s_j\|, \|\theta^s-\theta^s_{j+1}\|\}$ achieves its maximum. This is possible since ${\mathcal A}_{N_\epsilon}(c)$ is a compact set. Since $h(\theta_0, \theta_0)=0$, there exists an increasing sequence of integers $n_k$ and absolutely continuous curves $\theta_k:\R\to \T^n$  satisfying $\theta_k(0)=\theta_0$ and $\theta_k(t+n_k)=\gamma_k(t)$, and
$$ \lim_{k\to \infty}\int_0^{n_k}L(t, \theta_k(t), \dot{\theta}_k(t))- c\cdot \theta_k(t) + \alpha(c) dt = 0.$$
Similar to the first case,  $L(t, \theta, v)\geq \e b r_1^2/4$ for $\theta^s\notin D(\theta_j^s, r_1)\cup D(\theta_{j+1}^s, r_1)$, we conclude that for sufficiently large $k$, $\theta_k(\R)$ must intersect $D(\theta_j^s, r_1)\cup D(\theta_{j+1}^s, r_1)$. Without loss of generality, we may assume that it intersect $D(\theta_j^s, r_1)$.

Let $t_1=\min\{t_1<0, \theta_k(t_1)\in D(\theta_j^s, r_1)\}$ and $t_2=\min\{0<t_2, \theta_k(t_2)\in D(\theta_j^s, r_1)\}$. We first study the action of $\theta_k$ on the interval $[0, t_2]$. If $\theta_k([0, t_2])$ does not intersect $D(\theta_{j+1}^s, r_1)$, write $t_3=t_4=t_2$, otherwise,  Write $t_3=\min\{0<t_3\le t_2, \gamma(t_3)\in D(\theta_{j+1}^s, r_1)\}$ and $t_4=\max\{t_3\le t_4 \le t_2, \gamma(t_4)\in D(\theta_{j+1}^s, r_1)\}$.

We still use $r_0$ to denote the maximal distance $\min\{\|\theta^s_0-\theta^s_j\|, \|\theta^s_0-\theta^s_{j+1}\|\}$, assume that $r_0\ge 2b_1$. Let $t_5\in [0, t_3]$ be the smallest solution in $[0, t_3]$ such that $d(\theta^s(t_5))=r_0/2$, then by the same calculation as in \eqref{eq:loc},
$$\int_{0}^{t_5}L(t, \theta_k(t), \dot{\theta}_k(t))- c\cdot \dot\theta_k(t)+\alpha(c) dt \geq \frac{r_0^2\sqrt{b\epsilon}}{2\sqrt{A}}. $$
Furthermore for any weak KAM solution $u(\theta, t)$
$$ \int_{t_3}^{t_4}L- c\cdot \dot\theta_k(t)+\alpha(c) dt \geq u(t_4, \theta_k(t_4))-u(t_3, \theta_k(t_3))\ge -4r_1\sqrt{nA\epsilon},$$
while the integrand is nonnegative on both $[t_5, t_3]$ and $[t_4, t_2]$. We conclude that
$$\int_{0}^{t_2}L- c\cdot \dot\theta_k(t)+\alpha(c) dt\geq \frac{r_0^2\sqrt{\epsilon}}{2\sqrt{A}} - 4r_1\sqrt{nA\epsilon}.  $$
The same estimate can be made for the action on the interval $[t_1, 0]$. In addition
$$ \int_{t_2}^{n_k + t_1}L- c\cdot \dot\theta_k(t)+\alpha(c) dt \geq u(n_k+t_1, \theta_k(n_k+t_1))-u(t_2, \theta_k(t_2))\ge -4r_1\sqrt{nA\epsilon},$$
note that $t_2<n_k+t_1$ and that $\theta_k$ is $n_k$ periodic.

Finally we conclude that
$$ \int_{t_1}^{n_k + t_1}L-c\cdot \dot\theta_k + \alpha(c) dt \geq \frac{r_0^2\sqrt{b\epsilon}}{2\sqrt{A}} - 12b_1\sqrt{nA\epsilon}.  $$
Let $k\to \infty$, the integral on the left hand side approaches $0$. We obtain
$$\frac{r_0^2\sqrt{b\epsilon}}{2\sqrt{A}} \leq 12b_1\sqrt{nA\epsilon}$$
and
$$ r_0^2 \le \frac{24 A\sqrt{n}}{b}\sqrt{\eta}. $$

We choose $\epsilon_0$ sufficiently small such that $\eta\le 2\|H_0\|_{C_0}$, choose $\kappa=\sqrt{\frac{48 A\sqrt{2n}}\lambda}$ and verify that we have proved the statements of Proposition~\ref{loc} in both cases.
\qed

Before moving on we point out that the estimates in the double peak case indeed implies that the curves $\theta_k$, which are not calibrated, can be localized in the limit of $k\to \infty$. We state it in the following lemma for future use.

\begin{lem}\label{aubry-est}
For the double peak case, i.e. $c=p_*(c^f)$ with  $c^f\in [a_{j+1}-b, a_{j+1}+b]$, consider any $\theta_0\in {\mathcal A}_{N_\epsilon}(c)$, let $n_k$ be an increasing sequence of integers, $\theta_k=(\theta^s_k, \theta^f_k):\R \to M$ be a sequence of $n_k-$periodic absolute continuous curves such that $\gamma(0)=\theta_0$ and
$$ \lim_{k\to \infty}\int_0^{n_k}L(t, \theta_k, \dot{\theta}_k)-c\cdot \dot\theta_k +\alpha(c) dt =0,$$
then there exists $K\in {\mathbb N}$ such that for all $k>K$
$$ \max_{t\in \R} \min\{\|\theta^s_k(t)-\theta^s_j\|, \|\theta^s_k(t)-\theta^s_{j+1}\|\}< 2\kappa\delta^{\frac14}. $$
\end{lem}

\begin{proof}
  Fix a curve $\theta_k$, write $d(\theta_k(t))=\min\{\|\theta^s_k(t)-\theta^s_j\|, \|\theta^s_k(t)-\theta^s_{j+1}\|\}$. Let $\tau$ be where the maximum of $d(\theta_k(t))$ is reached. Consider the shifted curve $\theta'_k(t)=\theta_k(t-\tau)$, the arguments in section~\ref{sec:double-peak} go through with $\theta_k$ replaced with $\theta'_k$ and $r_0$ replaced by $\max d(\theta_k(t))$. We have that
$$ \lim_{k\to \infty}\max_{t\in \R}d(\theta_k(t))\le \kappa \delta^{\frac14}$$
and the lemma follows.
\end{proof}

\subsection{Proof of Theorem~\ref{localize}}

Proposition~\ref{vertical} provides the vertical part of the localization, while Proposition~\ref{loc} provides the horizontal localization, with $\rho_2=\kappa \delta^{\frac14}$. Clearly we can choose $\delta_0$ small enough such that $\rho_2<\rho_1$.
\qed

\subsection{Proof of Theorem~\ref{graph}}

For the first case, where $c=p_*(c^f)$ with $c^f\in [a_j+b, a_{j+1}-b]$. For a sufficiently small choice of $\epsilon_0$, Theorem~\ref{localize} implies that ${\mathcal N}_{N_\epsilon}(c)\subset V_j$, where $V_j=\{(\theta,p,t); p^f\in [a_j-b, a_{j+1}+b], \|(\theta^s,p^s)-(\theta^s_j,p^s_*)\|\le \rho_1\}$ was defined in Theorem~\ref{nhic-mult}. Since $V_j$ is maximally invariant and that ${\mathcal N}_{N_\epsilon}(c)$ is an invariant set, we conclude that ${\mathcal N}_{N_\epsilon}(c)\subset X_j$.

For the second case, where $c=p_*(c^f)$ with $c^f\in [a_{j+1}-b, a_{j+1}+b]$, we can similarly claim that ${\mathcal A}_{N_\epsilon}(c)\subset  V_j\cup V_{j+1}$, moreover ${\mathcal A}_{N_\epsilon}(c)\cap V_j$ and ${\mathcal A}_{N_\epsilon}(c)\cap V_{j+1}$ must both be invariant, and hence ${\mathcal A}_{N_\epsilon}(c)\cap V_j\subset X_j$  and ${\mathcal A}_{N_\epsilon}(c)\cap V_{j+1}\subset X_{j+1}$.

In order to prove the projection part of Theorem \ref{graph}, let us
consider a Weak KAM solution $u$ of $N_{\e}$ at cohomology $c$.
Let $(t_i,\theta_i,p_i), i=1,2$ be two points
in $\tilde \mI(u,c)$.
We shall denote by the same symbol $\kappa$ various different constants which
are independent of $\delta$ and $\e$.
By Proposition \ref{vertical}, we have
$$\|p_2-p_1\|\leq 36A \sqrt{\e}\|\theta_2-\theta_1\|
\leq 36A\sqrt{\e}(\|\theta_2^f-\theta_1^f\|+\|\theta_2^s-\theta_1^s\|).
$$
Assume that these two points belong to one of the NHICs $X_j$,
we also have
$$
\|\theta_2^s-\theta_1^s\|\leq (\kappa\delta/\sqrt{\e})
(\|\theta_2^f-\theta_1^f\|+\|p_2-p_1\|).
$$
We get
$$
\|p_2-p_1\|\leq \kappa\sqrt{\e}\|\theta_2^f-\theta_1^f\|+\kappa \delta\|p_2-p_1\|
$$
thus, if $\delta$ is small enough,
$$
\|p_2-p_1\|\leq \kappa \sqrt{\e}\|\theta^f_2-\theta^f_1\|
$$
and then
$$
\|\theta^s_2-\theta^s_1\|\leq (\kappa \delta/\sqrt{\e}) \|\theta^f_2-\theta^f_1\|.
$$
We have proved that the restriction to
$\tilde \mI(u,c)$ of the coordinate map $\theta^f$ has a Lipschitz inverse.

Note that the Ma\~{n}e set $\tilde{\mathcal N}_{N_\epsilon}(c)$, as well as the components of the Aubry set  $\tilde{\mathcal A}_{N_\epsilon}(c)\cap V_j$ and $\tilde{\mathcal A}_{N_\epsilon}(c)\cap V_{j+1}$ are both contained in some $\tilde{I}(u,c)$, since we have just proved that they belong to NHIC, their projection to the $\theta^f$ component has a Lipshitz inverse.
\qed

\section{Variational Construction}\label{sec:variational}

More detailed information on these sets can be obtained, if we are allowed to make an additional perturbation to avoid degenerate situations.

\begin{thm}\label{property-twist}
Let  $N_\epsilon=H_0+\epsilon Z + \epsilon R$ be such that $Z$ satisfy
the genericity conditions [G0]-[G2] and that the parameters $\epsilon$
and $\delta$ is such that Theorem~\ref{nhic-mult} applies. Then there exists
arbitrarily small $C^r$ perturbation $ \epsilon R'$ of $\epsilon R$,
such that the following hold for $N_\epsilon'=H_0+\epsilon Z + \epsilon R'$:
  \begin{enumerate}
  \item There exists a partition of $[a_-, a_+]$ into $\bigcup_{j=0}^{l-1}[\bar{a}_j, \bar{a}_{j+1}]$, which is a refinement of the partition $\{[a_i, a_{i+1}]\}$. Each $[\bar{a}_j, \bar{a}_{j+1}]$ still corresponds to an invariant  cylinder $X_j$. We have that for  $c^f\in (\bar{a}_j,\bar{a}_{j+1})$,  the Aubry set $\tilde{\mA}_{N_\epsilon'}(c)$ is contained in $X_j$; for $c^f=\bar{a}_{j+1}$, $\tilde{\mA}_{N_\epsilon'}(c)$ has nonempty component in both $X_j$  and $X_{j+1}$, if $X_j\ne X_{j+1}$.
  \item The sets $\tilde{\mA}_{N_\epsilon'}(c)\cap X_j$, when nonempty, contains a unique minimal invariant probability measure. In particular, this implies that $\tilde{\mA}_{N_\epsilon'}(c)=\tilde{\mN}_{N_\epsilon'}(c)$ for $c^f\ne \bar{a}_j$ for any $j$.
  \item An immediate consequence of part (2) is the following dichotomy, for $c^f\ne \bar{a}_j$,  $j=1, \cdots, l$, one of the two holds.
    \begin{enumerate}
    \item ${\mathcal A}_c={\mathcal N}_c$ and $\pi_{\theta^f}{\mathcal A}_c=\T$. In this case, ${\mathcal A}_c$ is an invariant circle.
    \item $\pi_{\theta^f}{\mathcal N}_c\subsetneq \T$.
    \end{enumerate}
  \end{enumerate}
\end{thm}

Using the information obtained from the normal form system $N_\epsilon$, we now return to the original Hamiltonian $H_\epsilon$. Using the symplectic invariance of the Mather, Aubry and Ma\~{n}e set developed in \cite{Be2}, we have that the same conclusion as in Theorem~\ref{graph} and Theorem~\ref{property-twist} can be drawn about $H_\epsilon$.
\begin{thm}\label{orignal-H}
  Let $H_\epsilon=H_0 + \epsilon H_1$ such that the resonant component of $H_1$ satisfy conditions [G0]-[G2]. There exists $\epsilon_0>0$ and an interval $[a_-, a_+]\subset [a_{min}, a_{max}]$ depending only on $H_0$ and $H_1$, and for each $0<\epsilon<\epsilon_0$ there exists arbitrarily small $C^r$ perturbation $H'_\epsilon$ of $H_\epsilon$, such that the conclusions of Theorem~\ref{graph} and Theorem~\ref{property-twist} holds for the Hamiltonian $H'_\epsilon$ at cohomologies $c=p_*(c^f)$, where $c^f\in [a_-, a_+]$.
\end{thm}

These information on the Ma\~n\'e set allow to use the variational
mechanism of \cite{Be1}.
Let us denote  by $\Gamma(\e)$ the set of cohomology classes
$c\in \Gamma$ such that $c^f\in [a_-,a_+]$.
We would like to prove that each cohomology
$c\in \Gamma(\e)$ is in the interior of its forcing class in the terminology of
\cite{Be1}, which implies that all the cohomology classes in $\Gamma(\epsilon)$ are contained in a single forcing class.
By proposition 5.3 in \cite{Be1}, we could conclude  the existence of an orbit
$(\theta(t),p(t))$ of $H_{\e}$ such that $p(0)=c$ and $p(T)=c'$ for some $T\in \Nm$.
Note that this implies the existence of various more complicated orbits,
see \cite{Be1}.

In order to carry out this program,
we denote by $\Gamma_0(\e)$ the set of cohomology classes
$c\in \Gamma(\e)$ such that
the set $\theta^f(\tilde \mI(c,u))$ is properly contained in $\Tm$ for each
weak KAM solution $u$ at level $c$.
By Theorem 0.11 in \cite{Be1}, each cohomology $c\in \Gamma_0(\e)$
is in the interior of its forcing class.

Let $\Gamma_2(\epsilon)$ denote that set of $c\in \Gamma(\epsilon)$ such that the Aubry set ${\mathcal A}(c)$ has exactly two static classes. In this case the Ma\~{n}e set ${\mathcal N}(c)\supsetneq {\mathcal A}(c)$. To ensure that $c\in \Gamma_2(\epsilon)$ is in the interior of its forcing class, some further degeneracy conditions are needed. To be more specific, let $\Gamma_2^*(c)$ denote the set of $c\in \Gamma_2(\epsilon)$ such that the set
$${\mathcal N}(c) - {\mathcal A}(c) $$
is totally disconnected. This can also be stated in terms of barrier function. Let $\theta_0$ and $\theta_1$ be contained in each of the two static classes of ${\mathcal A}(c)$, we define
$$
{b}^+_c(\theta)= h_c( \theta_0, \theta)+
 h_c(\theta,  \theta_1)
$$
and
$$
{b}^-_c(\theta)= h_c( \theta_1, \theta)+
 h_c(\theta, \tilde \theta_0),
$$
where $h_c$ is the Peierls barrier for cohomology class $c$. Then $\Gamma^*_2(\epsilon)$ is the set of $c\in \Gamma_2(\epsilon)$ such that the minima of each $b^+$ and $b^-$ outside of ${\mathcal A}(c)$ are totally isolated.

We call $\Gamma_1(\e)$ the set of cohomology classes $c\in \Gamma$
such that there exists only one weak KAM solution $u$ at level $c$,
and $\theta^f(\tilde \mI(c,u))=\Tm$.
Note then that
$$
\tilde \mN(c)=\tilde \mA(c)=\tilde \mI(c,u)
$$
is an invariant circle.
We have $\Gamma_0(\e)\cap \Gamma_1(\e)=\emptyset$ for each
$\e\in ]0,\e_0[$, by definition.
We first consider the covering
\begin{align*}
 \xi:\Tm^n&\lto \Tm^n\\
\theta=(\theta^f, \theta^s_1,\theta^s_2,\cdots,\theta^s_{n-1})
&\lmto \xi(\theta)=
 (\theta^f, 2\theta^s_1,\theta^s_2,\cdots,\theta^s_{n-1}).
\end{align*}
This covering lifts to a a symplectic covering
\begin{align*}
\Xi:T^*\Tm^n&\lto T^* \Tm^n\\
(\theta,p)=(\theta, p^f, p^s_1,p^s_2,\ldots,p^s_{n-1})
&\lmto \Xi(\theta,p)=
(\xi(\theta),p^f, p^s_1/2,p^s_2,\ldots,p^s_{n-1}),
 \end{align*}
and we define the Lifted Hamiltonian $\tilde H=H\circ \Xi$.
It is known that
$$
\tilde \mA_{\tilde H}(\tilde c)=\Xi^{-1}\big(\tilde \mA_{\tilde H}(c)\big)
$$
where $\tilde c=\xi ^*c =(c^f,c^s_1/2,c^s_2,\ldots,c^s_{n-1})$.
On the other hand, the inclusion
$$
\tilde \mN_{\tilde H}(\tilde c)\supset \Xi^{-1}\big(\tilde \mN_{\tilde H}(c)\big)
=\Xi^{-1}\big(\tilde \mA_{\tilde H}(c)\big)
$$
is not an equality for $c\in \Gamma_1(\e)$.
More precisely, for $c\in \Gamma_1(\e)$, the set $\tilde \mA_{\tilde H}(\tilde c)$
is the union of two circles, while $\tilde \mN_{\tilde H}(\tilde c)$
contains heteroclinic connections between these circles.
Similarly to the case of $\Gamma_2(\epsilon)$, we call $\Gamma_1^*(\e)$ the set of cohomologies
$c\in \Gamma_1(\e)$ such that the set
$$
\mN_{\tilde H}(\tilde c)-\mA_{\tilde H}(\tilde c)
$$
is totally disconnected. Alternatively, we can chose a point $\theta_0$ in the projected Aubry set
$\mA(u,c)$ of $H$, and consider its two preimages $\tilde \theta_0$
and $\tilde \theta_1$ under $\xi$.
We define
$$
\tilde{b}^+_c(\theta)=\tilde h(\tilde \theta_0, \theta)+
\tilde h(\theta, \tilde \theta_1)
$$
and
$$
\tilde{b}^-_c(\theta)=\tilde h(\tilde \theta_1, \theta)+
\tilde h(\theta, \tilde \theta_0)
$$
where $\tilde h$ is the Peierl's barrier associated to $\tilde H$.
$\Gamma_1^*(\epsilon)$ is then the set of cohomologies
$c\in \Gamma_1(\epsilon)$ such that the  minima of each of the functions
$b^{\pm}_c$
located outside of the Aubry set
$
\mA_{\tilde H} (\tilde c)
$
are isolated.

The following theorem is proved in \cite{Be1}.
\begin{thm}\label{var}
If $c$ and $c'$ belong to the same connected component
of $\Gamma_0(\e)\cup \Gamma_1^*(\e) \cup \Gamma_2^*(\e)$, then  there exists an orbit
$(\theta(t),p(t))$ and of $H_{\e}$ a time $T\in \Nm$
such that $p(0)=c$ and $p(T)=c'$.
\end{thm}
We have proved the main result provided  $\Gamma(\epsilon)=\Gamma_0(\e)\cup \Gamma_1^*(\e)\cup \Gamma^*(\e)$.

\begin{thm}\label{nondegeneracy}
 Let $H_\epsilon$ be a Hamiltonian such that Theorem~\ref{orignal-H} holds, then there exists an arbitrarily small $C^r$ perturbation $H_\epsilon''$ to $H'_\epsilon$, such that for the Hamiltonian $H''_\epsilon$  we have that $\Gamma(\e)=\Gamma_0(\e)\cup \Gamma_1^*(\e)\cup \Gamma_2^*(\e)$.
\end{thm}

We note that the conclusions of Theorem~\ref{property-twist} already implies that $\{c\in \Gamma(\epsilon), c^f\ne \bar{a}_j, j=2, \cdots l-1\}\subset \Gamma_0(\epsilon)\cup \Gamma_1(\epsilon)$, while $\{c\in \Gamma(\epsilon), c^f=\bar{a}_j, j=2, \cdots l-1\}\subset \Gamma_2(\epsilon)$. In other words, $\Gamma(\epsilon)=\Gamma_0(\epsilon)\cup \Gamma_1(\epsilon) \cup \Gamma_2(\epsilon)$. It suffices to prove that $\Gamma_1(\epsilon)=\Gamma_1^*(\epsilon)$ and $\Gamma_2(\epsilon)=\Gamma_2^*(\epsilon)$.

For the rest of this section, we prove Theorem~\ref{property-twist} and Theorem~\ref{nondegeneracy}.

\subsection{Local extension of $\alpha(c)$ and ${\mathcal A}_{N_\e}(c)$}

Consider the normal form system $N_\epsilon$ and pick $c=p_*(c^f)$ with $c^f\in [a_-, a_+]$. For such a $c$ the function $Z(\theta^s, c)$ has a single peak. It follows from Theorem~\ref{graph} that the Aubry set $\tilde{\mathcal A}_{N_\epsilon}(c)$ (which is a subset of $\tilde{\mathcal N}_{N_\epsilon}(c)$) is contained in a single NHIC $X_j$ and the projected graph theorem holds. For the rest of the cohomology classes, the  double peak case, the picture is less clear as $\tilde{\mathcal A}_{N_\epsilon}(c)$ are contained in the union of two NHICs. To get a more precise picture, we will locally extend the set function (of $c^f$)
$$\tilde{\mathcal A}_{N_\epsilon}((p^s_*(c^f),c^f))|_{[a_j+b, a_{j+1}-b]}$$
from $[a_j+b, a_{j+1}-b]$ to $[a_j-2b, a_{j+1}+2b]$. The extended local Aubry set will still be contained in $X_j$. These definitions are inspired by Mather's definitions of relative $\alpha-$function and Aubry set.

Let
$$\rho_0=\min_{p^f\in [a_{j+1}-2b, a_{j+1}+2b]}\|\theta^s_j(p^f)-\theta^s_{j+1}(p^f)\|/3.$$
 It follows from properties [G1'] and [G2']  for $p^f \in [a_j-2b, a_{j+1}+2b]$,
\[   Z(\theta^s_j,p^s_*, p^f) - Z(\theta^s, p^s_*, p^f) \ge b \|\theta^s-\theta^s_j\|^2
\]
for $\|\theta^s-\theta^s_j\|\le \rho_0$. By choosing a smaller $\delta$ if necessary, we may make sure $\rho_0 > \rho_1$, where $\rho_1$ was defined in Theorem~\ref{nhic-mult}.  We write $U_j(p)=\{\|\theta^s-\theta^s_j(p)\|\le \rho_0\}$, our choice of $\rho_0$ guarantees that $U_j(p)\cap U_{j+1}(p)=\emptyset$ for $p^f\in [a_{j+1}-2b, a_{j+1}+2b]$. To define the extension, we introduce the following modification of the Hamiltonian $N_\epsilon$. Let ${Z}_j(\theta^s, p)$ be a function $\T^{n-1}\times \R^n\to \R$ satisfying the following properties:
\begin{itemize}
\item There exists $C$ depending only on $\lambda$, $\|Z\|_{C^2}$ and $n$ such that $\|{Z}_j\|_{C^2}\le C$.
\item ${Z}_j(\theta^s, p)=Z(\theta^s, p)$ whenever $\|\theta^s-\theta^s_j(p)\|\le \rho_0$.
\item ${Z}_j(\theta^s, p)\le Z(\theta^s, p)$  for all $\theta^s$ and $p$.
\item For $p^f\in [a_j-2b, a_{j+1}+ 2b]$, we have that ${Z}_j(\theta^s_j,p^s_*, p^f) - {Z}_j(\theta^s, p^s_*, p^f) \ge \frac{b}2 \|\theta^s-\theta^s_j\|^2$ hold for all $\theta^s\in \T^{n-1}$.
\end{itemize}
To see that such a modification exists, let $\bar{\rho}>\rho_0$ be such that  ${Z}_j(\theta^s_j,p^s_*, p^f) - {Z}_j(\theta^s, p^s_*, p^f) \ge \frac{b}2 \|\theta^s-\theta^s_j\|^2$ on $\|\theta^s-\theta^s_j\|\le {\rho}$. How large $\bar{\rho}-\rho_0$ is depends only on $\lambda$ and $\|Z\|_{C^2}$. Let $Q:\T^{n-1}\times \R^n\to \R$ be a smooth function such that $Z(\theta^s_j, p^s_*, p^f)-Q= \frac{b}2 \|\theta^s-\theta^s_j\|^2$ for $\|\theta^s-\theta^s_j\|\le \bar{\rho}$ and $Z(\theta^s_j, p^s_*, p^f)-Q\ge \frac{b}2 \|\theta^s-\theta^s_j\|^2$ for all $p^f$ and $\theta^s$. The norm of $Q$ only depends on $\lambda$ and $n$.  Let $\phi_{\rho_0, \bar{\rho}}:\T^n\times \R^n \to \R$ be a smooth function such that $\phi_{\rho_0, \bar{\rho}}=1$ on $\{\theta^s, \|\theta^s-\theta^s_j\|\le \rho_0\}$ and $\phi_{\rho_0, \bar\rho}=0$ on  $\{\theta^s, \|\theta^s-\theta^s_j\|> \bar\rho\}$. The norm of $\phi_{\rho_0, \bar{\rho}}$ depends only on $n$, $\rho_0$ and $\bar\rho$. Then we can choose
$$ {Z}_j=(1-\phi_{\rho_0, \bar\rho})Z + \phi_{\rho_0, \bar\rho}Q. $$
We write ${N}_{\epsilon, j}=H_0 + \epsilon {Z}_j + \epsilon R$.

For each $c=p_*(c^f)$ with $c^f\in [a_j-2b, a_{j+1}+2b]$ we define
$$ \alpha_j(c)=\alpha_{{N}_{\epsilon, j}}(c), \qquad {\mathcal A}_{N_\epsilon, j}(c)=\tilde{A}_{{N}_{\epsilon, j}}(c). $$
It is not clear that these definitions are independent of the choice of the modification ${Z}_j$ or the decomposition $N_\epsilon=H_0 + \epsilon Z + \epsilon R$. We resolve these questions, and provide some more properties of these definition in the following proposition.
\begin{prop}\label{relative}
  Let $N_\epsilon=H_0+\epsilon Z + \epsilon R$ be a Hamiltonian satisfying the genericity conditions [G0]-[G2] and that $\|R\|_{C^2}\le \delta$. There exists $\epsilon_0, \delta_0>0$ such that for $0<\epsilon<\epsilon_0$ and $0<\delta<\delta_0$ the following hold.
  \begin{enumerate}
  \item The definitions of $\alpha_j$ and $\tilde{\mathcal A}_{N_\epsilon, j}(c)$ are independent of the decomposition $N_\epsilon=H_0 + \epsilon Z + \epsilon R$ as long as $Z$ satisfies [G0]-[G2] and $\|R\|_{C^2}\le \delta$; the definitions are also independent of the modification ${Z}_j$, as long as it satisfies the same 4 bullet point properties.
  \item For each $c=p_*(c^f)$ with $c^f\in [a_j-2b, a_{j+1}+2b]$, we have the local Aubry set $\tilde{\mathcal A}_{N_\epsilon, j}(c)$ is contained in the set $\{\|\theta^s-\theta^s_j\|\le \rho_2\}$ where $\rho_2$ is as in Theorem~\ref{localize}. It follows that $\tilde{\mathcal A}_{N_\epsilon, j}(c)\subset X_j$ and $\pi_{\theta^f}|\tilde{\mathcal A}_{N_\epsilon, j}(c)$ is one-to-one with Lipshitz inverse.
  \item For $c=p_*(c^f)$, $\alpha(c)=\alpha_j(c)$ if $c^f \in [a_j+b, a_{j+1}-b]$; $\alpha(c)=\max\{\alpha_j(c), \alpha_{j+1}(c)\}$ if $c^f \in (a_{j+1}-b, a_{j+1}+b)$. In particular, $\alpha_j(c)>\alpha_{j+1}(c)$ for $c^f=a_{j+1}-b$ and $\alpha_{j+1}(c)>\alpha_j(c)$ for $c^f=a_{j+1}+b$.

  \item For any $c^f\in [\alpha_{j+1}-b, \alpha_{j+1}+b]$, if $\alpha(c)=\alpha_j$ and $\alpha(c)\ne \alpha_{j+1}(c)$, then $\tilde{\mathcal A}_{N_\epsilon}(c)=\tilde{\mathcal A}_{N_\epsilon, j}(c)$. Similar statement hold with $j$ and $j+1$ exchanged.
  \end{enumerate}
\end{prop}
\begin{proof}
  We will prove the second statement first. The modified Hamiltonian ${N}_{\epsilon, j}$ is such that the single peak case of Theorem~\ref{localize} applies, with $b$ replaced by $b/2$. By choosing a smaller $\delta$ if necessary, we can guarantee that $\rho_2$ can be chosen the same as in Theorem~\ref{localize}. Theorem~\ref{graph} also applies, where we obtain the projection property.

We will now show that the set $\tilde{\mathcal A}_{N_\epsilon, j}(c)$ depends only on the value of $N_\epsilon$ on the set $\{(\theta, p), \|\theta^s-\theta^s(p)\|\le \rho_0\}$, which will imply that the definition of $\tilde{\mathcal A}_{N_\epsilon, j}(c)$ is independent of decomposition or choice of the modification, since for all different decompositions and modifications, the Hamiltonian agree on this neighborhood. As before, we denote by $N^*_\epsilon(\theta, v, t)$ the Lagrangian corresponding to $N_\epsilon$ and ${N}^*_{\epsilon, j}$ the Lagrangian corresponding to $N^*_{\epsilon, j}$. The projected Aubry set ${\mathcal A}_{N_\epsilon, j}(c)$ is defined by the set of $\theta\in \T^n$ such that $h_{N^*_{\epsilon, j}, c}(\theta, \theta)=0$, where the subscript is added to stress the Lagrangian and cohomology class in the definition. It follows from the second statement of the proposition that any $\theta$ such that $h_{N^*_{\epsilon, j}, c}(\theta, \theta)=0$ must be contained in $\{\|\theta^s-\theta^s(c)\|\le \rho_2\}$. The following lemma implies independence of the local Aubry set on the docomposition or the choice of the modification.

\begin{lem}\label{local-aubry}
  Let $N_{\epsilon, j}=H_0 + \epsilon Z_j + \epsilon R$ and $\bar{N}_{\epsilon, j}=H_0 + \epsilon \bar{Z}_j + \epsilon \bar{R}$ be such that
  \begin{itemize}
  \item $N_{\epsilon, j}=\bar{N}_{\epsilon, j}$ for $\|\theta^s-\theta^s(p)\|\le \rho_0$.
  \item For $p^f\in [a_j-2b, a_{j+1}+ 2b]$, we have that ${Z}_j(\theta^s_j,p^s_*, p^f) - {Z}_j(\theta^s, p^s_*, p^f) \ge \frac{b}2 \|\theta^s-\theta^s_j\|^2$ and that  $\bar{Z}_j(\theta^s_j,p^s_*, p^f) - \bar{Z}_j(\theta^s, p^s_*, p^f) \ge \frac{b}2 \|\theta^s-\theta^s_j\|^2$ for all $\theta^s\in \T^{n-1}$.
  \item $\|R\|_{C^2}, \|\bar{R}\|_{C^2}\le \delta$.
  \end{itemize}
Then for sufficiently small $\epsilon$, $\delta$, and for $c=p_*(c^f)$ with $c^f\in [a_j-2b, a_{j+1}+2b]$
$$ h_{N^*_{\epsilon, j}, c}(\theta, \theta)=0 \Longleftrightarrow h_{\bar{N}^*_{\epsilon, j}, c}(\theta, \theta)=0. $$
\end{lem}
\begin{proof}[Proof of Lemma~\ref{local-aubry}]
  Let $\theta_0\in {\mathcal A}_{N_\epsilon, j}(c)$, we refer to Lemma~\ref{aubry-est} before and note that there exists an increasing sequence of integers $n_k$, $\theta_k=(\theta^s_k, \theta^f_k):\R \to M$ a sequence of $n_k-$periodic absolutely continuous curves such that $\theta_k(0)=\theta_0$ and
$$ \lim_{k\to \infty}\int_0^{n_k}N_{\epsilon, j}^*(t, \theta_k, \dot{\theta}_k)-c\cdot \dot\theta_k +\alpha_j(c) dt =0.$$
Moreover, the curves $\theta_k$ can be chosen to be minimizing, i.e. they minimizes the integral in the above displayed formula among the $n_k-$periodic absolutely continuous curves such that $\theta_k(0)=\theta_0$. In particular, $\theta_k|[0, n_k]$ must be trajectories of the Euler-Lagrange flow. Lemma~\ref{aubry-est} states that for sufficiently large $k$ we may assume that the whole curve $\theta_k$'s are contained in $\|\theta^s-\theta^s(c)\|\le \rho_0$ (choose a smaller $\delta$ if necessary).

Let $(\theta_k, p_k)$ be the corresponding Hamiltonian trajectory to $(\theta_k, \dot\theta_k)$, we will show that for $n_k> 1/\sqrt{\epsilon}$, $\|p_k(t)-c\|\le C\sqrt{\epsilon}$, where $C$ is a constant depending only on $A$ and $n$. Let $\tau\in [0, n_k]$ be where $\|p_k(t)\|$ takes its maximum. Consider a shift $\theta_k'(t)=\theta_k(t+\tau-1/\sqrt{\epsilon})$ of $\theta_k$, and let $p_k'$ be the corresponding action variable, then $\|p_k'\|$ reaches maximum at $t=1/\sqrt\epsilon$. We will write $T=1/\sqrt{\epsilon}$ in the rest of the proof. Similar to the proof of Proposition~\ref{scc}, we lift $\theta_k'$ to a curve  in $\R^n$ without changing its name, and define
 $$\theta_{k,x}'(t)=\theta_k(x)+tx/\tau. $$
We have the following
\begin{equation*}
\begin{aligned}
  &\int_0^{T}N_{\epsilon, j}^*(t, \theta_{k,x}', \dot{\theta}_{k,x}')-c\cdot \dot\theta_k +\alpha_j(c) dt - \int_0^{T}N_{\epsilon, j}^*(t, \theta_{k}', \dot{\theta}_{k}')-c\cdot \dot\theta_k' +\alpha_j(c) dt \\
  &\leq (-c+\partial_vN_{\epsilon, j}^*(\tau, \theta_k(\tau), \delta\theta_k(\tau)))\cdot x + 3A\sqrt{\epsilon}|x|^2
=  (-c + p_k'(T))\cdot x + 3A\sqrt{\epsilon}|x|^2,
\end{aligned}
\end{equation*}
the computation is identical to \eqref{action-cp} and the two formulas that follows it. Assume $\|p_k'(T)-c\|>0$ (otherwise there is nothing to prove), and we choose $x$ to be a unit integer vector that minimizes $x\cdot (-c + p_k'(T))$ among unit integer vectors. We have that there exists $C'>0$ depending on $n$ that $(-c + p_k'(T))\cdot x \leq -C' \|p_k'(T)-c\|$.  Since $\theta_k'(T)$ and $\theta_k(T)$ projects to the same point on the torus, by minimality of $\theta_k$ we have that  $$0 \leq  (-c + p_k(T))\cdot x + 3A\sqrt{\epsilon}|x|^2\leq -C' \|p_k'(T)-c\| + 3A\sqrt{\epsilon},$$
it follows that $\|p_k'(T)-c\|\le 3A/C' \sqrt{\epsilon}$. Choose $C=3A/C'$ and we have proved our claim.

To summarize, we have proved that for $n_k$ sufficiently large, the curves $(\theta_k, p_k)$ satisfy $\|\theta^s_k-\theta^s_j(c)\|\le \rho_2$ and $\|p_k-c\|\le C\sqrt{\epsilon}$. By choosing a sufficiently small $\epsilon$, we can guarantee that $\|\theta^s_k - \theta^s_j(p_k)\|< \rho_0$. This implies that the Hamiltonians $N_{\epsilon, k}$ and $\bar{N}_{\epsilon, k}$ take the same values on the curves $(\theta_k, p_k)$, by taking the Legendre transform, we can conclude that the Lagrangian $N^*_{\epsilon, k}$ and $\bar{N}^*_{\epsilon, k}$ must take the same values as well. It follows that
$$ 0= h_{N_{\epsilon, j}, c}(\theta_0, \theta_0)= \liminf_{k\to \infty}\int_0^{n_k}N_{\epsilon, j}^*(t, \theta_{k}, \dot{\theta}_{k}')-c\cdot \dot\theta_k +\alpha_j(c) dt\geq h_{\bar{N}_{\epsilon, j}, c}(\theta_0, \theta_0)\geq 0 . $$

Hence $h_{N_{\epsilon, j}, c}(\theta_0, \theta_0)=0 \Longrightarrow h_{\bar{N}_{\epsilon, j}, c}(\theta_0, \theta_0)=0$. The other direction also holds since the argument is completely symmetric.  This concludes the proof of the lemma.
\end{proof}

The alpha function of a Lagrangian $L$ can be defined by  $\alpha(c)=-\inf_{\mu}(L-c\cdot \dot\theta) d\mu$, where $\mu$ is taken over all invariant probability measures supported on the Aubry set $\tilde{\mathcal A}(c)$. Consider two Hamiltonians $N_{\epsilon, j}$ and $\bar{N}_{\epsilon, j}$ as before, since the Aubry sets are identical for these Hamiltonians, and the Hamiltonians coincide on a neighborhood of the Aubry sets, the alpha function $\alpha_j(c)$ defined for these Hamiltonians must also be the same. This conclude the proof of the first statement of our proposition.

We now prove statements 3 and 4. Consider the cohomology classes $c=p_*(c^f)$ with $c^f\in [a_j+b, a_{j+1}-b]$, we note that the function $Z$ already satisfies the conditions that we require of the modification, and since the local Aubry set is independent of specific modifications, we conclude that $\tilde{\mathcal A}_{N_\epsilon}(c)=\tilde{\mathcal A}_{N_\epsilon, j}(c)$ and $\alpha_j(c)=\alpha(c)$.

We now focus on the cohomology classes  $c=p_*(c^f)$ with $c^f\in [a_{j+1}-b, a_{j+1}+b]$. Using Theorem~\ref{localize}, for these cohomology classes the Aubry set $\tilde{\mathcal A}_{N_\epsilon}(c)$ is contained in the vertical neighborhood $\{\|p-c\|\le 36A\sqrt{\epsilon}\}$, and horizontally in the neighborhood $\{\|\theta^s-\theta^s_j(c)\|\le \rho_2\}\cup \{\|\theta^s-\theta^s_{j+1}(c)\|\le \rho_2\}$. Take a point $\theta_0\in {\mathcal A}_{N_\epsilon}(c) \cap \{\|\theta^s-\theta^s_j(c)\|\le \rho_2\}$, by going through the same argument as in the proof of Lemma~\ref{local-aubry}, we can conclude that $h_{N_\epsilon, c}(\theta_0, \theta_0)=0$ implies that $h_{N_{\epsilon, j}, c}(\theta_0, \theta_0)=0$. It follows that $\tilde{\mathcal A}_{N_\epsilon}(c)\cap U_j(c)\subset \tilde{\mathcal A}_{N_\epsilon, j}(c)$; the same holds for $j+1$. We have that
$$ \alpha(c)=- \min\{\inf_{\mu_1}(N^*_{\epsilon}-c\cdot \dot\theta)d\mu_1 ,\inf_{\mu_2}(N^*_{\epsilon}-c\cdot \dot\theta)d\mu_2\} \leq \max\{\alpha_j(c),\alpha_{j+1}(c)\}$$
where $\mu_1$ is supported on $\tilde{\mathcal A}_{N_\epsilon}(c)\cap U_j(c)$ while $\mu_2$ is supported on $\tilde{\mathcal A}_{N_\epsilon}(c)\cap U_{j+1}(c)$. On the other hand, since $\alpha(c)=-\inf_\mu (N^*_{\epsilon}-c\cdot \dot\theta)d\mu$ with $\mu$ taken over \emph{all} probability invariant measures, $\alpha(c)\geq \alpha_j(c), \alpha_{j+1}(c)$. We conclude that $\alpha(c)= \max\{\alpha_j(c), \alpha_{j+1}(c)\}$. We have proved statement 3.

Moreover, assume that ${\mathcal A}_{N_\epsilon}(c) \cap U_j(c)\ne \emptyset$, then there exists $\theta_0$ in this set such that $h_{N_\epsilon, c}(\theta_0, \theta_0)=0$, as well as a sequence of localized periodic curves $\theta_k$ converging to it. By taking any weak-*-limit of probability measures supported on these curves, we obtain at least one measure $\nu$ supported on  ${\mathcal A}_{N_\epsilon}(c)\cap U_j(c)$ such that $\alpha(c)=-\int (N^*_\epsilon -c\cdot \dot\theta)d\nu$. This implies that $\alpha(c)\le \alpha_j(c)$, hence $\alpha(c)=\alpha_j(c)$. As a conclusion, if $\alpha(c)\ne \alpha_j(c)$ then ${\mathcal A}_{N_\epsilon}(c) \cap U_j(c)= \emptyset$. This proves statement 4 and concludes the proof of Proposition~\ref{relative}.
\end{proof}

\subsection{Generic property of  $\tilde{\mathcal A}_{N_\epsilon}(c)$}

In this section we discuss the property of the sets $\tilde{\mathcal A}_{N_\epsilon}(c)$ for $c=(p^s_*(c^f))$ with $c^f\in [a_j-2b, a_{j+1}+2b]$ and their properties when we allowed to subject the Hamiltonian to an additional perturbation. It is convenient for us to fix a modified Hamiltonian $N_{\epsilon, j}$ and base all discussions on this system.

From Proposition~\ref{relative}, we have that the sets $\tilde{\mathcal A}_{N_\epsilon, j}(c)$ (we will write $\tilde{\mA}_j(c)$ for short in this section) are contained in the NHIC $X_j$, and $\pi_{\theta^f}|\tilde{\mA}_j(c)$ is one-to-one. We will study finer structures of the Aubry sets, by relating to the Aubry-Mather theory of two dimensional area preserving twist maps. We will prove the following statement.
\begin{prop}\label{unique}
  There exists $\epsilon_0,\delta_0>0$ such that for  $0<\epsilon<\epsilon_0$ and $0<\delta<\delta_0$, there exists arbitrarily small $C^r$ perturbation $N_\epsilon'$ of $N_\epsilon$, such that for each $c^f\in [a_j-2b, a_{j+1}+2b]$, $\tilde{\mA}_{N_\epsilon'}(p_*(c^f))$ supports a unique $c-$minimal measure.
\end{prop}

  We note that the time-one-map of the Hamiltonian flow is a twist map defined on $\T^n\times \R^n$. The generating function of this twist map is $G_j(x,x'):\R^n \times \R^n \to \R$ ,
$$ G_j(x, x')= \inf_{\gamma(0)=x, \gamma(1)=x'} \int_0^1N^*_{\epsilon, j}(t, \gamma, \dot\gamma)dt. $$
Consider an orbit $\{(\theta(t), p(t))\}$ of the Hamiltonian flow, its trajectory in the configuration space can be lifted to a curve $x(t)\in \R^n$, which is unique modulo integer translation. The sequence $x_k=x(k)$, $k\in \Z$ will be called a \emph{configuration}. A configuration's rotation number is defined by $\lim_{k\to \infty}(x_{a+k}-x_a)/k$, if such a limit exists.

Let $\{x_k\}=\{(x^s_k, x^f_k)\}$ be a configuration corresponding to an orbit in $\tilde{\mA}_j(c)$, we will say that this configuration belong to the Aubry set for short. Since $\tilde{\mA}_j(c)\subset X_j$, we have that the slow component $x^s$ stays bounded all the time. Take two configurations $\{x_k\}$ and $\{y_k\}$, we say that they intersect in the fast direction (in short, intersect, as this is the only type of intersection we will consider) if there exists an integer $m$ and indices $k_1, k_2$ such that $x^f_{k_1}> y^f_{k_1}+m$ and $x^f_{k_2}< y^f_{k_2}+m$. We have the following statements, analogous to the twist map case.

\begin{lem}\label{twist}
  \begin{enumerate}
  \item Any two distinct configurations $\{x_k\}$ and $\{y_k\}$ in $\tilde\mA_j(c)$ does not intersect.
  \item Any configuration $\{x_k\}$ in $\tilde\mA_j(c)$ has a uniquely defined rotation number $\rho=(0, \rho^f)$.
  \end{enumerate}
\end{lem}
\begin{proof}
For the first statement, we prove by contradiction. Assume that $x(t)$ and $y(t)$ are the lifts of two distinct trajectories such that  $\{x(k)\}$ and $\{y(k)\}$ intersect. It follows that there exists $m$ and $k_1$. $k_2$ such that $x^f(k_1)>y^f(k_1)+m$ and  $x^f(k_2)< y^f(k_2)+m$. It follows that there exists $\tau\in \R$ such that $x^f(\tau)=y^f(\tau)+m$. Let $\theta_1(t)$ and $\theta_2(t)$ be the projections of $x(t)$  and $y(t)$ to $\T^n$, we have that $\theta^f_1(\tau)=\theta^f_2(\tau)$. Assume that $p_i(t)$, $i=1,2$ are the corresponding action variables for trajectories $\theta_i$. Let $k\le \tau<k+1$, we have that $(\theta_i(k), p_i(k))\in \tilde{\mA}_j(c)$. From the graph theorem, we have that $(\theta_i(k), p_i(k))$ is a function of $\theta^f_i(k)$. Applying the flow, we have that $(\theta_i(t), p_i(t))$ is a function of $(\theta_i^f(t), t)$. It follows that $(\theta_1(\tau),  p_1(\tau))=(\theta_2(\tau), p_2(\tau))$, hence $(\theta_1(t), p_1(t))=(\theta_2(t), p_2(t))$ for all $t$, a contradiction.

For the second statement, since any trajectory from $\tilde{\mA}_j(c)$ must contained in $X_j$, we have that any lift $x(t)$ of such a trajectory must have its slow component uniformly bounded. Hence $\lim_{k\to \infty}x^s(k)/k=0$. It suffices to consider only $\{x_k^f\}$. Since $x^f$ is one-dimensional, most argument from the standard Aubry-Mather theory applies, once we establish the non-intersecting property. We refer to \cite{MF}, section 11, where existence of rotation number was proved under a weaker assumption (the Aubry crossing lemma).
\end{proof}

Let $\mu$ be a $c-$minimal measure for $N_{\epsilon, j}$, we know that it is necessarily supported on $\tilde{A}_j(c)$. The rotation number of $\mu$ is $\rho(\mu)\in H_1(\T^n, \R)\simeq \R^n$, defined by
$$ \int_{\T^n\times\R^n} \langle c, v \rangle d\mu(\theta, v) = \langle c, \rho(\mu)\rangle. $$
Using the no-intersection property (Lemma~\ref{twist}, 1), most of the statements we will be need follows from standard Aubry-Mather theory. Most of the arguments presented here are variations of those found in see \cite{MF}.
\begin{prop} For any $c=p_*(c^f)$, $c^f\in [a_j-2b, a_{j+1}+ 2b]$, the following hold.
  \begin{enumerate}
  \item All $c-$minimal measures supported on $\tilde{\mA}_j(c)$ have a common rotation number $\rho(c)=(0, \rho^f(c))$. Moreover, the function $\alpha_jp_*(c^f)$ as a function of $c^f$ is $C^1$.
  \item If $\rho^f(c)= p/q \in \Q$, written in lowest terms, then all minimal measures are supported on $q-$periodic orbits. These orbits corresponds to $(p,q)-$periodic configurations $\{x_k\}$ in the sense that $(x^s_{k+q}, x^f_{k+q}) = (x^s_k, x^f_k) + (0, p)$. Furthermore, they are the  minima of the functional
$$ \sum_{k=0}^{q-1}G_j(x_k, x_{k+1})$$
over the set of configurations that are $(p,q)-$periodic.
  \item If $\rho^f(c)\notin \Q$, then there is one unique $c-$minimal measure.
  \end{enumerate}
\end{prop}
\begin{proof}
  First we show that all the configurations on $\tilde{\mA}_j(c)$ has the same rotation number. To see this, consider any two configurations with different rotation numbers, since they must intersect, Lemma~\ref{twist} implies that they cannot both be contained in $\tilde{\mA}_j(c)$.

 We now look at the function $\alpha_jp_*(c^f)$. It is known that (see, e.g. \cite{Ma1}) $\alpha_j(c)$ is a convex function and any rotation number $\rho$ of a $c-$minimal measure is a subderivative of $\alpha_j$ at $c$. If for some $c$ the subderivative is unique, then $\alpha$ is differentiable at $c$. It follows $\alpha_j(p^s_*(c^f),c^f)$ is differentiable for each $c^f\in [a_j-2b, a_{j+1}+2b]$. The fact that it is $C^1$ follows from the following statement: let $f(x)$ be convex,  $x_n$ is a sequence that converges to $x_*$, $p_n$ is a subderivative of $f(x)$ at $x_n$ and $p_n$ converges to $p_*$, then $p_*$ is a subderivative of $f(x)$ at $x_*$. This concludes the proof of the first statement.

We now prove the second statement. Consider any configuration $\{x_k\}$ with rotation number $p/q$, we have that $x^f_{k+q}-x^f_k-p$ does not change sign for this configuration. Assume that it does, say $x^f_{k_1+q}-x^f_{k_1}-p>0$ and $x_{k_2+q}-x_{k_2}-p<0$, then the configurations $\{x_k\}$ and $x_{k + k_2-k_1}$ intersects, contradiction. On the other hand, since the rotation number is $p/q$, we have that $\lim_{k\to \infty} x^f_{k+q}-x^f_k-p = 0$. It follows that any $x_k$ such that $x^f_{k+q}-x_k^f-p\ne 0$ does not project to a point on the support of an invariant measure, since this point is not recurrent. By the same argument, we can show that $x^s_{k+q}-x^s_k=0$ for any point that projects to the support of an invariant measure.

We have proved that any point on the support of an invariant measure lifts to a configuration with $x_{k+q}-x_k=(0,p)$. Let $\mu$ be a $c-$minimal measure supported on $(\theta(k), p(k))$, $k=0, \cdots q-1$, and let $x_k$ be the corresponding configuration. Since
$$\int(N_{\epsilon, j}^*-c\cdot \dot\theta)d\mu = \int N_{\epsilon, j}^* d\mu + c \cdot \rho =  \sum_{k=0}^{q-1}G_j(x_k, x_{k+1}) + c\cdot \rho, $$
$\mu$ minimizes $\int(N_{\epsilon, j}^*-c\cdot \dot\theta)d\mu$ implies that $\{x_k\}$ minimizes $\sum_{k=0}^{q-1}G_j(x_k, x_{k+1})$.

For the irrational rotation number case, we refer to \cite{MF}, section 12. Consider $\tilde{\mA}_j(c)$ as a subset of $\T$ and the dynamics on this subset. It is proved that the system is semi-conjugate to a rigid rotation of irrational rotation number, and the semi-conjugacy is not one-to-one on at most countably many points. It follows that the dynamics on $\tilde{\mA}_j(c)$ has one unique invariant measure, since irrational rotation is uniquely ergodic.
\end{proof}

For irrational rotation numbers, we have that the corresponding minimal measure is unique. For rational rotation numbers, it is well known that for the twist map, generically there exists only one minimal periodic orbit of rotation number $p/q$. We have the same conclusions here. The following statement and Lemma~\ref{twist} imply Proposition~\ref{unique}.

\begin{prop}\label{rational}
  \begin{enumerate}
  \item By subjecting the generating function $G_j(x,x')$ to an arbitrarily small  $C^r$ perturbation, we have that for any rational rotation number $p/q$, there are exactly $q$ periodic configurations of type $(p,q)$. (In this case there exists a unique minimal periodic orbit with rotation number $p/q$.)
  \item The perturbation to $G_j$ in part 1 can be realized by an arbitrarily small $C^r$ perturbation to the Hamiltonian $N_{\epsilon, j}$, localized in the set $\{(\theta,p): \|\theta^s-\theta^s_j(p)\|< \rho_0\}$. As a result, this perturbation can be realized by a small perturbation to the original Hamiltonian $N_\epsilon$.
  \end{enumerate}
\end{prop}
\begin{proof}
  Let $\{x_k\}$ be a minimizing configuration of type $(p,q)$, let $U$ be an open set that contains $x_0$ but none of the $x_1, \cdots, x_{q-1}$. Let $g_U(x)$ be nonnegative periodic function that is supported on $U$,  $g_U(x_0)=0$ is the unique minimum and $\partial^2g$ is positive definite. If we consider the new generating function
$$ G_j(x,x') + g_U(x),$$
the action $\sum_{k=0}^{q-1}G_j(x_k,x_{k+1})$ is unaffected, while the action increases for other configurations. It follows that $\{x_k\}$ and its translations are the unique minimal configurations. However, this perturbation cannot be realized by a localized perturbation to the Hamiltonian (to be more precise, it is localized horizontally, but not vertically). We consider the following modification of the above construction.

Let $\Phi$ denote a lift of the time-one-map of the Hamiltonian flow. The generating function uniquely determines the map $\Phi$ in the sense that given  $x, x'\in \R^n$, write $p=-\partial_1 G_j$ and $p_2=\partial_2G_j$ then $\Phi(x,p)=(x',p')$. On the other hand, Theorem~\ref{localize} implies that any orbit in the Aubry set $\tilde{\mA}_j(c)$ is localized in the set $\{\|p-c\|\le 6nA\sqrt{\epsilon}\}$, which leads us to the following definition. Let $V_x(6A\sqrt{n\epsilon})=\{x'\in \R^n, \|\partial_1G_j(x,x')-c\|\le 6A\sqrt{n\epsilon}\}$, and let $\rho_x$ be a smooth function that takes value $1$ on $V_x(6A\sqrt{n\epsilon})$ and takes value $0$ on $V_x(12A\sqrt{n\epsilon})$. We have that the generating function
$$ G_j(x,x')+g_U(x)\rho_{x}(x')$$
will make $\{x_k\}$ and its translation the unique minimizing configurations of type $(p,q)$. The norm of the perturbation can be arbitrarily small since the norm of $g_U$ can be arbitrarily small.

To treat all rational rotation numbers, we consider a sequence of such perturbations $g_{U_i}(x)\rho_{x}(x')$, each subsequent perturbation can be chosen to be small enough, such that the result of earlier perturbations are not destroyed. The final perturbed generating function is
$$ G'_j(x,x')=G_j(x,x') + \sum_{i\ge 1}g_{U_i}(x)\rho_{x}(x'). $$

We now show that the perturbation can be realized by a localized perturbation of the Hamiltonian. Write $g(x,x')=\sum_{i\ge 1}g_{U_i}(x)\rho_{x}(x')$ and let $\Phi'$ denote the perturbed time-one-map of the Hamiltonian flow.  Since $\partial_1g=0$ for all $x\notin \bigcup_iU_i$ or $x'\notin V_x(12A\sqrt{n\epsilon})$,  the perturbed time-one-map  $\Phi'$ is identical to the original $\Phi$ for any $(x,p)\notin \bigcup_iU_i\times \{\|p-c\|\le 12A\sqrt{n\epsilon}\}$. Since we can choose $U_i$ such that $\bigcup_iU_i \subset \{\|\theta^s-\theta^s_j(c)\|<d<\rho_0\}$, for sufficiently small $\epsilon$ we can guarantee
  $$\bigcup_iU_i\times \{\|p-c\|\le 12A\sqrt{n\epsilon}\}\subset \{\|\theta^s-\theta^s_j(p)\|< \rho_0\}.$$
It follows that $\Phi=\Phi'$ for any $(\theta, p)\notin \{\|\theta^s-\theta^s_j(p)\|< \rho_0\}$. This perturbation of the time-one-map can be realized by a perturbation to the Hamiltonian localized in the same neighborhood.
\end{proof}

\subsection{Generic property of $\alpha(c)$ and proof of Theorem \ref{property-twist}}

After obtaining the desired properties for the local Aubry set, we now return to the Hamiltonian $N_\epsilon$. If $c^f\in [a_j+b, a_{j+1}-b]$, we have that $\tilde{\mA}_{N_\epsilon}(c)=\tilde{\mA}_{N_\epsilon, j}(c)$. For $c^f\in [a_{j+1}-b, a_{j+1}+b]$, Proposition~\ref{relative}, statement 3 and 4 shows that it suffices to identify whether $\alpha(c)$ is equal to $\alpha_j(c)$ or $\alpha_{j+1}(c)$.

\begin{prop}\label{transversal}
Assume that $N_\epsilon=H_0 + \epsilon Z + \epsilon R$ is such that $Z$ satisfy [G0]-[G2] and that $\|R\|_{C^2}\le \delta$. Then there exists $\epsilon_0, \delta_0>0$ such that for $0<\epsilon<\epsilon_0$ and $0<\delta<\delta_0$, there exists an arbitrarily small perturbation  $N_\epsilon'$ of $N_\epsilon$, with the following properties. For the Hamiltonian $N_\epsilon'$ Proposition~\ref{relative} and Proposition~\ref{unique} still hold, in addition, there exists only finitely many $c^f\in [a_{j+1}-b, a_{j+1}+b]$ such that $\alpha_j(p^s_*(c^f),c^f)=\alpha_{j+1}p_*(c^f)$.
\end{prop}
\begin{proof}
By taking a small perturbation if necessary, let us assume that we start with a Hamiltonian $N_\epsilon$ such that Proposition~\ref{relative} and Proposition~\ref{unique} already hold. Consider the interval $c^f\in [a_j-2b, a_{j+1}+2b]$ first.  Let $P_j^\eta(\theta, p, t):\T^n\times \R^{n-1}\times [a_j-2b, a_{j+1}+2b]\to \R$ be a family of smooth functions such that
$$ P_j^\eta(\theta, p, t)=
\begin{cases}
  \eta, & \|\theta^s-\theta^s_j(p^f)\|\leq \rho_0 \text{ and } p^f\in [a_j-3r/2, a_{j+1}+3r/2] \\
  0 , & \|\theta^s-\theta^s_j(p^f)\|\geq 4\rho_0/3   \\
  0, & p^f\in \{a_j-2b, a_{j+1} +2b\}
\end{cases}.
$$
Clearly $\|P_j^\eta\|_{C^r}$ can be arbitrarily close to $0$ by choosing $\eta$ close to $0$.

Let $N^\eta=N_\epsilon + P_j^\eta$. The new perturbation can be considered part of $R$ and if $\eta$ is sufficiently close to $0$, Proposition~\ref{relative} still hold. This implies that the local Aubry sets still depends only on the value of the Hamiltonian on the set $\|\theta^s-\theta^s_j(p^f)\|\leq \rho_0$, on which the perturbation is simply a constant (for $c^f\in [a_j-3r/2, a_{j+1}+3r/2]$). We have that $\alpha_{N^\eta, j}(c)=\alpha_{N_\epsilon, j}+ \eta$ and that $\tilde{\mA}_{N^\eta, j}(c)=\tilde{\mA}_{N_\epsilon, j}(c)$ for $c^f\in [a_j-3r/2, a_{j+1}+3r/2]$. It follows that all properties of the local Aubry set $\tilde{\mA}_{N_\epsilon, j}(c)$ is intact, while $\alpha_j(c)$ undergoes a shift.

On the other hand, Consider the functions $\alpha_jp_*(c^f)$ and $\alpha_{j+1}p_*(c^f)$ as functions on $[a_{j+1}-3r/2, a_{j+1}+3r/2]$. Since they are both $C^1$, by Sard's lemma, the critical values of $\alpha_j-\alpha_{j+1}$ has zero measure. It follows that there exists a full measure set of $\eta\in \R$ such that $\alpha'_j-\alpha'_{j+1}=0$ implies $\alpha_j-\alpha_{j+1}+\eta \ne 0$. In other words, the two functions $\alpha_j+\eta$ and $\alpha_{j+1}$ intersect transversally, which implies that there are only finitely many values where $\alpha_j-\alpha_{j+1}+\eta=0$.

We can perform this modification for each $[a_j-2b, a_{j+1}+2b]$, and $\eta$ can be chosen to be arbitrarily close to $0$.
\end{proof}

\begin{proof}[Proof of Theorem~\ref{property-twist}]
Since there are only finitely many $c^f\in [a_{j+1}-b, a_{j+1}+b]$ on which $\alpha_j=\alpha_{j+1}$, we add these points to the set $\{a_0, \cdots, a_s\}$ to form a new partition $\{[\bar{a}_j, \bar{a}_{j+1}]\}$. On each open interval $(\bar{a}_j, \bar{a}_{j+1})$ $\alpha(c)$ is only equal to one of the $\alpha_j$ and $\alpha_{j+1}$. Use Proposition~\ref{relative} and the first statement follows.

The second statement follows from Proposition~\ref{unique}.
\end{proof}

\subsection{nondegeneracy of the barrier functions}
In this section we prove Theorem~\ref{nondegeneracy}. We have concluded that in order to prove Theorem~\ref{nondegeneracy}, it suffices to show that $\Gamma_1(\epsilon)=\Gamma_1^*(\epsilon)$ and $\Gamma_2(\epsilon)=\Gamma_2^*(\epsilon)$. We show that this is the case by proving the following equivalent statement.

\begin{prop}\label{property-barrier}
 Let $H_\epsilon'$ be a perturbation of $H_\epsilon$ such that the conclusions of Theorem~\ref{orignal-H} holds, then there exists an arbitrarily small $C^r$ perturbation $H_\epsilon''$ to $H'_\epsilon$, such that for the Hamiltonian $H''_\epsilon$ Theorem~\ref{orignal-H} still hold, in addition:
  \begin{enumerate}
  \item Consider $c^f\in (\bar{a}_j, \bar{a}_{j+1})$ such that ${\mathcal A}_c={\mathcal N}_c$ and $\pi_{\theta^f}{\mathcal A}_c=\T$.  Take $\zeta\in {\mathcal M}_c$, and let $\zeta_1$ and $\zeta_2$ be its lifts to the double cover. We have that the functions $\tilde{h}_c(\zeta_1,\theta)+\tilde{h}_c(\theta, \zeta_2)$ and $\tilde{h}_c(\zeta_2,\theta)+\tilde{h}_c(\theta, \zeta_1)$ have isolated minima outside of the lifts of ${\mathcal A}_c$.
  \item For $c=\bar{a}_{j+1}$, take $\zeta\in {\mathcal A}_c\cap X_j$ and $\eta \in {\mathcal A}_c\cap X_{j+1}$. We have that both $h_c(\zeta,\theta)+h_c(\theta, \eta)$ and $h_c(\eta,\theta)+h_c(\theta, \zeta)$ has isolated minima outside of ${\mathcal A}_c$.
  \end{enumerate}
\end{prop}

This proposition is essentially proved by Cheng and Yan in \cite{CY2}, here we briefly describe their approach.

Consider the Hamiltonian $H_\epsilon'$ such that conclusions of Theorem~\ref{orignal-H} holds. In the rest of the section, let's refer to $H_\epsilon'$ simply as $H$.  For now, let us also fix an interval $(\bar{a}_j, \bar{a}_{j+1})$ and consider only cohomology classes with $c^f$ in that interval. Let $\Gamma_1^j=\Gamma_1(\epsilon)\cap \{c, c^f\in (\bar{a}_j, \bar{a}_{j+1})\}$, we would like to show by perturbing the Hamiltonian, we can make the functions $\tilde{h}_c(\xi_1,\theta)+\tilde{h}_c(\theta, \xi_2)$ and $\tilde{h}_c(\xi_2,\theta)+\tilde{h}_c(\theta, \xi_1)$ nondegenerate.

Recall that $\tilde{h}_c$ is the barrier function defined on the covering space $(2\T)^n \times \R^n$, and $\xi:(2\T)^n \times \R^n \to \T^n\times \R^n$ is the covering map.  $\tilde{H}$ is the Hamiltonian lifted to the covering space.

Define the generating function $G(x,x'):\R^n\times\R^n\to \R$ by
$$ G(x, x')=\inf_{\gamma(0)=x, \gamma(1)=x'}\int_0^1L(t, \gamma, \dot\gamma), $$
where $L$ is the Lagrangian corresponding to $H$. A convenient way of introducing perturbations to the functions $\tilde{h}_c$ is by perturbing the generating functions. Denote by $\pi:\R^n\to \T^n$ the standard projection.

We consider the following perturbation
$$ G'(x,x')=G(x,x')+ G_1(x')$$
and denote by $\tilde{h}'_c$ the corresponding perturbed barrier function. We have the following statement.
\begin{lem}\label{pert-barrier}(\cite{CY2}, Lemma 7.1)
  For $c=p_*(c^f)$ with $c^f\in (\bar{a}_j, \bar{a}_{j+1})$, the following hold.
  \begin{enumerate}
  \item There exists a family of open sets $O_c\subset (2\T)^n$ such that the full orbit of any $(\tilde{\theta}, \tilde{p})\in \tilde{\mN}_{\tilde{H}}(c)\setminus \tilde{\mA}_{\tilde{H}}(c)$ must intersect $O_c$ in the $\tilde{\theta}$ component.
  \item There exists $\rho>0$ such that if we perturb $G(x,x')$ by function $G_1(x')$ with $\supp G_1\subset B_\rho(u)$, where $B_\rho(u)$ is the ball of radius $\rho$ centered at $u$, then for each $c$ such that $B_\rho(u)\subset \pi^{-1}O_c$ the corresponding barrier function
$$ \tilde{h}_c'(\xi_1, \theta) + \tilde{h}_c'(\theta, \xi_2) = \tilde{h}_c(\xi_1, \theta) + \tilde{h}_c(\theta, \xi_2) + G_1(\theta)   $$
for each $\theta\in O_c$.
  \item $\xi O_c\cap \{\theta: \|\theta^s-\theta^s_j(c)\|\le \rho_2\}=\emptyset$, in particular, $\xi O_c\cap \mN_{H}(c)=\emptyset$. Moreover $\tilde{U}=\bigcup_{c^f\in (\bar{a}_j, \bar{a}_{j+1})}O_c$ is an open set.
  \end{enumerate}
\end{lem}

As before, let us write $\tilde{b}^+_c(\theta)=\tilde{h}_c(\xi_1, \theta) + \tilde{h}_c(\theta, \xi_2)$ and $\tilde{b}^-_c(\theta)=\tilde{h}_c(\xi_2, \theta) + \tilde{h}_c(\theta, \xi_1)$. Elements of ${\mN}_{\tilde{H}}(c) \setminus {\mA}_{\tilde{H}}(c)$ coincide with the minimal set of the functions $\tilde{b}^\pm_c$. To prove that this set is isolated, it suffices to prove its intersection with $O_c$ is isolated, as any accumulation point of ${\mN}_{\tilde{H}}(c) \setminus {\mA}_{\tilde{H}}(c)$ has a diffeomorphic image in $O_c$. We say that the function $\tilde{b}^\pm_c(\theta)$ is degenerate if its minimal set has at least one accumulation point. Cheng and Yan proved that it is possible to introduce a perturbation to make $\tilde{b}^\pm_c$ nondegenerate for all $c\in \Gamma_1^j$ simultaneously.

This is not possible in general, if the functions $\tilde{b}^\pm_c$ behave badly as $c$ varies. Since regularity of $\tilde{b}_c^\pm$ in $c$ is hard to prove, Cheng and Yan resolves this problem nicely by introducing an additional parameter.  Recall that for each $c\in \Gamma_1^j$, the Aubry set $\tilde\mA$ is an invariant curve on the time-zero section of the invariant cylinder $X_j$, call it $\gamma_c$. Fix an arbitrary curve $\gamma_0=\{p^f=p^f_0\}\cap \{t=0\}\cap X_j$, we introduce a parameter $\sigma$ which is the area between $\gamma_c$ and $\gamma_0$ on the cylinder $X_j\cap \{t=0\}$. $\sigma$ is monotone in $c^f$ and is only defined for $c\in \Gamma_1^j$. Cheng and Yan proved that
\begin{lem}\label{holder}(\cite{CY2}, Lemma 6.4)
There exists constant $C>0$ such that, for $\sigma$ and $\sigma'$ such that $c(\sigma), c(\sigma')\in \Gamma_1^j$, $\zeta\in \mA_{H}(c)$ and $m\notin \{\|\theta^s-\theta^s_j(c(\sigma))\|\le \rho_2\}\cup \{\|\theta^s-\theta^s_j(c(\sigma'))\|\le \rho_2\}$,
  $$|h_{c(\sigma)}(\zeta, m)-h_{c(\sigma')}(\zeta, m)|\le C(\sqrt{|\sigma-\sigma'|}+|c(\sigma)-c(\sigma')|),$$
  $$|h_{c(\sigma)}( m, \zeta)-h_{c(\sigma')}(m, \zeta)|\le C(\sqrt{|\sigma-\sigma'|}+|c(\sigma)-c(\sigma')|).$$
\end{lem}
It follows that the function $h_{c(\sigma)}$ can be extended to $h_{c,\sigma}$ that is $\frac12-$H\"{o}lder in $c$ and $\sigma$, this regularity turns out to be enough. To see how this is carried out, let us consider a subset $B_{d'}(c^f_*)\times R_d(u)$, where $B_{d'}(c^f_*)=\{c: |c^f-c^f_*|<d'\}$  and $R_d(u)\subset \T^n$ is an open cube centered at $u$ with edge $d$.
\begin{lem}(\cite{CY2}, Lemma 7.2)
  There is a residue set of functions $G_1\in C^r_0(R_d(u),\R)$ such that
$$\tilde{b}^\pm_c(\theta) + G_1(\theta)$$
has isolated minima in $R_d(u)$ for each $c\in \Gamma_1^j\cap B_{d'}(c)$.
($C^r_0$ stands for $C^r$ functions with compact support).
\end{lem}
\begin{rmk}
  The nontrivial part of this statement is that the nondegeneracy of $\tilde{b}_c^\pm$ can be achieved for all $c\in  \Gamma_1^j\cap B_{d'}(c)$ simultaneously. The regularity acquired in Lemma~\ref{holder} is crucial to the proof. We refer to \cite{CY2} for details.
\end{rmk}
To construct the desired perturbation to the barrier function, let us state another lemma, which is a consequence of the upper semi-continuity of the Ma\~{n}e set on the Lagrangian.
\begin{lem}\label{upper-semi}
  The property that the functions $\tilde{b}^\pm_c$ are non-degenerate on the set $B_{d'}(c^f_*)\times R_d(u)$ survives under sufficiently small perturbation.
\end{lem}

We proceed to prove Proposition~\ref{property-barrier}. Let $B_{d'_i}(c^f_i)\times R_{d_i}(u_i)\subset \tilde{U}$, be a sequence of sets such that $\tilde{U}=\bigcup_{i}B_{d'_i}(c^f_i)\times R_{d_i}(u_i)$. We may choose a sequence of perturbations $G_i:R_{d_i}(u_i)\to \R$, and let $G'_k(x,x')=G(x,x')+\sum_{i=1}^kG_i(x')$ and $\tilde{b}^\pm_{c, k}$ be the corresponding barrier functions corresponding to the generating function $G_k'$. We can choose the sequence $G_i$ inductively such that $\tilde{b}^\pm_c$ is non-degenerate on $(c,\theta)\in \bigcup_{i=1}^k(B_{d'_i}(c^f_i)\cap \Gamma_1^j)\times R_{d_i}(u_i)$, because new perturbations can be added that does not disrupt the nondegeneracy already established in the previous steps.  By repeat this process for each interval $(\bar{a}_j, \bar{a}_{j+1})$, we have constructed a perturbation to the generating function $G$, such that the first statement of Proposition~\ref{property-barrier} holds.

For the second statement, using the same arguments for Lemma~\ref{pert-barrier}, one can show that the same type of conclusions apply to $b^\pm_c$ as well.
\begin{lem}
 For each $c=p_*(c^f)$ with $c^f=\bar{a}_j$, $j=2, \cdots, l-1$ the following hold.
  \begin{enumerate}
  \item There exists a family of open sets $O_c\subset (2\T)^n$ such that the full orbit of any $(\tilde{\theta}, \tilde{p})\in \tilde{\mN}_{H}(c)\setminus \tilde{\mA}_{H}(c)$ must intersect $O_c$ in the $\tilde{\theta}$ component.
  \item There exists $b>0$ such that if we perturb $G(x,x')$ by function $G_1(x')$ with $\supp G_1\subset B_b(u)$, where $B_b(u)$ is the ball of radius $b$ centered at $u$, then for each $c$ such that $B_b(u)\subset \pi^{-1}O_c$ the corresponding barrier function
$$ h_c'(\zeta,\theta)+h_c'(\theta, \eta)= h_c(\zeta,\theta)+h_c(\theta, \eta) + G_1(\theta)  $$
for each $\theta\in O_c$. The same conclusion holds for $h_c(\eta,\theta)+h_c(\theta, \zeta)$.
\end{enumerate}
\end{lem}
 For a fixed $c$, it is easy to see $b^\pm_c(\theta) + G_1(\theta)$ has isolated minimal set in $R_d(u)$ for an open and dense set of $G_1$. Repeat the arguments for $\tilde{b}_c^\pm$, we obtain a perturbation for which the both statements of Proposition~\ref{property-barrier} hold.
\qed

\appendix

\section{Generic conditions}
\label{sec:generic}

We prove Theorem~\ref{open-dense} in this section. Consider the following (degeneracy) conditions on the function $Z(\theta^s, p_*(p^f)):\T^{n-1}\times [a_{min}, a_{max}]\to \R$.
\begin{enumerate}
\item[{[T0]}] For  $p^f\in [a_{min}, a_{max}]$, all local maxima of $Z(\theta^s, p_*(p^f))$ is nondegenerate.
\item[{[T1]}] For each $p^f\in [a_{min}, a_{max}]$ and there are at most two distinct $\theta^s_1, \theta^s_2\in \T^{n-1}$ such that $\partial_{\theta^s}Z(\theta^s_j, p_*(p^f))=0$ for $j=1, 2$ and that $Z(\theta^s_1, p_*(p^f))=Z(\theta^s_2, p_*(p^f))$.
\item[{[T2]}] For any $p^f\in [a_{min}, a_{max}]$ and distinct $\theta_1^s, \theta^s_2\in \T^{n-1}$ such that $\partial_{\theta^s}Z(\theta^s_j, p_*(p^f))=0$ for $j=1,2$, we have that
$$\partial_{p^f}Z(\theta_1^s, p_*(p^f))\ne \partial_{p^f}Z(\theta_2^s, p_*(p^f)).$$
\end{enumerate}
Let ${\mathcal U}'$ denote the set of functions in $S_r$ that satisfies one or more of the conditions [T0]-[T2].

\begin{prop}\label{nowheredense}
${\mathcal U}'$ is open and dense.
\end{prop}

\begin{proof}[Proof of Theorem~\ref{open-dense}]

The set ${\mathcal U}$ is open, since if $H_1$ satisfies conditions [G0]-[G2] with some $\lambda>0$, any $H_1'$ sufficiently close to $H_1$ in $C^r$ norm  satisfies these conditions with a slightly smaller $\lambda'>0$.

We now prove that ${\mathcal U}$ is dense by showing that ${\mathcal U}\supset {\mathcal U}'$. The conditions [T0]-[T2] implies the statement that any $p^f$ is either a nondegenerate regular point or a nondegenerate bifurcation point, and that there are at most finitely many bifurcation points. To see that [T0]-[T2] also imply conditions [G0]-[G2], let $\{[a_j, a_{j+1}]\}_{j-0}^{s-1}$ be the partition of $[a_{min}, a_{max}]$ by bifurcation points. Each $p^f\in (a_j, a_{j+1})$ defines a unique global maximum $\theta^s_j(p^f)$. The function $\theta^s(p^f)$ is continuous since any converging sequence $\theta^s(p^f_k)$ also converges to a global maximum, and it must be smooth by implicit function theorem. The function extends to $[a_j, a_{j+1}]$ by continuity, and using the nondegeneracy of the maximum and implicit function theorem, we can extend $\theta^s_j$ smoothly to  the interval $[a_j-d, a_{j+1}+d]$, such that each $\theta^s_j(p^f)$ is a nondegenerate local maxima. Assume that for each $p^f\in [a_j-d, a_{j+1}+d]$ we have $-\partial^2_{\theta^s \theta^s} Z(\theta^s, p_*(p^f))\ge d' I$ as a quadratic form, hence $Z$ satisfies [G0] with  $\lambda=\min\{d, d'\}$. [G1] and [G2] are direct consequences of [T0]-[T2].
\end{proof}

\section{Normally hyperbolic manifold}
\label{sec:abstract-nhic}

Let $F:\Rm^n\lto \Rm^n$
be a $C^1$ vector field. We give sufficient conditions
for the existence of a Normally hyperbolic invariant graph of $F$.
We split the space $\Rm^n$ as $\Rm^{n_u}\times\Rm^{n_s}\times \Rm^{n_c}$,
and denote by $x=(u,s,c)$ the points of $\Rm^n$. We denote by
$(F_u,F_s,F_c)$ the components of $F$:
$$
F(x)=(F_u(x),F_s(x),F_c(x)).
$$
We study the flow of $F$ in the domain
$$\Omega=B^u\times B^s\times \Omega^c
$$
where $B^u$ and $B^s$ are  the open Euclidean balls of radius $r_u$  and $r_s$ in $\Rm^{n_u}$ and $\Rm^{n_s}$, and $\Omega^c$ is a convex open subset of
$\Rm^{n_c}$.
We denote by
$$L(x)=dF(x)=\begin{bmatrix}
L_{uu}(x)&L_{us}(x)&L_{uc}(x)\\L_{su}(x)&L_{ss}(x)&L_{sc}(x)\\L_{cu}(x) &L_{cs}(x)&
L_{cc}(x)
\end{bmatrix}
$$
the linearized vector field at point $x$. We  assume that $\|L(x)\|$ is bounded
on $\Omega$, which implies that each trajectory of $F$ is defined until it leaves $\Omega$.
We denote by $W^c$ the union of full orbits contained in $\Omega$.
In other words, this is the set of initial conditions $x\in \Omega$
such that there exists a solution $x(t):\Rm\lto \Omega$
of the equation $\dot x =F(x)$ satisfying $x(0)=0$.
We denote by $W^{sc}$ the set of points whose positive  orbit remains inside
$\Omega$.
In other words, this is the set of initial conditions $x\in \Omega$
such that there exists a solution $x(t):[0,\infty)\lto \Omega$
of the equation $\dot x =F(x)$ satisfying $x(0)=0$.
Finally, we denote by $W^{uc}$
the set of points whose negative  orbit remains inside
$\Omega$.
In other words, this is the set of initial conditions $x\in \Omega$
such that there exists a solution $x(t):(\infty,0]\lto \Omega$
of the equation $\dot x =F(x)$ satisfying $x(0)=0$.
These sets have specific features under the following assumptions:

\begin{hyp}[Isolating block]\label{block}
We have:
\begin{itemize}
\item
$F_c=0$ on $B^u\times B^s\times \partial\Omega^c$.
\item $F_u(u,s,c)\cdot u> 0$ on  $\partial B^u \times \bar B^s \times \bar \Omega^c$.
\item $F_s(u,s,c)\cdot s< 0$ on  $ \bar B^u \times \partial B^s \times \bar \Omega^c$.
\end{itemize}
\end{hyp}
\begin{hyp}\label{cone}
There exist  positive constants $\alpha$, $m$ and $M$  such that:
\begin{itemize}
\item
$L_{uu}(x)\geq \alpha I , \quad L_{ss}(x)\leq -\alpha I
$ for each $x\in \Omega$
in the sense of quadratic forms.
\item$
\|L_{us}(x)\|+\|L_{uc}(x)\|+\|L_{su}(x)\|+\|L_{sc}(x)\|+
\|L_{cu}(x)\|+\|L_{cs}(x)\|+\|L_{cc}(x)\|\leq m
$ for each $x\in \Omega$.
\end{itemize}
\end{hyp}

\begin{thm}\label{abstractNHI}
Assume that  Hypotheses \ref{block} and \ref{cone} hold, and that
$$
K:= \frac{m}{\alpha-2m}\leq \frac{1}{\sqrt{2}}.
$$
Then the set $W^{sc}$ is  the graph of a $C^1$ function
$$w^{sc}:B^s \times \Omega^c \lto B^u,
$$
 the set $W^{uc}$ is  the graph of a $C^1$ function
$$w^{uc}:B^u \times \Omega^c \lto B^s,
$$
and  the set $W^c$ is the graph of a $C^1$ function
$$w^c=(w^c_u,w^c_s):\Omega^c\lto B^u\times B^s.
$$
Moreover, we have the estimates
$$\|dw^{sc}\|\leq K,\quad \|dw^{uc}\|\leq K,
\quad \|dw^c\|\leq 2K.
$$
\end{thm}
\proof
This results could be reduced to several already existing ones,
see \cite{Fe, HPS,McG,Ch}
or proved directly by well-known methods. We shall use Theorem 1.1 in \cite{Ya}
which is the closest to our needs because it is expressed  in terms
of vector fields.
We first derive some conclusions from the isolating block conditions.
We denote by $\pi^{sc}$ the projection $(u,s,c)\lmto (s,c)$, and so on.

\begin{lem}\label{h3}
If Hypothesis \ref{block} holds, then
$$
\pi^{uc}(W^{uc})=B^u\times \Omega^c \quad
\text{and} \quad
\pi^{sc}(W^{sc})=B^s\times \Omega^c.
$$
Moreover, the closures of $W^{sc}$ and $W^{uc}$
satisfy
$$
\bar W^{sc}\subset B^s\times \bar B^c\times \bar \Omega^c,
\quad
\bar W^{uc}\subset \bar B^s\times  B^c\times \bar \Omega^c.
$$
\end{lem}
\proof
Let us define $T^+(x)\in [0, \infty]$ as the first positive time
where the orbit of $x$ hits the boundary $\partial \Omega$.
Let us denote by $\varphi(t,x)$ the flow of $F$.
If $T(x)<\infty$,  we have
$\varphi(T(x),x)\in \partial B^u\times B^s\times \Omega$,
as follows from Hypothesis \ref{block}.
Then, it is easy to check that the function $T$ is continuous, and even $C^1$,
at $x$.

We  prove the first equality of the Lemma  by contradiction,
and assume that there exists a point $(u,c)\in B^u\times \Omega^c$
such that $W^{uc}$ does not intersect the disc $\{u\}\times B^s \times \{c\}$.
Then, the first exit map
$$B^s\ni s\lmto \varphi(T(x),x)\in \partial B^s
$$
extends by continuity to a continuous retraction
from $\bar B^s$ to its boundary $\partial B^s$.
 Such a retraction does not exist.
The proof of the other equality is similar.

Finally, we have
$$
\bar W^{uc} \subset \bar B^u\times \bar B^s \times \bar \Omega ^c
=\big( B^u\times \bar B^s \times \bar \Omega ^c\big)\bigcup
\big( \partial B^u\times \bar B^s \times \bar \Omega ^c\big).
$$
Hypothesis \ref{block} implies that each point of
$\partial B^u\times \bar B^s \times \bar \Omega ^c$ has a neighborhood
formed of points which leave $\Omega$ after a small time.
As a consequence, the set $\partial B^u\times \bar B^s \times \bar \Omega ^c$
can't intersect $\bar W^{uc}$, and we have proved that
$
\bar W^{uc} \subset B^u \times \bar B^s \times \bar \Omega ^c.
$
The other inclusion can be proved in a similar way.
\qed

In order to prove the statement of the Theorem  concerning $W^{sc}$, we apply Theorem 1.1 of \cite{Ya}.
More precisely, using the notation of that paper, we set
$$
a=u/K, \quad
z=(s,c),\quad f(a,z)=F_u(Ka,z)/K,\quad g(a,z)=(F_s(Ka,z),F_c(Ka,z)).
$$
We have the estimates
$$
\partial_af=L_{uu}\geq \alpha, \quad
\partial_z g=
\begin{bmatrix}
L_{ss}&L_{sc}\\L_{cs}&L_{cc}
\end{bmatrix}\leq m
$$
in the sense of quadratic forms.
Moreover, we have the estimates
$$
\|\partial_z f\|\leq \frac{m}{K},\quad \|\partial_ag\|\leq Km.
$$
Since
$$
m+m/K+Km<2m+m/K=\alpha
$$
we conclude that Hypothesis 2 of \cite{Ya} is satisfied.
Hypothesis 1 of \cite{Ya} is verified by the domain $\Omega$,
and Hypothesis 3 is precisely the conclusion of Lemma \ref{h3}.
As a consequence, we can apply Theorem 1.1 of \cite{Ya},
and conclude that the set $W^{sc}$ is the graph of a $C^1$ and $1$-Lipschitz
map above $B^s\times \Omega^c$
in $(a,z)$ coordinates, and therefore the graph of a $K$-Lipschitz
$C^1$ map $w^{sc}:B^s\times \Omega^c\lto B^u$ in  $(u,s,c)$ coordinates.

In order to prove the statement concerning $W^{uc}$,
we apply Theorem 1.1 of \cite{Ya} with
$$
a=s/K, \quad
z=(u,c),\qquad \qquad $$
$$
f(a,z)=-F_s(Ka,z)/K,\quad g(a,z)=-(F_u(Ka,z),F_c(Ka,z)).
$$
It is easy to check as above that all hypotheses are satisfied.

Let us now study the set $W^c=W^{sc}\cap W^{uc}$.
First, let us prove that $W^c$ is a $C^1$ graph above $\Omega^c$.
We know that $W^{sc}$ is the graph of a $K$-Lipshitz $C^1$ function
$w^{sc}(s,c)$ and that  $W^{uc}$
is the graph of a $K$-Lipshitz $C^1$ function
$w^{uc}(u,c)$. The point $(u,s,c)$ belongs to $W^c$
if and only if
$$
u=w^{sc}(s,c) \quad \text{and}\quad s=w^{uc}(u,c),
$$
or in other words if and only if $(u,s)$ is a fixed point
of the $K$-Lipschitz $C^1$ map
$$
(u,s)\lmto (w^{sc}(s,c),w^{uc}(u,c)).
$$
For each $c$, this contracting  map has a unique  fixed point in
$\bar B^u\times \bar B^s$, which corresponds to a point
of $\bar W^{sc}\cap \bar W^{uc}$. It follows from  Lemma \ref{h3}
that this point is contained in $B^u\times B^s$.
Then,  it depends in a $C^1$ way of the
parameter $c$. We have proved that $W^c$ is the graph
of a $C^1$ function $w^c$.
In order to estimate the Lipschitz constant of this graph,
we consider two points $(u_i,s_i,c_i), i=0,1$ in $\Gamma$.
We have
$$
\|u_1-u_0\|^2\leq K^2 (\|s_1-s_0\|^2+\|c_1-c_0\|^2)
$$
and
$$
\|s_1-s_0\|^2\leq K^2 (\|u_1-u_0\|^2+\|c_1-c_0\|^2).
$$
Taking the sum gives
$$
(1-K^2)(\|u_1-u_0\|^2+\|s_1-s_0\|^2)\leq 2 K^2\|c_1-c_0\|^2
$$
and
$$
\|(u_1,s_1)-(u_0,s_0)\|\leq \sqrt{\frac{2K^2}{1-K^2}}\|c_1-c_0\|
\leq 2K\|c_1-c_0\|,
$$
since $K\leq 1/\sqrt{2}$.
We conclude that $w^c$ is $2K$-Lipschitz.
\qed

We need an addendum for applications:
\begin{prop}\label{translation}
Assume  in addition that there exists a translation  $g$
of $\Rm^{n_c}$
such that
$$
g(\Omega^c)=\Omega^c\quad \text{and} \quad
F\circ (id\otimes id\otimes g)=F.
$$
 Then we have
$$
w^{sc}\circ(id\otimes g)=w^{sc},\quad
w^{uc}\circ(id\otimes g)=w^{uc}, \quad
w^c\circ g=w^c.
$$
\end{prop}
\proof
It follows immediately from the definition of the sets
$W^{sc}$, $W^{uc}$ and $W^c$ that $g(W^{sc})=W^{sc}$,
$g(W^{uc})=W^{uc}$ and $g(W^c)=W^c$.
\qed

In applications the first condition of Hypothesis \ref{block}
is usually not satisfied, except in the case where
$\Omega^c=\Rm^{n_c}$.
It is thus useful to state a more "applicable" variant of the result.
In view of the applications we have in mind,
it is useful to split the central variables into two groups
and consider the case
$$
\Omega^c=\Rm^{n^1_c}\times \Omega^{c_2},
$$
where $\Omega^{c_2}$ is a convex open set in $\Rm^{n^2_c}$,
$n^1_c+n^2_c=n_c$.
Given a positive parameter ${\sigma}$, let $\Omega^{c_2}_{\sigma}$ be the set of points
$c_2\in \Rm^{n^2_c}$ such that $d(c,\Omega^{c_2})<{\sigma}$.
This is a convex open subset of $\Rm^{n^2_c}$ containing $\Omega^{c_2}$.
We denote by $\Omega^c_{\sigma}$ the product $\Rm^{n^1_c}\times \Omega^{c_2}_{\sigma}$
and by
 $\Omega_{\sigma}$ the product $B^u\times B^s\times \Omega^c_{\sigma}$.
With the notation $F_c=(F_{c_1},F_{c_2})$, we have:

\begin{prop}\label{realNHI}
Assume that there exists $\lambda,m,\sigma>0$ such that
\begin{itemize}
\item $F_u(u,s,c)\cdot u> 0$ on  $\partial B^u \times \bar B^s \times \bar \Omega^c_{\sigma}$.
\item $F_s(u,s,c)\cdot s< 0$ on  $ \bar B^u \times \partial B^s \times \bar \Omega^c_{\sigma}$.
\item
$L_{uu}(x)\geq \alpha I , \quad L_{ss}(x)\leq -\alpha I$ for each $x\in \Omega_{\sigma}$
in the sense of quadratic forms.
\item
$
\|L_{us}(x)\|+\|L_{uc}(x)\|+\|L_{ss}(x)\|+\|L_{sc}(x)\|
\|L_{cu}(x)\|+\|L_{cs}(x)\|+\|L_{cc}(x)\|+
2\|F_{c_2}(x)\|/{\sigma}\leq m
$ for each $x\in \Omega_{\sigma}$.
\end{itemize}
Assume furthermore that
$$
K:= \frac{m}{\alpha-2m}\leq \frac{1}{\sqrt{2}}.
$$
Then there exists a $C^1$ function $\rho:\Omega^c_{\sigma}\lto[0,1]$
which is equal to $1$ on $\Omega^c$ and
such that the vector field
$$
\tilde F(u,s,c):=(F_u(u,s,c_1,c_2),F_s(u,s,c_1,c_2),
F_{c_1}(u,s,c_1,c_2),
\rho(c_2)F_{c_2}(u,s,c_1,c_2))
$$
satisfies all the hypotheses of Theorem \ref{abstractNHI} on $\Omega_{\sigma}$.
Note that $\tilde F=F$ on $\Omega$.
\end{prop}

\proof
We take a function $\rho:\Omega^{c_2}_{\sigma}\lto [0,1]$ such that :
\begin{itemize}
\item $\rho=0$ near the boundary of $\Omega^{c_2}_{\sigma}$,
\item $\rho=1$ on $\Omega^{c_{\sigma}}$,
\item $\|d\rho\|\leq 2/{\sigma}$ uniformly.
\end{itemize}
Denoting by $\tilde L_{**}$ the variational matrix associated to $\tilde F$,
we see that
$$
\tilde L_{cu}(u,s,c)=\rho(c_2) L_{cu}(u,s,c),\quad
 \tilde L_{cs}(u,s,c)=\rho(c_2) L_{cs}(u,s,c),
$$
$$
\tilde L_{c_1c_1}(u,s,c)=\rho(c_2) L_{c_1c_1}(u,s,c),\quad
 \tilde L_{c_1c_2}(u,s,c)=\rho(c_2) L_{c_1c_2}(u,s,c),
$$
and
$$
\tilde L_{c_2c_2}(u,s,c)=\rho(c_2) L_{c_2c_2}(u,s,c)+d\rho(c_2)\otimes F_{c_2}(u,s,c).
$$
As a consequence, we have
\begin{align*}
&\|\tilde L_{us}(x)\|+\|\tilde L_{uc}(x)\|+\|\tilde L_{ss}(x)\|+\|\tilde L_{sc}(x)\|+
\|\tilde L_{cu}(x)\|+\|\tilde L_{cs}(x)\|+\|\tilde L_{cc}(x)\|
\\
=&\rho(c_2) \big(\|L_{us}(x)\|+\|L_{uc}(x)\|+\|L_{ss}(x)\|+\|L_{sc}(x)\|
\|L_{cu}(x)\|+\|L_{cs}(x)\|+\|L_{cc}(x)\|\big)\\
+&\|F_{c_2}(x)\|\|d\rho(c_2)\|
\leq m.
\end{align*}
\qed

Under the hypotheses of Proposition \ref{realNHI},
the sets  $\tilde W^{sc}, \tilde W^{uc}, \tilde W^c$
associated to $\tilde F$ are graphs of $C^1$ functions
$$\tilde w^{sc}:B^s \times \Omega^c_r \lto B^u,
\quad \tilde w^{uc}:B^u \times \Omega^c_r \lto B^s,
\quad
\tilde w^c:\Omega^c_r\lto B^u\times B^s
$$
which satisfying the estimates
$$\|d\tilde w^{sc}\|\leq K,\quad \|d\tilde w^{uc}\|\leq K,
\quad \|d\tilde w^c\|\leq 2K.
$$
The restrictions to $\Omega$
$$
\mW^{sc}=\tilde W^{sc}\cap\Omega,\quad
\mW^{uc}=\tilde W^{uc}\cap \Omega,\quad
 \mW^c=\mW^{sc} \cap \mW^{uc}=\tilde W^c\cap \Omega,
$$
are weakly invariant by $F$ in the sense that this vector field is tangent
to them.
They satisfy various interesting properties.
For example, each $F$-invariant set contained in $\Omega$ is
contained in $\mW^c$.


\end{document}